\documentclass[a4paper, 11pt]{article}

\usepackage{amsmath, amssymb, amsthm} % AMS series
\usepackage{color}
\usepackage[top=40truemm,bottom=30truemm,left=30truemm,right=30truemm]{geometry}
\usepackage{hyperref}
\usepackage[all]{xy}

\newcommand{\V}{\mathcal{V}}

\newcommand{\ep}{\varepsilon}

\newcommand{\imply}{\mspace{10mu} \Longrightarrow \mspace{10mu}}

\DeclareMathOperator{\cls}{cl}
\DeclareMathOperator{\dom}{dom}
\DeclareMathOperator{\ev}{ev}
\DeclareMathOperator{\lip}{lip}
\DeclareMathOperator{\Map}{Map}
\DeclareMathOperator{\sgn}{sgn}
\DeclareMathOperator{\supp}{supp}

\theoremstyle{plain}
\newtheorem{theorem}{Theorem}[section]
\newtheorem{Theorem}{Theorem}
\newtheorem{lemma}[theorem]{Lemma}
\newtheorem{proposition}[theorem]{Proposition}
\newtheorem{corollary}[theorem]{Corollary}
\newtheorem{fact}{Fact}

\theoremstyle{definition}
\newtheorem{definition}[theorem]{Definition}
\newtheorem{notation}{Notation}
\newtheorem{example}{Example}[section]

\theoremstyle{remark}
\newtheorem{remark}[theorem]{Remark}

\begin{document}

\title{Theory of well-posedness for delay differential equations via prolongations and $C^1$-prolongations:
its application to state-dependent delay}
\author{Junya Nishiguchi
	\thanks{Research Alliance Center for Mathematical Sciences, Tohoku University,
	2-1-1 Katahira, Aoba-ku, Sendai, 980-8577, Japan}
	\footnote{E-mail: \url{junya.nishiguchi.b1@tohoku.ac.jp}}
	}

\maketitle

\begin{abstract}
In this paper, we establish a theory of well-posedness for delay differential equations (DDEs)
via notions of \textit{prolongations} and \textit{$C^1$-prolongations},
which are continuous and continuously differentiable extensions of histories to the right, respectively.
In this sense, this paper serves as a continuation and an extension of the previous paper
by this author (\cite{Nishiguchi 2017}).
The results in \cite{Nishiguchi 2017} are applicable to various DDEs,
however, the results in \cite{Nishiguchi 2017} cannot be applied to general class of state-dependent DDEs,
and its extendability is missing.
We find this missing link by introducing notions of 
($C^1$-) prolongabilities, regulation of topology by ($C^1$-) prolongations,
and Lipschitz conditions about ($C^1$-) prolongations, etc.
One of the main result claims that
the continuity of the semiflow with a parameter generated by the trivial DDEs $\dot{x} = v$
plays an important role for the well-posedness.
The results are applied to general class of state-dependent DDEs.

\begin{flushleft}
\end{flushleft}

\begin{flushleft}
\end{flushleft}
\end{abstract}

\tableofcontents

\section{Introduction}

A delay differential equation (DDE) is a differential equation
for which the derivative of the unknown function $x = x(\cdot)$ at $t \in \mathbb{R}$ also depends on the past information
	\begin{equation*}
		\{\mspace{2mu} x(s) : s < t \mspace{2mu}\}
			\mspace{10mu} \text{or, more precisely} \mspace{10mu}
		x(t + \cdot)|_{(-\infty, 0)}
	\end{equation*}
before $t$.
DDEs appear as various mathematical models for \textit{time-delay systems},
which are \textit{dynamic systems} having delayed time-lag in the causality
(refs.\ Erneux~\cite{Erneux 2009}, Smith~\cite{Smith 2011},
Lakshmanan \& Senthilkumar~\cite{Lakshmanan--Senthilkumar 2010},
and Walther~\cite{Walther 2014a}).

The purpose of this paper is to present a unified theory of well-posedness
which is applicable for various DDEs
including (i) general class of state-dependent DDEs and (ii) DDEs with infinite lag.
This was studied by this author (see \cite{Nishiguchi 2017}),
where a unified theory of well-posedness was established
based on the notion of \textit{prolongations} of histories.
Here a prolongation of some history is simply a continuous extension of that history to the right.
In this sense, this paper serves as a continuation and an extension of the previous one.
We note that the terminology of \textit{continuations} was used in \cite{Nishiguchi 2017} as that of prolongations.
In this paper it is decided that the terminology of continuations should be replaced with that of prolongations.
The theory in \cite{Nishiguchi 2017} covers some class of state-dependent DDEs and DDEs with infinite lag,
however, it is unclear whether this theory can be extended to cover general class of state-dependent DDEs.

\vspace{0.5\baselineskip}

We briefly review the dynamics viewpoint for DDEs to clarify the position of this paper.
A mathematical formulation of DDEs is the notion of retarded functional differential equations (RFDEs) of the form
	\begin{equation*}
		\dot{x}(t) = F(t, x_t) \mspace{20mu} (\text{$t \in \mathbb{R}$, $x(t) \in E$}),
	\end{equation*}
where $E = (E, \|\cdot\|_E)$ is some Banach space
(ref.\ Hale~\cite{Hale--Lunel 1993}).
In the right-hand side, the map $F$, called the \textit{history functional} in this paper,
is an $E$-valued functional of the \textit{history} $x_t$ of $x$ at $t$.
Therefore, the above equation stands for the past dependence of the derivative $\dot{x}(t)$.
For the formulation of a DDE as an RFDE,
it is necessary to choose a space of histories, called \textit{history spaces},
which constitutes the ambient space of the domain of definition $\dom(F)$ of $F$.

The dynamical systems point of view for DDEs was introduced
through this formulation (see Hale~\cite{Hale 1963b, Hale 2006b}).
For example, a given autonomous DDE with finite lag can be formulated as an autonomous RFDE
with the Banach space of continuous histories,
namely, the Banach space $C([-r, 0], E)_\mathrm{u}$ when the maximal time-lag is less than or equal to $r > 0$.
Here the symbol u represents the topology of uniform convergence.
In this case, the history functional $F$ becomes a map
	\begin{equation*}
		F \colon C([-r, 0], E)_\mathrm{u} \supset \dom(F) \to E,
	\end{equation*}
and the history $x_t \in C([-r, 0], E)_\mathrm{u}$ is defined by
	\begin{equation*}
		x_t(\theta) = x(t + \theta) \mspace{20mu} (\theta \in [-r, 0]).
	\end{equation*}
Then under appropriate assumptions,
the corresponding autonomous RFDE generates a continuous semiflow
$\varPhi_F \colon \mathbb{R}_+ \times \dom(F) \to \dom(F)$ $(\mathbb{R}_+ := [0, \infty))$ via the relation
	\begin{equation*}
		\varPhi_F(t, \phi) = x_F(\cdot; \phi)_t,
	\end{equation*}
where $x_F(\cdot; \phi) \colon [-r, +\infty) \to E$ is the unique solution of the RFDE
satisfying the initial condition $x_F(\cdot; \phi)_0 = \phi$.
For non-autonomous DDEs, the corresponding RFDEs are also non-autonomous,
and continuous processes should be used instead of continuous semiflows.
The terminology of processes was introduced by Dafermos~\cite{Dafermos 1971},
and periodic processes were also studied by Hale~\cite{Hale 1988}.
See Appendix~\ref{sec:maximal semiflows and processes} for the definition of processes.
See also the latter discussion of this introduction.

A class of infinite-dimensional dynamical systems is obtained in this way.
Here the well-posedness issue of the initial value problems (IVPs) for DDEs with initial history data
enters in order to obtain continuous semiflows or continuous processes,
i.e., (non-autonomous local) topological semi-dynamical systems, generated by DDEs.
This was treated in Hale~\cite{Hale 1977} and Hale \& Verduyn Lunel~\cite{Hale--Lunel 1993}, for example.
However, we have the following difficulties:
\begin{itemize}
\item DDEs with state-dependent time-lag or, called state-dependent DDEs, cannot be covered by those theories.
A class of state-dependent DDEs is given by
	\begin{equation*}
		\dot{x}(t) = f \Bigl( t, x(t), x \bigl( t - \tau(x(t)) \bigr) \Bigr)
			\mspace{20mu}
		(\text{$t \in \mathbb{R}$, $x(t) \in E$}),
	\end{equation*}
where $f \colon \mathbb{R} \times E \times E \to E$ and $\tau \colon E \to \mathbb{R}_+$ are continuous functions.
When $\tau(E) \subset [0, r]$ for some $r > 0$,
this equation is a DDE with finite lag.
A difficulty for the above DDE is the lack of smoothness
of the corresponding history functional $F \colon \mathbb{R} \times C([-r, 0], E)_\mathrm{u} \to E$ defined by
	\begin{equation*}
		F(t, \phi) = f \Bigl( t, \phi(0), \phi \bigl( -\tau(\phi(0)) \bigr) \Bigr)
	\end{equation*}
(see Mallet-Paret et al.~\cite{Mallet-Paret--Nussbaum--Paraskevopoulos 1994},
Louihi et al.~\cite{Louihi--Hbid--Arino 2002}, and Walther~\cite{Walther 2003c}, for example).
The smoothness of $f$ does not imply that of $F$ nor the Lipschitz condition of $F$ in general.
We refer the reader to Hartung, Krisztin, Walther, \& Wu~\cite{Hartung--Krisztin--Walther--Wu 2006}
for a general reference of state-dependent DDEs.
See Walther~\cite{Walther 2016} for DDEs with state-dependent unbounded time-lag.
\item We have another difficulty for DDEs when the time-lag is infinite or unbounded.
By the non-compactness of $\mathbb{R}_- := (-\infty, 0]$ which is the domain of definition of whole histories,
it is possible to choose various spaces of histories depending the equations
(ref.\ Hino, Murakami, \& Naito~\cite{Hino--Murakami--Naito 1991}).
\end{itemize}

\vspace{0.5\baselineskip}

We summarize the theory and results in \cite{Nishiguchi 2017}
to make clear the motivation, objectives, and results of this paper.
Let
	$I \subset \mathbb{R}_-$ be an interval,
	$H$ be a linear topological space which consists of maps from $I$ to $E$ with linear operations for functions, and
	$F \colon \mathbb{R} \times H \supset \dom(F) \to E$ be a map.
Then we consider an \textit{RFDE with history space $H$}
	\begin{equation}\label{eq:RFDE}
		\dot{x}(t) = F(t, I_tx) \mspace{20mu} (\text{$t \in \mathbb{R}$, $x(t) \in E$})
	\end{equation}
and its IVP
	\begin{equation}\label{eq:IVP}
		\left\{
		\begin{alignedat}{2}
			\dot{x}(t) &= F(t, I_tx), & \mspace{20mu} & t \ge t_0, \\
			I_{t_0}x &= \phi_0, & & (t_0, \phi_0) \in \dom(F).
		\end{alignedat}
		\right. \tag{$*$}
	\end{equation}
Here
\begin{itemize}
\item $H$ is taken by the first letter of history (but, $X$ is used in \cite{Nishiguchi 2017}), and
\item $I_tx \colon I \to E$ is the history of $x = x(\cdot)$ at $t$ defined by
	\begin{equation*}
		I_tx(\theta) = x(t + \theta).
	\end{equation*}
\end{itemize}
In \cite{Nishiguchi 2017}, we adopted the usual notation $x_t$ for the history,
however, we will adopt the notation $I_tx$ in this paper to clarify the domain of definition of histories.
See Section~\ref{sec:formulation and notions}, in particular Subsection~\ref{subsec:RFDEs with history spaces}
for the details.
In this setting, we discussed the well-posedness issue for IVP~\eqref{eq:IVP}.
As requisite properties for $H$,
we assume that $H$ is \textit{prolongable} and \textit{regulated by prolongations}.
We note that these properties were collectively called the \textit{continuability} in \cite{Nishiguchi 2017}.
See Subsections~\ref{subsec:prolongations} and \ref{subsec:prolongability and regulation}
for the definitions of these properties.
Then one of the main result of \cite{Nishiguchi 2017} is the following.

\begin{Theorem}[\cite{Nishiguchi 2017}]\label{thm:maximal WP}
Suppose that $H$ is prolongable and regulated by prolongations.
Then the following statements are equivalent:
\begin{enumerate}
\item[\emph{(a)}] IVP~\eqref{eq:IVP} is maximally well-posed for any history functional $F$
which is continuous, uniformly locally Lipschitzian about prolongations, and defined on some open set.
\item[\emph{(b)}] $\mathbb{R}_+ \times H \ni (t, \phi) \mapsto S_0(t)\phi \in H$ is continuous.
\end{enumerate}
\end{Theorem}

It should be noted that this theorem is valid for the case that $E$ is infinite dimensional.
In (b), $\mathbb{R}_+ \times H \ni (t, \phi) \mapsto S_0(t)\phi \in H$ is the semiflow
generated by the trivial RFDE $\dot{x} = 0$ with history space $H$.
Therefore, Theorem~\ref{thm:maximal WP} claims that
under the assumptions of $F$ given in (a),
in order to obtain the \textit{maximal well-posedness} of IVP~\eqref{eq:IVP},
we only have to check the well-posedness for the trivial RFDE $\dot{x} = 0$ with history space $H$.
Here the terminology of maximal well-posedness is not standard.
See Subsection~\ref{subsec:maximal well-posedness} for the definition.
The part (a) $\Rightarrow$ (b) is trivial.
The idea of the proof of (b) $\Rightarrow$ (a) is that
under the assumption of the unique existence of a maximal solution $x_F(\cdot; t_0, \phi_0)$ of RFDE~\eqref{eq:RFDE}
satisfying the initial condition $I_{t_0}[x_F(\cdot; t_0, \phi_0)] = \phi_0$,
we decompose the \textit{solution process} $\mathcal{P}_F$ generated by RFDE~\eqref{eq:RFDE} by
	\begin{align}\label{eq:decomposition}
		\mathcal{P}_F(\tau, t_0, \phi_0)
		&:= I_{t_0 + \tau}[x_F(\cdot; t_0, \phi_0)] \notag \\
		&= I_\tau[y(\cdot; t_0, \phi_0)] + S_0(\tau)\phi_0,
	\end{align}
where $y(\cdot; t_0, \phi_0)$ is obtained by some normalization of $x_F(\cdot; t_0, \phi_0)$
and is a prolongation of the trivial history $\boldsymbol{0}$.
See Subsection~\ref{subsec:maximal well-posedness} for the definition of $\mathcal{P}_F$.
This decomposition reminds us the perturbation theory
developed in \cite[Chapter II]{Diekmann--vanGils--Lunel--Walther 1995}.

In this way, the continuity of the semiflow $\mathbb{R}_+ \times H \ni (t, \phi) \mapsto S_0(t)\phi \in H$,
appearing the second term of \eqref{eq:decomposition},
enters for the proof of the maximal well-posedness of IVP~\eqref{eq:IVP}.
For the first term, the Lipschitz condition called \textit{uniform local Lipschitz about prolongations} is related,
which is a main key notion introduced in \cite{Nishiguchi 2017}.
We next explain the point of the notion of (uniform) local Lipschitz about prolongations briefly.
Basically, the inequality used for this Lipschitz condition is of the form
	\begin{equation}\label{eq:Lipschitzian}
		\|F(t, \phi_1) - F(t, \phi_2)\|_E
		\le L \cdot \|\phi_1 - \phi_2\|_\infty,
		\tag{Lip}
	\end{equation}
where $L > 0$ is some constant, and $\|\cdot\|_\infty$ is the infinity norm for maps.
When $H = C([-r, 0], E)_\mathrm{u}$ and \eqref{eq:Lipschitzian} holds
for some neighborhood of some base point $(t_0, \phi_0) \in \dom(F)$,
\eqref{eq:Lipschitzian} gives the usual Lipschitz condition for the history functional $F$.
However, in general, \eqref{eq:Lipschitzian} does not give the usual Lipschitz condition by the following reasons:
	\begin{itemize}
	\item It is not assumed that $H$ is a metric space.
	If $H$ has some metric structure, the metric is not necessarily given by $\|\cdot\|_\infty$.
	\item When $I = \mathbb{R}_-$, the upper bound given by $\|\phi_1 - \phi_2\|_\infty$ does not make sense.
	\end{itemize}
The condition which the property of local Lipschitz about prolongations requires is that
\eqref{eq:Lipschitzian} holds for all $(t, \phi_1), (t, \phi_2)$ obtained as the histories of prolongations, that is to say,
we require that \eqref{eq:Lipschitzian} holds for all $(t, \phi_1), (t, \phi_2)$ satisfying
	\begin{equation*}
		(t, \phi_1) = (t, I_t\gamma_1)
			\mspace{10mu} \text{and} \mspace{10mu}
		(t, \phi_2) = (t, I_t\gamma_2)
	\end{equation*}
for some prolongations $\gamma_1, \gamma_2$ of the base point $(t_0, \phi_0)$.
See Subsection~\ref{subsec:prolongations} for the definition of prolongations.
Then for such $(t, \phi_1)$ and $(t, \phi_2)$,
$\phi_1 - \phi_2 = I_t[\gamma_1 - \gamma_2]$ has the support in $[-t, 0]$,
and therefore, the right-hand side of \eqref{eq:Lipschitzian} is finite.
For the property of uniform local Lipschitz about prolongations,
it is required that the Lipschitz constant can be chosen uniformly in the base point $(t_0, \phi_0)$.
See Subsection~\ref{subsec:Lip conditions} or \cite[Definitions V and VII]{Nishiguchi 2017}
for the precise definition of the property of (uniform) local Lipschitz about prolongations.

As explained above, the uniform local Lipschitz about prolongations is an weak property
in the sense that we do not need to compare the difference of $F$ between histories having different tails.
Therefore, the important meaning in Theorem~\ref{thm:maximal WP} is that
under the well-posedness for the trivial RFDE $\dot{x} = 0$ with history space $H$,
the assumption that the history functional $F$ is uniformly locally Lipschitzian and the mild conditions are sufficient
to ensure the maximal well-posedness of IVP~\eqref{eq:IVP}.
Here Decomposition~\eqref{eq:decomposition} plays an important role.
This fact reveals a mechanism of the continuity of the solution process $\mathcal{P}_F$.
Furthermore, this is unexpected
because even for the continuity of $\mathcal{P}_F$, the weak Lipschitz condition is sufficient.
We refer the reader to Kappel \& Schappacher~\cite[Remark 1.6]{Kappel--Schappacher 1978}
about the discussion for a scalar DDE
	\begin{equation*}
		\dot{x}(t) = \sgn(x(t - 1)),
	\end{equation*}
where $\sgn(\cdot) \colon \mathbb{R} \to \{-1, 0, +1\}$ is the signum function.

\vspace{0.5\baselineskip}

An interesting feature related to the Lipschitz condition about prolongations can be seen
in DDEs with a single constant lag of the form
	\begin{equation}\label{eq:DDE with a single constant lag}
		\dot{x}(t) = f(x(t - r))
			\mspace{20mu}
		(\text{$t \in \mathbb{R}$, $x(t) \in E$}),
	\end{equation}
where $f \colon E \to E$ is a continuous function, and $r > 0$ is a constant lag.
This is an autonomous DDE with finite lag and can be written as an autonomous RFDE
	\begin{equation*}
		\dot{x}(t) = F([-r, 0]_tx)
	\end{equation*}
with the history functional $F \colon C([-r, 0], E)_\mathrm{u} \to E$ defined by
	\begin{equation*}
		F(\phi) = f(\phi(-r)).
	\end{equation*}
Then one can see that the above $F$ is \textit{constant about memories} for any map $f$ because
	\begin{equation*}
		F(\phi_1) - F(\phi_2)
		= f(\phi_1(-r)) - f(\phi_2(-r))
		= 0
	\end{equation*}
holds for all $\phi_1, \phi_2 \in C([-r, 0], E)$ satisfying
	\begin{equation*}
		\supp(\phi_1 - \phi_2) \subset [-R, 0]
	\end{equation*}
for some $0 < R < r$.
Here the constancy about memories implies the Lipschitz about memories,
which also implies the uniform local Lipschitz about prolongations.
This property of $F$ should be compared to the property that
the Lipschitz condition of $f$ is necessary to ensure the usual Lipschitz condition for this $F$.
This property also corresponds to the result about unique global existence for DDE~\eqref{eq:DDE with a single constant lag}
via \textit{step method},
by which the unique existence of a global solution of DDE~\eqref{eq:DDE with a single constant lag}
under a specified initial condition follows.
We refer the reader to Smith~\cite[Chapter 3]{Smith 2011} as a reference of step method and related results.
We come to the conclusion that Theorem~\ref{thm:maximal WP} contains step method in some sense.

Theorem~\ref{thm:maximal WP} can cover some classes of DDEs with infinite or unbounded lag
by the following reasons:
\begin{itemize}
\item The axiom established by Hale \& Kato~\cite{Hale--J.Kato 1978} and Kato~\cite{Kato 1978}
imply the prolongability and the regulation by prolongations of history spaces.
\item The Lipschitz condition with respect to seminorm implies the uniform local Lipschitz about prolongations.
\end{itemize}
Therefore, the theory of well-posedness in \cite{Nishiguchi 2017} contains the results under the \textit{Hale--Kato axiom}.
Indeed, the history space $H = C(\mathbb{R}_-, E)_\mathrm{co}$,
where the symbol co represents the compact-open topology,
cannot be covered by the Hale--Kato axiom.
However, this history space is within the scope of Theorem~\ref{thm:maximal WP}.
See \cite[Subsections 2.3 and 6.1]{Nishiguchi 2017} for the relationship
between (i) the Hale--Kato axiom and (ii) the prolongability and the regulation by prolongations.

Theorem~\ref{thm:maximal WP} can also cover a class of state-dependent DDEs of the form
	\begin{equation}\label{eq:sdDDE}
		\dot{x}(t) = f \bigl( t, x(t), x(t - \tau(t, I_tx)) \bigr),
	\end{equation}
where $f \colon \mathbb{R} \times E \times E \to E$ is a continuous function, and
$\tau$ is a continuous delay functional which \textit{ignores} the values of histories on some interval $(-R, 0]$.
This notion was introduced by Rezounenko~\cite{Rezounenko 2009}
for the cases that (i) $\tau$ is independent from $t$, and (ii) $I = [-r, 0]$ for some $r > 0$.
In \cite{Nishiguchi 2017}, it was shown that the history functional $F \colon \mathbb{R} \times C(I, E)_\mathrm{co} \to E$ defined by
	\begin{equation*}
		F(t, \phi) = f \bigl( t, \phi(0), \phi(-\tau(t, \phi)) \bigr)
	\end{equation*}
is uniformly locally Lipschitzian about prolongations
under the assumption that $\tau$ is \textit{constant about memories}.
See Definition~\ref{dfn:constancy about memories} for the definition of the constancy about memories.
However, Theorem~\ref{thm:maximal WP} cannot be applied to state-dependent DDEs
with general delay functionals by the lack of smoothness for the history functional.
Furthermore, relationships between the (uniform) local Lipschitz about prolongations
and the almost local Lipschitz which is the Lipschitz condition
introduced by Mallet-Paret, Nussbaum, Paraskevopoulos~\cite{Mallet-Paret--Nussbaum--Paraskevopoulos 1994}
for state-dependent DDEs are missing.

\vspace{0.5\baselineskip}

In this paper, to overcome the above mentioned difficulty for state-dependent DDEs
and to extend the theory developed in \cite{Nishiguchi 2017},
we introduce
(i) ($C^1$-) prolongations and ($C^1$-) prolongation spaces,
(ii) ($C^1$-) prolongabilities and regulation of topology by ($C^1$-) prolongations of history spaces,
(iii) rectangles by ($C^1$-) prolongations, and neighborhoods by ($C^1$-) prolongations, and
(iv) (uniform) local Lipschitz about $C^1$-prolongations for history functionals, etc.
We note that prolongations are only considered in \cite{Nishiguchi 2017}.
After that, we investigate properties of these notions, e.g.,
a necessary condition of ($C^1$-) prolongabilities for history spaces,
a characterization of regulation of topology by ($C^1$-) prolongations, and
relationships between Lipschitz conditions.
Based on these notions,
we prove the existence and uniqueness result,
study a mechanism of the continuity of solution processes,
and finally apply the obtained results to the maximal well-posedness of state-dependent DDEs.

We summarize the results which will be obtained in this paper as follows.

\vspace{0.2\baselineskip}

\textbf{$\bullet$ ($C^1$-) prolongability:}
The prolongability or $C^1$-prolongability of a history space $H$ implies
	\begin{equation*}
		C_\mathrm{c}(I, E) \subset H
			\mspace{10mu} \text{or} \mspace{10mu}
		C^1_\mathrm{c}(I, E) \subset H,
	\end{equation*}
respectively
(see Propositions~\ref{prop:closedness under prolongations} and \ref{prop:closedness under C^1-prolongations}).
These inclusions reveal a connection between ($C^1$-) prolongability and regularity of histories.

\vspace{0.2\baselineskip}

\textbf{$\bullet$ Regulation by ($C^1$-) prolongations:}
When $I = [-r, 0]$ for some $r > 0$,
the regulation properties of topology of $H$ by prolongations or $C^1$-prolongations are characterized by
the continuity of inclusions
	\begin{equation*}
		C(I, E)_\mathrm{u} \subset H
			\mspace{10mu} \text{or} \mspace{10mu}
		\left( C^1(I, E), \|\cdot\|_{C^1} \right) \subset H,
	\end{equation*}
respectively (see Theorems~\ref{thm:r < infty, regulation} and \ref{thm:r < infty, C^1-regulation}).
Corresponding results are also obtained when $I = \mathbb{R}_-$
(see Theorems~\ref{thm:r = infty, regulation} and \ref{thm:r = infty, C^1-regulation}).
These results shows that the regulation by prolongations or $C^1$-prolongations are characterized topologically.

\vspace{0.2\baselineskip}

\textbf{$\bullet$ Neighborhoods by ($C^1$-) prolongations:}
Theorem~\ref{thm:nbd by C^1-prolongations} shows that
when $H$ is $C^1$-prolongable and regulated by $C^1$-prolongations,
a neighborhood $W$ of $(\sigma, \psi) \in \mathbb{R} \times H$ in $[\sigma, +\infty) \times H$ is
a neighborhood by $C^1$-prolongations.
Furthermore, if
	\begin{equation*}
		\mathbb{R}_+ \times E \times H \ni (t, v, \phi) \mapsto S_v(t)\phi \in H
	\end{equation*}
is continuous, then $W$ is a uniform neighborhood by $C^1$-prolongations.
Here the above map is the solution semiflow with a parameter $v \in E$
generated by RFDEs $\dot{x} = v$ with history space $H$.

\vspace{0.2\baselineskip}

\textbf{$\bullet$ Relationships between Lipschitz conditions about ($C^1$-) prolongations:}
Propositions~\ref{prop:comparison of Lip about prolongations} and \ref{prop:comparison of uniform Lip about prolongations}
state the following relation:
	\begin{align*}
		&\text{(uniform) local Lipschitz about prolongations} \\
		&\mspace{40mu} \Rightarrow \text{(uniform) local Lipschitz about $C^1$-prolongations}.
	\end{align*}
We also show that the almost local Lipschitz for a history functional $F$ implies that
the restriction of $F$ to the space of local Lipschitz continuous maps $C^{0, 1}_\mathrm{loc}(I, E)$ is
locally Lipschitzian about $C^1$-prolongations.
(see Theorems~\ref{thm:almost local Lip for I = I^r} and \ref{thm:almost local Lip for I = I^infty}).

\vspace{0.2\baselineskip}

\textbf{$\bullet$ Existence and uniqueness:}
The main result about the existence and uniqueness is following.

\begin{Theorem}\label{thm:existence and uniqueness}
Let $(t_0, \phi_0) \in \dom(F)$.
Suppose that
(i) $H$ is $C^1$-prolongable and regulated by $C^1$-prolongations,
(ii) $F$ is continuous, and
(iii) $\dom(F)$ is a neighborhood by $C^1$-prolongations of $(t_0, \phi_0)$.
If $F$ is locally Lipschitzian about $C^1$-prolongations at $(t_0, \phi_0)$,
then there exist $T > 0$ such that
$\eqref{eq:IVP}_{t_0, \phi_0}$ has a unique $C^1$-solution $x \colon [t_0, t_0 + T] + I \to E$.
\end{Theorem}

This theorem contains and extends the result about existence and uniqueness in \cite{Nishiguchi 2017}
because of the following implications:
	\begin{gather*}
		\text{prolongability $\Rightarrow$ $C^1$-prolongability}, \\
		\text{Regulation by prolongations $\Rightarrow$ Regulation by $C^1$-prolongations}, \\
		\text{Neighborhood by prolongations $\Rightarrow$ Neighborhood by $C^1$-prolongations}.
	\end{gather*}
Theorem~\ref{thm:existence and uniqueness} is obtained by a combination of
Corollaries~\ref{cor:local existence with Lip, (E1)} and \ref{cor:local uniqueness (E1)}.

\vspace{0.2\baselineskip}

\textbf{$\bullet$ Continuity of solution processes:}
One of the main result about the continuity of solution process $\mathcal{P}_F$ is following.

\begin{Theorem}\label{thm:maximal WP, C^1}
Suppose that
(i) $H$ is $C^1$-prolongable and regulated by $C^1$-prolongations,
(ii) $F$ is continuous, and
(iii) $\dom(F)$ is a uniform neighborhood by $C^1$-prolongations of each $(t_0, \phi_0) \in \dom(F)$.
If $F$ is uniformly locally Lipschitzian about $C^1$-prolongations, and if
	\begin{equation*}
		\mathbb{R}_+ \times E \times H \ni (t, v, \phi) \mapsto S_v(t)\phi \in H
	\end{equation*}
is continuous,
then IVP~\eqref{eq:IVP} is maximally well-posed for $C^1$-solutions.
\end{Theorem}

Theorem~\ref{thm:maximal WP, C^1} can be considered as
a generalization of Theorem~\ref{thm:maximal WP},
however, the continuity of the semiflow
	$\mathbb{R}_+ \times H \ni (t, \phi) \mapsto S_0(t)\phi \in H$
in Theorem~\ref{thm:maximal WP} is replaced with
that of the parametrized semiflow
	$\mathbb{R}_+ \times E \times H \ni (t, v, \phi) \mapsto S_v(t)\phi \in H$
in Theorem~\ref{thm:maximal WP, C^1}.

Another result about the continuity of solution process $\mathcal{P}_F$ is obtained when $E = \mathbb{R}^n$.

\begin{Theorem}\label{thm:maximal WP without uniform Lip}
Let $E = \mathbb{R}^n$.
Suppose that
(i) $H$ is prolongable and regulated by prolongations,
(ii) $F$ is continuous, and
(iii) $\dom(F)$ is a uniform neighborhood by prolongations of $(t_0, \phi_0)$.
If $F$ is locally Lipschitzian about prolongations, and if the semiflow
	\begin{equation*}
		\mathbb{R}_+ \times H \ni (t, \phi) \mapsto S_0(t)\phi \in H
	\end{equation*}
is continuous, then IVP~\eqref{eq:IVP} is maximally well-posed.
\end{Theorem}

Theorem~\ref{thm:maximal WP without uniform Lip} is an extension of
the part (b) $\Rightarrow$ (a) in Theorem~\ref{thm:maximal WP}.
The differences from Theorem~\ref{thm:maximal WP} is that
$E$ is finite-dimensional, and
the local Lipschitz about prolongations for $F$ is sufficient for the maximal well-posedness of IVP~\eqref{eq:IVP}.
See also Theorem~\ref{thm:maximal WP without uniform Lip, extension},
which is a result containing Theorem~\ref{thm:maximal WP without uniform Lip}.

\vspace{0.2\baselineskip}

\textbf{$\bullet$ Maximal well-posedness for state-dependent DDEs:}
Applying Theorem~\ref{thm:maximal WP without uniform Lip, extension} to
a general class of state-dependent DDEs of the form~\eqref{eq:sdDDE},
we obtain the maximal well-posedness for such equations
(see Corollary~\ref{cor:maximal WP, state-dependent}).

\vspace{0.5\baselineskip}

This paper is organized as follows.
In Section~\ref{sec:formulation and notions},
we give a mathematical formulation of DDEs.
We also give notions about history spaces obtained by using ($C^1$-) prolongations, and
Lipschitz conditions about ($C^1$-) prolongations.
In Sections~\ref{sec:prolongations and history spaces} and \ref{sec:Lip conditions},
we investigate properties of the notions about history spaces and relationships between the Lipschitz conditions.
Section~\ref{sec:existence and uniqueness} is about the existence and uniqueness theorem,
and in Section~\ref{sec:mechanisms}, we find a mechanism for the maximal well-posedness of IVP~\eqref{eq:IVP}.
In Section~\ref{sec:state-dependent DDEs},
we finally apply the result about maximal well-posedness to state-dependent DDEs.
We have five appendixes about
Appendix~\ref{sec:maximal semiflows and processes}:
	definitions and continuity properties of maximal semiflows and maximal processes,
Appendix~\ref{ap:maximal sol}:
	fundamental properties of maximal solutions,
Appendix~\ref{sec:family of maps}:
	relationships between continuity and equi-continuity for families of maps,
Appendix~\ref{sec:Lip about prolongations for ODEs}:
	an equivalence of the usual Lipschitz condition and the Lipschitz about prolongations for ODEs,
and Appendix~\ref{sec:proofs}:
	proofs of some propositions and theorem in Sections~\ref{sec:existence and uniqueness} and \ref{sec:mechanisms}.

\subsubsection*{Notation for function spaces}

Let $J \subset \mathbb{R}$ be an interval and $X = (X, \|\cdot\|_X)$ be a Banach space.
Let $\Map(J, X)$ be the set of all maps from $J$ to $X$, which is a linear space with the linear operations for maps.

\vspace{0.2\baselineskip}

\textbf{$\bullet$ Spaces of continuous maps:}
The set of all continuous maps from $J$ to $X$ is denoted by $C(J, X)$,
which is a linear subspace of $\Map(J, X)$.
For $f \in C(J, X)$, let
	\begin{equation*}
		\|f\|_C
		:= \|f\|_\infty
		= \sup_{t \in J} \|f(t)\|_X.
	\end{equation*}
For a sub-interval $J_0 \subset J$, we write $\|f\|_{C(J_0)} := \|f|_{J_0}\|_C$.
Let $C_\mathrm{c}(J, X)$ be the set of all continuous maps from $J$ to $X$ with compact support.
Here the support of a map $f \colon J \to X$ is given by
	\begin{equation*}
		\supp(f) := \cls \{\mspace{2mu} t \in J : f(t) \ne 0 \mspace{2mu}\}
		\mspace{20mu} (\text{$\cls$ is the closure operator}).
	\end{equation*}
$C_\mathrm{c}(J, X)$ is a linear subspace of $C(J, X)$.
By definition, $C_\mathrm{c}(J, E) = C(J, E)$ when $J$ is compact.

\vspace{0.2\baselineskip}

\textbf{$\bullet$ Spaces of continuously differentiable maps:}
The set of all $C^1$-maps (i.e., all continuously differentiable maps) from $J$ to $X$ is denoted by $C^1(J, X)$,
which is a linear subspace of $C(J, X)$.
For $f \in C^1(J, X)$, let
	\begin{equation*}
		\|f\|_{C^1} := \|f\|_C + \|f'\|_C.
	\end{equation*}
For a sub-interval $J_0 \subset J$, we write $\|f\|_{C^1(J_0)} := \|f|_{J_0}\|_{C^1}$.
Let $C^1_\mathrm{c}(J, X)$ be the set of all $C^1$-maps from $J$ to $X$ with compact support.
By definition, $C^1_\mathrm{c}(J, X) = C_\mathrm{c}(J, X) \cap C^1(J, X)$.
Therefore, $C^1_\mathrm{c}(J, X) = C^1(J, X)$ holds when $J$ is compact.

\vspace{0.2\baselineskip}

\textbf{$\bullet$ Spaces of Lipschitz continuous maps:}
For a constant $L > 0$, $f \in \Map(J, X)$ is said to be \textit{$L$-Lipschitz continuous} if
	\begin{equation*}
		\|f(t_1) - f(t_2)\|_X \le L \cdot |t_1 - t_2| \mspace{20mu} (\forall t_1, t_2 \in J)
	\end{equation*}
holds.
For $f \in \Map(J, X)$, let
	\begin{equation*}
		\lip(f) := \inf \{\mspace{2mu} L > 0 : \text{$f$ is $L$-Lipschitz continuous} \mspace{2mu}\}.
	\end{equation*}
$f$ is said to be \textit{Lipschitz continuous} if $\lip(f) < \infty$.
Let $C^{0, 1}(J, X)$ denote the set of all Lipschitz continuous maps from $J$ to $X$,
which is a linear subspace of $C(J, X)$.
For $f \in C(J, X)$, let
	\begin{equation*}
		\|f\|_{C^{0, 1}} := \|f\|_C + \lip(f).
	\end{equation*}
Let $C^{0, 1}_\mathrm{loc}(J, X)$ denote the set of all locally Lipschitz continuous maps from $J$ to $X$,
which is a linear subspace of $C(J, X)$.
$C^1(J, X)$ is a linear subspace of $C^{0, 1}_\mathrm{loc}(J, X)$.
When $J$ is compact, $C^{0, 1}_\mathrm{loc}(J, X) = C^{0, 1}(J, X)$, and for all $f \in C^1(J, X)$,
	\begin{equation*}
		\lip(f) = \|f'\|_C.
	\end{equation*}

\paragraph{Terminologies for convergence}
Let $J \subset \mathbb{R}$ be an compact interval and $X = (X, \|\cdot\|_X)$ be a Banach space.
$\|\cdot\|_{C(J)}$-norm topology on $C(J, X)$ and $\|\cdot\|_{C^1(J)}$-norm topology on $C^1(J, X)$ are expressed by
the \textit{topology of uniform convergence} on $C(J, X)$ and
the \textit{topology of uniform $C^1$-convergence} on $C^1(J, X)$,
respectively.
Let $C(J, X)_\mathrm{u}$ be the Banach space $\bigl( C(J, X), \|\cdot\|_{C(J)} \bigr)$.

\paragraph{Notation}

For a Banach space $X = (X, \|\cdot\|_X)$,
the open (resp.\ closed) ball with the center $x \in X$ and the radius $r > 0$ is denoted by
$B_X(x; r)$ (resp.\ $\bar{B}_X(x; r)$):
	\begin{align*}
		B_X(x; r) := \{\mspace{2mu} y \in X : \|y - x\|_X < r \mspace{2mu}\}, \mspace{15mu}
		\bar{B}_X(x; r) := \{\mspace{2mu} y \in X : \|y - x\|_X \le r \mspace{2mu}\}.
	\end{align*}

\section{Formulation and notions}\label{sec:formulation and notions}

The purpose of this section is to introduce
\begin{itemize}
\item the formulation of delay differential equations (DDEs)
as retarded functional differential equations (RFDEs) with history spaces, and
\item various notions related to RFDEs with history spaces which will be used in this paper.
\end{itemize}

Let
	\begin{equation*}
		0 \in I \subset \mathbb{R}_- := (-\infty, 0]
	\end{equation*}
be an interval and $E = (E, \|\cdot\|_E)$ be a Banach space.
We interpret $I$ as the domain of definition of histories and call it a \textit{past interval}.
In this sense, $\mathbb{R}_-$ and $\Map(\mathbb{R}_-, E)$ stand for the whole past interval and the whole space of histories, respectively.

For an interval $J \subset \mathbb{R}$, we write
	\begin{equation*}
		J + I := \{\mspace{2mu} t + \theta : \text{$t \in J$, $\theta \in I$} \mspace{2mu}\}.
	\end{equation*}
Let $\gamma \colon \mathbb{R} \supset \dom(\gamma) \to E$ be a map where $\dom(\gamma)$ contains $J + I$.
For each $t \in J$, $I_t\gamma \in \Map(I, E)$ is defined by
	\begin{equation*}
		I_t\gamma(\theta) = \gamma(t + \theta).
	\end{equation*}
We call $I_t\gamma$ the \textit{history} of $\gamma$ at $t \in J$ with the past interval $I$.

We choose a subset $H \subset \Map(I, E)$ with properties that
\begin{enumerate}
\item[(i)] $H$ is a linear subspace of $\Map(I, E)$, and
\item[(ii)] the topology of $H$ is given so that the linear operations on $H$ are continuous with respect to that topology.
\end{enumerate}
Then $H$ becomes a linear topological space.
$H$ can be interpreted as a space of histories with a past interval $I$, and we call it a \textit{history space}.
Let $\boldsymbol{0} \colon I \to E$ be the map whose value is identically equal to $0 \in E$.
Then $\boldsymbol{0}$ belongs to $H$ which is the zero element of $H$.

\subsection{Retarded functional differential equations with history spaces}\label{subsec:RFDEs with history spaces}

We formulate RFDEs with history space $H$ as follows.

\begin{definition}
Let $H \subset \Map(I, E)$ be a history space and $F \colon \mathbb{R} \times H \supset \dom(F) \to E$ be a map.
We call a differential equation~\eqref{eq:RFDE}
	\begin{equation*}
		\dot{x}(t) = F(t, I_tx) \mspace{20mu} (\text{$t \in \mathbb{R}$, $x(t) \in E$})
	\end{equation*}
a \textit{retarded functional differential equation} (RFDE) with history space $H$.
We call $F$ the \textit{history functional} of \eqref{eq:RFDE}.
For a non-degenerate interval $J \subset \mathbb{R}$,
a map $x \colon J + I \to E$ is called a \textit{solution} of \eqref{eq:RFDE} if the following hold:
\begin{enumerate}
\item[(i)] $(t, I_tx) \in \dom(F)$ for all $t \in J$.
\item[(ii)] $x|_J \colon J \to E$ is differentiable, and
	\begin{equation*}
		(x|_J)'(t) = F(t, I_tx) \mspace{20mu} (\forall t \in J).
	\end{equation*}
\end{enumerate}
\end{definition}

We have a variant of notions of solutions.
When $x|_J$ is of class $C^1$ in the condition (ii), $x$ is called a \textit{$C^1$-solution} of \eqref{eq:RFDE}.

For each $0 \le r \le \infty$, let
	\begin{equation*}
		I^r :=
		\begin{cases}
			[-r, 0], & r < \infty, \\
			\mathbb{R}_-, & r = \infty
		\end{cases}
	\end{equation*}
throughout this paper.
Then $I^0 = \{0\}$ is the degenerate interval.

The following are comments about RFDEs with history space $H$.

\begin{itemize}
\item We note that the case $I = I^0$ corresponds to ordinary differential equations (ODEs)
under the identification
	\begin{equation*}
		\Map(I^0, E) = C(I^0, E) = E.
	\end{equation*}
For this case, the history space is chosen as $H = C(I^0, E)_\mathrm{u}$.

\item When $I = I^r$ for $0 < r < \infty$, \eqref{eq:RFDE} is a functional differential equations (FDEs) with \textit{finite retardation}.
Usually, $H$ is chosen as
	\begin{equation*}
		H = C([-r, 0], E)_\mathrm{u}
	\end{equation*}
which is the Banach space of continuous maps from $[-r, 0]$ to $E$ with the topology of uniform convergence.
We refer the reader to
Hale~\cite{Hale 1977},
Hale \& Verduyn Lunel~\cite{Hale--Lunel 1993},
and Diekmann, van Gils, Verduyn Lunel, \& Walther~\cite{Diekmann--vanGils--Lunel--Walther 1995}
as general references for the theory of RFDEs with history space $H = C([-r, 0], \mathbb{R}^n)_\mathrm{u}$.

\item When $I = I^\infty$, \eqref{eq:RFDE} is an FDE with \textit{infinite retardation}.
By the character of the non-compactness of $I^\infty$,
there are various choices of history spaces depending on differential equations.
We refer the reader to Hino, Murakami, \& Naito~\cite{Hino--Murakami--Naito 1991}
as a general reference of the theory of FDEs with infinite retardation.
\end{itemize}

The initial value problem of \eqref{eq:RFDE} is formulated as follows.

\begin{definition}
Let $H \subset \Map(I, E)$ be a history space and $F \colon \mathbb{R} \times H \supset \dom(F) \to E$ be a map.
We consider the family of systems of equations~\eqref{eq:IVP}
	\begin{equation*}
		\left\{
		\begin{alignedat}{2}
			\dot{x}(t) &= F(t, I_tx), & \mspace{20mu} & t \ge t_0, \\
			I_{t_0}x &= \phi_0, & & (t_0, \phi_0) \in \dom(F)
		\end{alignedat}
		\right.
	\end{equation*}
with a parameter $(t_0, \phi_0)$.
This is called the \textit{initial value problem} (IVP) of \eqref{eq:RFDE}.
For a specific $(t_0, \phi_0)$, the corresponding system will be denoted by $\eqref{eq:IVP}_{t_0, \phi_0}$ in this paper.
A solution $x \colon J + I \to E$ of \eqref{eq:RFDE} is called a solution of $\eqref{eq:IVP}_{t_0, \phi_0}$ if
\begin{enumerate}
\item[(i)] $J$ is a left-closed interval with the left end point $t_0$, and
\item[(ii)] $I_{t_0}x = \phi_0$.
\end{enumerate}
We call a solution $x \colon J + I \to E$ of $\eqref{eq:IVP}_{t_0, \phi_0}$ a \textit{$C^1$-solution}
if $x|_J$ is of class $C^1$.
\end{definition}

For an interval $J \subset \mathbb{R}$, let
	\begin{equation*}
		|J| \in [0, \infty]
	\end{equation*}
denote its length.
Then for a solution $x \colon J + I \to E$ of $\eqref{eq:IVP}_{t_0, \phi_0}$,
$J$ is expressed as follows:
\begin{itemize}
\item Case 1: $|J| < \infty$. Then $J = [t_0, t_0 + |J|]$ or $J = [t_0, t_0 + |J|)$.
\item Case 2: $|J| = \infty$. Then $J = [t_0, t_0 + |J|) = [t_0, +\infty)$.
\end{itemize}

\subsection{Maximal well-posedness}\label{subsec:maximal well-posedness}

In this paper, we will use the following terminologies for IVP~\eqref{eq:IVP}:
\begin{itemize}
\item We say that \eqref{eq:IVP} satisfies the \textit{local existence} for $C^1$-solutions
if for each $(t_0, \phi_0) \in \dom(F)$, $\eqref{eq:IVP}_{t_0, \phi_0}$ has a $C^1$-solution.
\item We say that \eqref{eq:IVP} satisfies the \textit{local uniqueness} for $C^1$-solutions
if the following statement holds for every $(t_0, \phi_0) \in \dom(F)$:
For any $C^1$-solutions $x_i \colon J_i + I \to E$ $(i = 1, 2)$ of $\eqref{eq:IVP}_{t_0, \phi_0}$,
there exists $T > 0$ such that $x_1|_{[t_0, t_0 + T]} = x_2|_{[t_0, t_0 + T]}$.
\end{itemize}

When $\eqref{eq:IVP}_{t_0, \phi_0}$ has the unique maximal $C^1$-solution
	\begin{equation*}
		x_F(\cdot; t_0, \phi_0) \colon [t_0, t_0 + T_F(t_0, \phi_0)) + I \to E
		\mspace{20mu}
		(0 < T_F(t_0, \phi_0) \le \infty)
	\end{equation*}
for every $(t_0, \phi_0) \in \dom(F)$,
we define the map $\mathcal{P}_F \colon \mathbb{R}_+ \times \dom(F) \supset \dom(\mathcal{P}_F) \to H$ by
	\begin{equation}\label{eq:sol process}
		\begin{aligned}
			\dom(\mathcal{P}_F)
			&= \bigcup_{(t_0, \phi_0) \in \dom(F)} [0, T_F(t_0, \phi_0)) \times \{(t_0, \phi_0)\}, \\
			\mathcal{P}_F(\tau, t_0, \phi_0)
			&= I_{t_0 + \tau} [x_F(\cdot; t_0, \phi_0)].
		\end{aligned}
	\end{equation}
We call $\mathcal{P}_F$ the \textit{solution process} generated by RFDE~\eqref{eq:RFDE}.
See Appendix~\ref{ap:maximal sol} for the notion of maximal $C^1$-solutions.

\begin{definition}
We say that IVP~\eqref{eq:IVP} is \textit{maximally well-posed} for $C^1$-solutions
if both of the following conditions are satisfied:
\begin{enumerate}
\item[(i)] For every $(t_0, \phi_0) \in \dom(F)$, $\eqref{eq:IVP}_{t_0, \phi_0}$ has the unique maximal $C^1$-solution.
\item[(ii)] The solution process $\mathcal{P}_F$ is a continuous maximal process in $\dom(F)$.
\end{enumerate}
\end{definition}

See Appendix~\ref{sec:maximal semiflows and processes} for the notion of maximal processes and those continuity.

\subsection{Translations on history spaces}

\subsubsection{Family of transformations determined by trivial ODEs}

In the theory which will be developed in this paper, the family of IVPs
	\begin{equation}\label{eq:IVP, trivial RFDE}
		\left\{
		\begin{alignedat}{2}
			\dot{x}(t) &= v, & \mspace{20mu} & t \ge 0, \\
			I_0x &= \phi_0, & & \phi_0 \in H
		\end{alignedat}
		\right.
	\end{equation}
with a parameter $v \in E$ plays an important role.
The unique global solution for a specific $(v, \phi_0) \in E \times H$ is equal to $\phi_0^{\wedge v}$:
for any $(\psi, v) \in \Map(I, E) \times E$,
	\begin{equation*}
		\psi^{\wedge v} \colon \mathbb{R}_+ + I \to E
			\mspace{20mu}
		(\mathbb{R}_+ := [0, +\infty))
	\end{equation*}
is defined by
	\begin{equation*}
		\psi^{\wedge v}(t) =
		\begin{cases}
			\psi(t), & t \in I, \\
			\psi(0) + tv, & t \in \mathbb{R}_+.
		\end{cases}
	\end{equation*}
We write $\bar{\psi} := \psi^{\wedge 0}$, which is a prolongation of $\psi$ by the constant $\psi(0)$.

For each $t \in \mathbb{R}_+$, let
	\begin{equation*}
		S(t) \colon E \times \Map(I, E) \to \Map(I, E)
	\end{equation*}
be the transformation defined by
	\begin{equation*}
		S(t)(v, \phi)
		:= S_v(t)\phi
		:= I_t[\phi^{\wedge v}].
	\end{equation*}
We note that
	\begin{equation*}
		S_v(t) \colon \Map(I, E) \to \Map(I, E)
	\end{equation*}
is not linear when $v \ne 0$.
However, $S(t) \colon E \times \Map(I, E) \to \Map(I, E)$ is linear because
	\begin{align*}
		(\phi + \psi)^{\wedge (v + w)}
		&=
		\begin{cases}
			\phi(t) + \psi(t), & t \in I, \\
			\phi(0) + \psi(0) + t(v + w), & t \in \mathbb{R}_+
		\end{cases} \\
		&= \phi^{\wedge v} + \psi^{\wedge w}.
	\end{align*}

\subsubsection{Translations}
Let $(\sigma, \psi) \in \mathbb{R} \times \Map(I, E)$ and $v \in E$.
We define the map
	\begin{equation*}
		\tau_{\sigma, \psi}^v \colon \mathbb{R}_+ \times \Map(I, E) \to [\sigma, +\infty) \times \Map(I, E)
	\end{equation*}
by
	\begin{equation*}
		\tau_{\sigma, \psi}^v(t, \phi) = \bigl( \sigma + t, I_t[\psi^{\wedge v}] + \phi \bigr).
	\end{equation*}
The inverse $\bigl( \tau_{\sigma, \psi}^v \bigr)^{-1} \colon [\sigma, +\infty) \times \Map(I, E) \to \mathbb{R}_+ \times \Map(I, E)$
is given by
	\begin{equation*}
		\bigl( \tau_{\sigma, \psi}^v \bigr)^{-1} \bigl( \tilde{t}, \tilde{\phi} \bigr)
		= \bigl( \tilde{t} - \sigma, \tilde{\phi} - I_{\tilde{t} - \sigma}[\psi^{\wedge v}] \bigr),
	\end{equation*}
which is obtained by solving $\tau_{\sigma, \psi}^v(t, \phi) = \bigl( \tilde{t}, \tilde{\phi} \bigr)$.

The map $\tau_{\sigma, \psi}^v$ appears as a translation.
In fact, if $x \colon [t_0, t_0 + T] + I \to E$ is a solution of $\eqref{eq:IVP}_{t_0, \phi_0}$
for a given $(t_0, \phi_0) \in \dom(F)$,
the map $y \colon [0, T] + I \to E$ defined by
	\begin{equation*}
		y(s) = x(t_0 + s) - \phi_0^{\wedge F(t_0, \phi_0)}(s)
	\end{equation*}
satisfies the following: $I_0y = \boldsymbol{0}$, and for all $s \in [0, T]$
	\begin{align*}
		y'(s)
		&= F \Bigl( t_0 + s, I_s \Bigl[ \phi_0^{\wedge F(t_0, \phi_0)} \Bigr] + I_sy \Bigr) \\
		&= F \circ \tau_{t_0, \phi_0}^{F(t_0, \phi_0)}(s, I_sy).
	\end{align*}
Then the system
	\begin{equation*}
		\left\{
		\begin{alignedat}{2}
			y'(s) &= F \circ \tau_{t_0, \phi_0}^{F(t_0, \phi_0)}(s, I_sy), & \mspace{20mu} & s \ge 0, \\
			I_0y &= \boldsymbol{0}
		\end{alignedat}
		\right.
	\end{equation*}
is considered as the normalization of $\eqref{eq:IVP}_{t_0, \phi_0}$.

When $I = I^0$,
	\begin{equation*}
		\tau_{\sigma, \psi}^v(t, \phi)
		= \bigl( \sigma + t, I^0_t[\psi^{\wedge v}] + \phi \bigr)
		= (\sigma, \psi) + (t, tv + \phi).
	\end{equation*}
Here the identification $\psi(0) = \psi$ is used.
Therefore, $\tau_{\sigma, \psi}^0$ equals to the usual translation in $\mathbb{R} \times E$.

\subsection{Prolongations, and prolongation spaces}\label{subsec:prolongations}

\subsubsection{Prolongations}

Let $(\sigma, \psi) \in \mathbb{R} \times \Map(I, E)$.
We call a map $\gamma \colon J + I \to E$ a \textit{prolongation} of $(\sigma, \psi)$ if
\begin{enumerate}
\item[(i)] $J$ is a left-closed interval with the left end point $\sigma$,
\item[(ii)] $I_{\sigma}\gamma = \psi$, and
\item[(iii)] $\gamma|_J$ is continuous.
\end{enumerate}
Furthermore, when $\gamma|_J$ is of class $C^k$ ($k \in \mathbb{Z}_{\ge 0}$),
we call $\gamma$ a \textit{$C^k$-prolongation} of $(\sigma, \psi)$.
We include the degenerate case $J = \{\sigma\}$ and adopt a convention that
any prolongation $\gamma \colon \sigma + I \to E$ of $(\sigma, \psi)$ is a $C^k$-prolongation.
When $\sigma = 0$, we simply call $\gamma$ a $C^k$-prolongation of $\psi$.

This notion of prolongations has appeared in \cite{Hale 1969}.
In this paper, the prolongations and the $C^1$-prolongations are only used.
When a $C^1$-prolongation $\gamma$ of $\boldsymbol{0}$ satisfies $\gamma'(0) = 0$,
$\gamma$ is called a $C^1$-prolongation of $\boldsymbol{0}$ \textit{with $0$-derivative}.

\begin{remark}
Let $\gamma_1, \gamma_2 \colon [\sigma, \sigma + T] + I \to E$ be $C^1$-prolongations of $(\sigma, \psi)$.
Then for all $t \in [\sigma, \sigma + T]$, we have
	\begin{equation*}
		\gamma_i(t) = \psi(0) + \int_\sigma^t \gamma_i'(u) \mspace{2mu} \mathrm{d}u
			\mspace{20mu}
		(i = 1, 2).
	\end{equation*}
Therefore,
	\begin{align*}
		\sup_{t \in [\sigma, \sigma + T]} \|\gamma_1(t) - \gamma_2(t)\|_E
		&\le \int_\sigma^{\sigma + T} \|\gamma_1'(u) - \gamma_2'(u)\|_E \mspace{2mu} \mathrm{d}u \\
		&\le T \cdot \sup_{t \in [\sigma, \sigma + T]} \|\gamma_1'(t) - \gamma_2'(t)\|_E.
	\end{align*}
\end{remark}

\subsubsection{Prolongation spaces, and those topologies}

Let $(\sigma, \psi) \in \mathbb{R} \times \Map(I, E)$.
We consider the following spaces of prolongations and $C^1$-prolongations, respectively:
Let $T \ge 0$, $0 \le \delta \le \infty$, and $v \in E$.
\begin{itemize}
\item We define $\varGamma_{\sigma, \psi}(T, \delta)$
as the set of all prolongations $\gamma \colon [\sigma, \sigma + T] + I \to E$ of $(\sigma, \psi)$ satisfying
	\begin{equation*}
		\bigl\| \gamma(\sigma + \cdot) - \bar{\psi}(\cdot) \bigr\|_{C[0, T]} \le \delta.
	\end{equation*}
\item We define $\varGamma_{\sigma, \psi}^1(T, \delta, v)$ as the set of all $C^1$-prolongations
$\gamma \colon [\sigma, \sigma + T] + I \to E$ of $(\sigma, \psi)$ satisfying
	\begin{equation*}
		\bigl\| \gamma(\sigma + \cdot) - \psi^{\wedge v}(\cdot) \bigr\|_{C^1[0, T]} \le \delta,
		\mspace{15mu} \gamma'(\sigma) = v.
	\end{equation*}
Here $\gamma'(\sigma)$ means the right-hand derivative of $\gamma$ at $\sigma$.
\end{itemize}
The set $\varGamma_{\sigma, \psi}(T, \delta)$ has appeared and was used in \cite{Kappel--Schappacher 1980}.

For the above spaces of prolongations,
we consider metrics
	\begin{gather*}
		\rho^0 \colon \varGamma_{\sigma, \psi}(T, \delta) \times \varGamma_{\sigma, \psi}(T, \delta) \to \mathbb{R}_+, \\
		\rho^1 \colon \varGamma_{\sigma, \psi}^1(T, \delta, v) \times \varGamma_{\sigma, \psi}^1(T, \delta, v) \to \mathbb{R}_+
	\end{gather*}
defined by
	\begin{equation*}
		\rho^0(\gamma_1, \gamma_2) = \|\gamma_1 - \gamma_2\|_\infty,
			\mspace{15mu}
		\rho^1(\gamma_1, \gamma_2) = \|\gamma_1 - \gamma_2\|_{C^1[0, T]},
	\end{equation*}
respectively.

\begin{remark}\label{rmk:embedding of prolongations}
For $0 < T' \le T$, one can interpret
	\begin{equation*}
		\varGamma_{\sigma, \psi}(T', \delta) \subset \varGamma_{\sigma, \psi}(T, \delta)
	\end{equation*}
by considering the prolongation by constant of each element of $\varGamma_{\sigma, \psi}(T', \delta)$.
\end{remark}

The following are easy remarks.
\begin{itemize}
\item When $v = 0$, we have
	\begin{equation*}
		\varGamma_{\sigma, \psi}^1(T, \delta, 0) \subset \varGamma_{\sigma, \psi}(T, \delta)
	\end{equation*}
because
	$\bigl\| \gamma(\sigma + \cdot) - \bar{\psi}(\cdot) \bigr\|_{C[0, T]}
	\le \bigl\| \gamma(\sigma + \cdot) - \psi^{\wedge 0}(\cdot) \bigr\|_{C^1[0, T]}$.
\item When $T = 0$,
	\begin{equation*}
		\varGamma_{\sigma, \psi}(0, \delta)
		= \varGamma_{\sigma, \psi}^1(0, \delta, v)
		= \{\psi(\cdot - \sigma) \colon \sigma + I \to E\}
	\end{equation*}
for any $0 \le \delta \le \infty$ and $v \in E$.
\item When $\delta = 0$,
	\begin{align*}
		\varGamma_{\sigma, \psi}(T, 0)
		&= \left\{ \bar{\psi}(\cdot - \sigma) \colon [\sigma, \sigma + T] + I \to E \right\}, \\
		\varGamma_{\sigma, \psi}^1(T, 0, v)
		&= \left\{ \psi^{\wedge v}(\cdot - \sigma) \colon [\sigma, \sigma + T] + I \to E \right\}
	\end{align*}
for any $T \ge 0$ and $v \in E$.
\end{itemize}

\subsubsection{Transformations between prolongation spaces}

Let $(\sigma, \psi) \in \mathbb{R} \times \Map(I, E)$ and $v \in E$.
\begin{itemize}
\item We consider the transformation $A_{\sigma, \psi}^v$
for any prolongation $\beta \colon J_0 + I \to E$ of $\boldsymbol{0}$ defined by
	\begin{equation*}
		A_{\sigma, \psi}^v\beta(t)
		= \psi^{\wedge v}(t - \sigma) + \beta(t - \sigma)
		\mspace{20mu} (t \in (\sigma + J_0) + I).
	\end{equation*}
Then
	\begin{equation*}
		A_{\sigma, \psi}^v\beta \colon (\sigma + J_0) + I \to E
	\end{equation*}
is a prolongation of $(\sigma, \psi)$.
\item We also consider the transformation $N_{\sigma, \psi}^v$
for any prolongation $\gamma \colon J_\sigma + I \to E$ of $(\sigma, \psi)$ defined by
	\begin{equation*}
		N_{\sigma, \psi}^v\gamma(s)
		= \gamma(\sigma + s) - \psi^{\wedge v}(s) \mspace{20mu} (s \in (-\sigma + J_\sigma) + I).
	\end{equation*}
Then $N_{\sigma, \psi}^v$ is the inverse transformation of $A_{\sigma, \psi}^v$, and
	\begin{equation*}
		N_{\sigma, \psi}^v\gamma \colon (-\sigma + J_\sigma) + I \to E
	\end{equation*}
is a prolongation of $\boldsymbol{0}$.
\end{itemize}
These are transformations about the addition and the normalization, respectively.

\subsection{Prolongabilities, and regulation of topology by prolongations}\label{subsec:prolongability and regulation}

\subsubsection{Closedness under prolongations, and prolongabilities}

\begin{definition}\label{dfn:prolongabilities}
Let $H \subset \Map(I, E)$ be a history space.
\begin{itemize}
\item We say that $\psi \in \Map(I, E)$ is \textit{closed under $C^k$-prolongations} ($k \in \mathbb{Z}_{\ge 0}$) in $H$
if for every $C^k$-prolongation $\gamma \colon \mathbb{R}_+ + I \to E$ of $\psi$,
	\begin{equation*}
		I_t\gamma \in H \mspace{20mu} (\forall t \in \mathbb{R}_+)
	\end{equation*}
holds.
When every $\psi \in H$ is closed under $C^k$-prolongations in $H$,
we say that $H$ is closed under $C^k$-prolongations.
\item We say that $H$ is \textit{$C^k$-prolongable} if
	\begin{enumerate}
	\item[(i)] $H$ is closed under $C^k$-prolongations, and
	\item[(ii)] for every $C^k$-prolongation $\gamma \colon \mathbb{R}_+ + I \to E$ of each $\psi \in H$,
	the curve
		\begin{equation*}
			\mathbb{R}_+ \ni t \mapsto I_t\gamma \in H
		\end{equation*}
	is continuous.
	\end{enumerate}
\end{itemize}
\end{definition}

The notion of prolongability appeared in \cite{Hale 1969} for the case of $I = I^\infty$.

\begin{remark}
By the above definition, we have the following properties:
	\begin{equation*}
		\text{the closedness under prolongations $\Longrightarrow$ the closedness under $C^1$-prolongations},
	\end{equation*}
and
	\begin{equation*}
		\{\text{Prolongable history spaces}\} \subset \{\text{$C^1$-prolongable history spaces}\}.
	\end{equation*}
\end{remark}

\subsubsection{Rectangles by prolongations}

Let $(\sigma, \psi) \in \mathbb{R} \times \Map(I, E)$, $T \ge 0$, $0 \le \delta \le \infty$, and $v \in E$.
\begin{itemize}
\item We define a subset $\varLambda_{\sigma, \psi}(T, \delta)$ by
	\begin{equation*}
		\varLambda_{\sigma, \psi}(T, \delta)
		= \bigcup_{\tau \in [0, T]}
			\{ \mspace{2mu}
			(\sigma + \tau, I_{\sigma + \tau}\gamma) : \gamma \in \varGamma_{\sigma, \psi}(\tau, \delta) 
			\mspace{2mu}\}.
	\end{equation*}
We call $\varLambda_{\sigma, \psi}(T, \delta)$ a \textit{rectangle by prolongations}.
\item We define a subset $\varLambda^1_{\sigma, \psi}(T, \delta, v)$ by
	\begin{equation*}
		\varLambda^1_{\sigma, \psi}(T, \delta, v)
		= \bigcup_{\tau \in [0, T]}
			\bigl\{ \mspace{2mu}
			(\sigma + \tau, I_{\sigma + \tau}\gamma) : \gamma \in \varGamma_{\sigma, \psi}^1(\tau, \delta, v) 
			\mspace{2mu} \bigr\}.
	\end{equation*}
We call $\varLambda_{\sigma, \psi}^1(T, \delta, v)$ a \textit{rectangle by $C^1$-prolongations}.
\end{itemize}

\begin{remark}\label{rmk:comparison of rectangle by prolongations}
For all $0 \le T \le T'$ and all $0 \le \delta \le \delta' \le \infty$,
	\begin{equation*}
		\varLambda_{\sigma, \psi}(T, \delta) \subset \varLambda_{\sigma, \psi}(T', \delta'), \mspace{15mu}
		\varLambda_{\sigma, \psi}^1(T, \delta, v) \subset \varLambda_{\sigma, \psi}^1(T', \delta', v) .
	\end{equation*}
Since $\varGamma^1_{\sigma, \psi}(T, \delta, 0) \subset \varGamma_{\sigma, \psi}(T, \delta)$,
	\begin{equation*}
		\varLambda^1_{\sigma, \psi}(T, \delta, 0) \subset \varLambda_{\sigma, \psi}(T, \delta)
	\end{equation*}
holds.
\end{remark}

Rectangles by prolongations are ``thin'' subsets of $\mathbb{R} \times H$.
The following example is illustrative.

\begin{example}
Let $H = C([-1, 0], \mathbb{R})_\mathrm{u}$.
For each $\delta > 0$, let $\phi_\delta$ be the map whose value identically equals to $\delta$.
Then for any $T \in (0, 1)$ and any $\delta > 0$,
$(T, \phi_\delta) \not\in \varLambda_{0, \boldsymbol{0}}(T, \delta)$.
\end{example}

As this example shows,
the thinness of rectangles by prolongations originates in the property that
the history $I_t\gamma$ of a prolongation $\gamma$ of $\psi$ has an almost same information of $\psi$
when $t > 0$ is very small.
See also Proposition~\ref{prop:rectangle by prolongation, cpt support}.

\subsubsection{Regulation of topology by prolongations}

\begin{definition}[cf.\ \cite{Nishiguchi 2017}]
\label{dfn:regulation}
Let $H \subset \Map(I, E)$ be a history space.
\begin{itemize}
\item We say that $H$ is \textit{regulated by prolongations}
if for every $T > 0$ and every neighborhood $N$ of $\boldsymbol{0}$ in $H$,
there exists $\delta > 0$ such that
	\begin{equation*}
		\varLambda_{0, \boldsymbol{0}}(T, \delta) \subset [0, T] \times N.
	\end{equation*}
\item We say that $H$ is \textit{regulated by $C^1$-prolongations}
if for every $T > 0$ and every neighborhood $N$ of $\boldsymbol{0}$ in $H$,
there exists $\delta > 0$ such that
	\begin{equation*}
		\varLambda_{0, \boldsymbol{0}}^1(T, \delta, 0) \subset [0, T] \times N.
	\end{equation*}
\end{itemize}
\end{definition}

\begin{remark}\label{rmk:comparison of regulations by prolongations}
In view of $\varLambda^1_{\sigma, \psi}(T, \delta, 0) \subset \varLambda_{\sigma, \psi}(T, \delta)$,
the relationship
	\begin{equation*}
		\text{Regulation by prolongations $\Longrightarrow$ Regulation by $C^1$-prolongations}
	\end{equation*}
holds.
\end{remark}

\subsubsection{Neighborhoods by prolongations}

\begin{definition}\label{dfn:nbd by prolongations}
Suppose that $H$ is closed under prolongations.
Let $W$ be a subset of $\mathbb{R} \times H$ and $(\sigma_0, \psi_0) \in \mathbb{R} \times H$.
We say that $W$ is a \textit{neighborhood by prolongations}
of $(\sigma_0, \psi_0)$
if there exist $T, \delta > 0$ such that
	\begin{equation*}
		\varLambda_{\sigma_0, \psi_0}(T, \delta) \subset W
	\end{equation*}
holds.
We say that $W$ is a \textit{uniform neighborhood by prolongations} of $(\sigma_0, \psi_0)$
if there exist $T, \delta > 0$ and a neighborhood $W_0$ of $(\sigma_0, \psi_0)$ in $W$ such that
	\begin{equation*}
		\bigcup_{(\sigma, \psi) \in W_0} \varLambda_{\sigma, \psi}(T, \delta) \subset W
	\end{equation*}
holds.
\end{definition}

\begin{definition}\label{dfn:nbd by C^1-prolongations}
Suppose that $H$ is closed under $C^1$-prolongations.
Let $W$ be a subset of $\mathbb{R} \times H$ and $(\sigma_0, \psi_0) \in \mathbb{R} \times H$.
We say that $W$ is a \textit{neighborhood by $C^1$-prolongations} of $(\sigma_0, \psi_0)$
if for every $v \in E$, there exist $T, \delta > 0$ such that
	\begin{equation*}
		\varLambda_{\sigma_0, \psi_0}^1(T, \delta, v) \subset W
	\end{equation*}
holds.
We say that $W$ is a \textit{uniform neighborhood by $C^1$-prolongations} of $(\sigma_0, \psi_0)$
if for every $v_0 \in E$, there exist $T, \delta > 0$, a neighborhood $W_0$ of $(\sigma_0, \psi_0)$ in $W$,
and a neighborhood $V_0$ of $v_0$ in $E$ such that
	\begin{equation*}
		\bigcup_{(\sigma, \psi, v) \in W_0 \times V_0} \varLambda_{\sigma, \psi}^1(T, \delta, v) \subset W
	\end{equation*}
holds.
\end{definition}

\subsection{Lipschitz conditions}\label{subsec:Lip conditions}

In this paper, we introduce some Lipschitz conditions suited for RFDEs with history space $H$.
For this purpose, we consider the following type of Lipschitz condition for a parameter $L > 0$:
For $(t, \phi_1), (t, \phi_2) \in \dom(F)$ satisfying that $\supp(\phi_1 - \phi_2)$ is compact,
	\begin{equation*}
		\|F(t, \phi_1) - F(t, \phi_2)\|_E \le L \cdot \|\phi_1 - \phi_2\|_\infty. \tag{Lip}
	\end{equation*}
The inequality for a specific $L$ is denoted by $L$-\eqref{eq:Lipschitzian}.
The compactness of $\supp(\phi_1 - \phi_2)$ is used to ensure
	\begin{equation*}
		\|\phi_1 - \phi_2\|_\infty < \infty
	\end{equation*}
in the right-hand side.
We note that the inequality is independent from whether $H$ has a metric structure or not.
Relationships between Lipschitz conditions introduced here will be investigated in Section~\ref{sec:Lip conditions}.

\subsubsection{Lipschitz about prolongations and $C^1$-prolongations}

\begin{definition}[cf.\ \cite{Kappel--Schappacher 1980}, \cite{Nishiguchi 2017}]
\label{dfn:Lip about prolongations}
Suppose that $H$ is closed under prolongations.
We say that $F$ is \textit{locally Lipschitzian about prolongations} at $(\sigma_0, \psi_0) \in \dom(F)$
if there exist $T, \delta, L > 0$ such that
$L$-\eqref{eq:Lipschitzian} holds for all $(t, \phi_1), (t, \phi_2) \in \varLambda_{\sigma_0, \psi_0}(T, \delta) \cap \dom(F)$.
We simply say that $F$ is locally Lipschitzian about prolongations
when $F$ is locally Lipschitzian about prolongations at each $(\sigma_0, \psi_0) \in \dom(F)$.
\end{definition}

\begin{remark}
When $\dom(F)$ is a neighborhood of prolongations of $(\sigma_0, \psi_0)$,
$\varLambda_{\sigma_0, \psi_0}(T, \delta) \subset \dom(F)$ holds by choosing sufficiently small $T, \delta > 0$.
Therefore,
	\begin{equation*}
		\varLambda_{\sigma_0, \psi_0}(T, \delta) \cap \dom(F) \ne \emptyset.
	\end{equation*}
\end{remark}

The Carath\'{e}odory type version of the above condition was introduced in \cite[(A3) in pp.156]{Kappel--Schappacher 1980}.
Kappel \& Schappacher obtained the uniqueness result by using the Carath\'{e}odory type condition.
In \cite{Kappel--Schappacher 1980} and \cite{Nishiguchi 2017},
the condition was stated without rectangles by prolongations.

The uniform version of the above Lipschitz condition with respect to an initial condition
was introduced in \cite{Nishiguchi 2017}.

\begin{definition}[cf.\ \cite{Nishiguchi 2017}]
\label{dfn:Lip about uniform prolongations}
Suppose that $H$ is closed under prolongations.
We say that $F$ is \textit{uniformly locally Lipschitzian about prolongations} at $(\sigma_0, \psi_0) \in \dom(F)$
if there exist a neighborhood $W_0$ of $(\sigma_0, \psi_0)$ in $\dom(F)$ and $T, \delta, L > 0$ such that
$L$-\eqref{eq:Lipschitzian} holds
for all $(\sigma, \psi) \in W_0$ and all $(t, \phi_1), (t, \phi_2) \in \varLambda_{\sigma, \psi}(T, \delta) \cap \dom(F)$.
We simply say that $F$ is uniformly locally Lipschitzian about prolongations
when $F$ is uniformly locally Lipschitzian about prolongations at each $(\sigma_0, \psi_0) \in \dom(F)$.
\end{definition}

We now introduce the Lipschitz condition about $C^1$-prolongations by using rectangles by $C^1$-prolongations.
For $v = F(\sigma, \phi)$, we write
	\begin{equation*}
		\varLambda_{\sigma, \psi}^1(T, \delta; F) := \varLambda_{\sigma, \psi}^1(T, \delta, v)
	\end{equation*}
throughout the paper.

\begin{definition}\label{dfn:Lip about C^1-prolongations}
Suppose that $H$ is closed under $C^1$-prolongations.
We say that $F$ is \textit{locally Lipschitzian about $C^1$-prolongations} at $(\sigma_0, \psi_0) \in \dom(F)$
if there exist $T, \delta, L > 0$ such that $L$-\eqref{eq:Lipschitzian} holds
for all $(t, \phi_1), (t, \phi_2) \in \varLambda_{\sigma_0, \psi_0}^1(T, \delta; F) \cap \dom(F)$.
We simply say that $F$ is locally Lipschitzian about $C^1$-prolongations
when $F$ is locally Lipschitzian about $C^1$-prolongations at each point $(\sigma_0, \psi_0) \in \dom(F)$.
\end{definition}

\begin{definition}\label{dfn:uniform Lip about C^1-prolongations}
Suppose that $H$ is closed under $C^1$-prolongations.
We say that $F$ is \textit{uniformly locally Lipschitzian about $C^1$-prolongations} at $(\sigma_0, \psi_0) \in \dom(F)$
if there exist a neighborhood $W_0$ of $(\sigma_0, \psi_0)$ in $\dom(F)$ and $T, \delta, L > 0$ such that
$L$-\eqref{eq:Lipschitzian} holds
for all $(\sigma, \psi) \in W_0$ and all $(t, \phi_1), (t, \phi_2) \in \varLambda_{\sigma, \psi}^1(T, \delta; F) \cap \dom(F)$.
We simply say that $F$ is uniformly locally Lipschitzian about $C^1$-prolongations
when $F$ is uniformly locally Lipschitzian about $C^1$-prolongations at each $(\sigma_0, \psi_0) \in \dom(F)$.
\end{definition}

\subsubsection{Lipschitz about memories}

\begin{definition}[\cite{Nishiguchi 2017}]\label{dfn:Lip about memories}
We say that $F$ is \textit{Lipschitzian about memories}
if there exist $R, L > 0$ such that
$L$-\eqref{eq:Lipschitzian} holds for all $(t, \phi_1), (t, \phi_2) \in \dom(F)$ satisfying the conditions
\begin{enumerate}
\item[(i)] $\supp(\phi_1 - \phi_2) \subset [-R, 0] \subset I$, and
\item[(ii)] $\phi_1 - \phi_2$ is continuous.
\end{enumerate}
We say that $F$ is \textit{locally Lipschitzian about memories} at $(\sigma_0, \psi_0) \in \dom(F)$
if there exists a neighborhood $W$ of $(\sigma_0, \psi_0)$ in $\mathbb{R} \times H$ such that
$F|_{W \cap \dom(F)}$ is Lipschitzian about memories.
When $F$ is locally Lipschitzian about memories at each $(\sigma_0, \psi_0) \in \dom(F)$,
we simply say that $F$ is locally Lipschitzian about memories.
\end{definition}

We generalize the notion of local Lipschitz about memories as follows.

\begin{definition}\label{dfn:almost Lip about memories}
Suppose that $H$ is closed under $C^1$-prolongations.
We say that $F$ is \textit{locally Lipschitzian about Lip-memories} at $(\sigma_0, \psi_0) \in \dom(F)$
if for every $M > 0$, there exist a neighborhood $W$ of $(\sigma_0, \psi_0)$ in $\mathbb{R} \times H$ and $R, L > 0$ such that
$L$-\eqref{eq:Lipschitzian} holds for all $(t, \phi_1), (t, \phi_2) \in W \cap \dom(F)$ satisfying the conditions
\begin{enumerate}
\item[(i)] $\supp(\phi_1 - \phi_2) \subset [-R, 0] \subset I$, and
\item[(ii)] $\lip \bigl( \phi_1|_{[-R, 0]} \bigr), \lip \bigl( \phi_2|_{[-R, 0]} \bigr) \le M$.
\end{enumerate}
When $F$ is locally Lipschitzian about Lip-memories at each $(\sigma_0, \psi_0) \in \dom(F)$,
we simply say that $F$ is locally Lipschitzian about Lip-memories.
\end{definition}

\begin{remark}
Condition (ii) implies the continuity of $(\phi_1 - \phi_2)|_{[-R, 0]}$,
and condition (i) implies the left-continuity of $\phi_1 - \phi_2$ at $-R$.
Therefore, by combining (i) and (ii), the continuity of $\phi_1 - \phi_2$ is derived.
Thus, we have
	\begin{equation*}
		\text{Local Lipschitz about memories $\imply$ Local Lipschitz about Lip-memories}
	\end{equation*}
by the definitions.
\end{remark}

\subsubsection{Almost local Lipschitz}

\begin{definition}[ref.\ \cite{Mallet-Paret--Nussbaum--Paraskevopoulos 1994}]
\label{dfn:almost locally Lip for I^r}
Let $I = I^r$ for some $r > 0$.
Suppose $H \supset C^{0, 1}(I^r, E)$.
We say that $F$ is \textit{almost locally Lipschitzian} at $(\sigma_0, \psi_0) \in \dom(F)$
if for every $M > 0$, there exist a neighborhood $W$ of $(\sigma_0, \psi_0)$ in $\mathbb{R} \times H$ and $L > 0$ such that
$L$-\eqref{eq:Lipschitzian} holds for all $(t, \phi_1), (t, \phi_2) \in W \cap \dom(F)$
satisfying $\lip(\phi_i) \le M$ $(i = 1, 2)$.
\end{definition}

The notion of almost local Lipschitz was originally introduced by
Mallet-Paret, Nussbaum, \& Paraskevopoulos~\cite[Definition 1.1]{Mallet-Paret--Nussbaum--Paraskevopoulos 1994}
for autonomous RFDEs with the history space $C([-r, 0], \mathbb{R}^n)_\mathrm{u}$ ($r > 0$).

We generalize the notion of almost local Lipschitz for the case $I = I^\infty$.

\begin{definition}\label{dfn:almost locally Lip for I^infty}
Let $I = I^\infty$.
Suppose $H \supset C^{0, 1}_\mathrm{loc}(I^\infty, E)$.
We say that $F$ is \textit{almost locally Lipschitzian} at $(\sigma_0, \psi_0) \in \dom(F)$
if for every sequence $(M_k)_{k = 1}^\infty$ of positive numbers,
there exist a neighborhood $W$ of $(\sigma_0, \psi_0)$ in $\mathbb{R} \times H$ and $L > 0$ such that
$L$-\eqref{eq:Lipschitzian} holds for all $(t, \phi_1), (t, \phi_2) \in W \cap \dom(F)$
satisfying the conditions
\begin{enumerate}
\item[(i)] $\supp(\phi_1 - \phi_2)$ is compact, and
\item[(ii)] $\lip \bigl( \phi_1|_{[-k, 0]} \bigr), \lip \bigl( \phi_2|_{[-k, 0]} \bigr) \le M_k$ for all $k \ge 1$.
\end{enumerate}
\end{definition}

\section{Prolongations and history spaces}\label{sec:prolongations and history spaces}

Let $I$ be a past interval and $E = (E, \|\cdot\|_E)$ be a Banach space.
The purpose of this section is to examine the notions introduced in Section~\ref{sec:formulation and notions}
except the Lipschitz conditions.

\subsection{Fundamental properties}

\subsubsection{Closedness under prolongations}

In this subsection, let $H \subset \Map(I, E)$ be a history space.
The properties of closedness under prolongations bring us relations
with the spaces of continuous maps and $C^1$-maps with compact support.

\begin{proposition}[\cite{Nishiguchi 2017}, \cite{Hino--Murakami--Naito 1991}]
\label{prop:closedness under prolongations}
Let $I = I^r$ for some $0 \le r \le \infty$.
If $\boldsymbol{0}$ is closed under prolongations in $H$, then $C_\mathrm{c}(I^r, E) \subset H$ holds.
\end{proposition}

The case $r = \infty$ is treated in \cite[Proposition 2.1]{Hino--Murakami--Naito 1991} by the translation.
The case $r < \infty$ was obtained in \cite{Nishiguchi 2017}.

\begin{proof}[\textbf{Proof of Proposition~\ref{prop:closedness under prolongations}}]
Suppose $r < \infty$ and let $\phi \in C(I^r, E)$.
We define a continuous map $\gamma \colon \mathbb{R}_+ + I \to E$ by
	\begin{equation*}
		\gamma(t) =
		\begin{cases}
			0 & (-r \le t \le 0), \\
			t\phi(-r) & (0 \le t \le 1), \\
			\bar{\phi}(t - (r + 1)) & (t \ge 1).
		\end{cases}
	\end{equation*}
Then $\gamma$ is a prolongation of $\boldsymbol{0}$.
Therefore, we have $\phi = I^r_{r + 1}\gamma \in H$ by the assumption.
\end{proof}

\begin{proposition}\label{prop:C^1 subset H}
\label{prop:closedness under C^1-prolongations}
Let $I = I^r$ for some $0 \le r \le \infty$.
If $\boldsymbol{0}$ is closed under $C^1$-prolongations with $0$-derivative in $H$,
then $C^1_\mathrm{c}(I^r, E) \subset H$ holds.
\end{proposition}

\begin{proof}
We divide the proof into the following two cases: $r = \infty$ and $r < \infty$.

\begin{itemize}
\item Case 1: $r = \infty$.
Let $\phi \in C^1_\mathrm{c}(I^\infty, E)$.
We choose $R > 0$ so that $\supp(\phi) \subset [-R, 0]$.
We consider the map $\gamma \colon \mathbb{R}_+ + I \to E$ given by
	\begin{equation*}
		\gamma(t) = \phi^{\wedge \phi'(0)}(t - R).
	\end{equation*}
Then $\gamma|_{\mathbb{R}^+}$ is of class $C^1$, where $\gamma'(0) = 0$.
Therefore,
$\gamma$ is a $C^1$-prolongation of $\boldsymbol{0}$ with $0$-derivative.
Therefore, $\phi = I^\infty_R\gamma \in H$ by the assumption.
\item Case 2: $r < \infty$.
Let $\phi \in C^1(I^r, E)$.
We construct a map $\ell \colon [0, 1] \to E$ by
	\begin{equation*}
		\ell(s) = s^2[-\phi'(-r) + 3\phi(-r)] + s^3[\phi'(-r) - 2\phi(-r)]
	\end{equation*}
so that
	\begin{alignat*}{2}
		\ell(0) &= 0, & \mspace{20mu} \ell(1) &= \phi(-r), \\
		\ell'(0) &= 0, & \ell'(1) &= \phi'(-r).
	\end{alignat*}
We define $\gamma \colon \mathbb{R}_+ + I \to E$ by
	\begin{equation*}
		\gamma(t) =
		\begin{cases}
			0 & (-r \le t \le 0), \\
			\ell(t) & (0 \le t \le 1), \\
			\phi^{\wedge \phi'(0)}(t - (r + 1)) & (t \ge 1).
		\end{cases}
	\end{equation*}
By the construction, $\gamma$ is a $C^1$-prolongation of $\boldsymbol{0}$ with $0$-derivative.
Therefore, $\phi = I^r_{r + 1}\gamma \in H$ holds by the assumption.
\end{itemize}
This completes the proof.
\end{proof}

\subsubsection{Prolongation spaces}

By definition, the following properties hold:
\begin{itemize}
\item $\gamma \in \varGamma_{\sigma, \psi}(T, \delta)$ is equivalent to
$N_{\sigma, \psi}^0\gamma \in \varGamma_{0, \boldsymbol{0}}(T, \delta)$, and
	\begin{equation*}
		N_{\sigma, \psi}^0 \colon \varGamma_{\sigma, \psi}(T, \delta) \to \varGamma_{0, \boldsymbol{0}}(T, \delta)
	\end{equation*}
is an isometry with respect to $\rho^0$.
\item $\gamma \in \varGamma_{\sigma, \psi}^1(T, \delta, v)$ is equivalent to
$N_{\sigma, \psi}^v\gamma \in \varGamma_{0, \boldsymbol{0}}^1(T, \delta, 0)$, and
	\begin{equation*}
		N_{\sigma, \psi}^v \colon \varGamma_{\sigma, \psi}^1(T, \delta, v) \to \varGamma_{0, \boldsymbol{0}}^1(T, \delta, 0)
	\end{equation*}
is an isometry with respect to $\rho^1$.
\end{itemize}

\paragraph{Completeness}

We have the following completeness results for the prolongation spaces.

\begin{proposition}[\cite{Nishiguchi 2017}]\label{prop:completeness of space of prolongations}
Let $(\sigma, \psi) \in \mathbb{R} \times \Map(I, E)$ and $T, \delta > 0$.
Then the metric space $\bigl( \varGamma_{\sigma, \psi}(T, \delta), \rho^0 \bigr)$ is complete.
\end{proposition}

\begin{proposition}\label{prop:completeness of space of C^1-prolongations}
Let $(\sigma, \psi) \in \mathbb{R} \times \Map(I, E)$, $v \in E$, and $T, \delta > 0$.
Then the metric space $\bigl( \varGamma_{\sigma, \psi}^1(T, \delta, v), \rho^1 \bigr)$ is complete.
\end{proposition}

\begin{proof}
We consider a closed linear subspace $X$ of the Banach space $(C^1([0, T], E), \|\cdot\|_{C^1})$ given by
	\begin{equation*}
		X := \{\mspace{2mu} \chi \in C^1([0, T], E) : \chi(0) = 0, \chi'(0) = 0, \|\chi\|_{C^1} \le \delta \mspace{2mu}\}.
	\end{equation*}
Therefore, $(X, \|\cdot\|_{C^1})$ is also a Banach space.
The completeness of $\bigl( \varGamma_{0, \boldsymbol{0}}^1(T, \delta, 0), \rho^1 \bigr)$ is obtained
because $j \colon \varGamma_{0, \boldsymbol{0}}^1(T, \delta, 0) \to X$ defined by
	\begin{equation*}
		j(\gamma) = \gamma|_{[0, T]}
	\end{equation*}
is isometrically isomorphic.
\end{proof}

\paragraph{Comparison of prolongations spaces}

\begin{lemma}\label{lem:addition and spaces of prolongations}
Let $B$ be a bounded subset of $E$ and $T, \delta \in (0, \infty)$.
Then there exist $0 < T_0 \le T$ and $0 < \delta_0 \le \delta/2$ such that
for all $0 < T' \le T_0$, all $(\sigma, \psi) \in \mathbb{R} \times \Map(I, E)$, and all $v \in B$,
	\begin{equation*}
		A_{\sigma, \psi}^v(\varGamma_{0, \boldsymbol{0}}(T', \delta_0)) \subset \varGamma_{\sigma, \psi}(T', \delta)
	\end{equation*}
holds.
\end{lemma}

\begin{proof}
Choose $M > 0$ so that $B \subset \bar{B}_E(0; M)$.
Let $T_0$ and $\delta_0$ be given so that
	\begin{equation*}
		0 < T_0 \le \min\{\delta/(2M), T\}, \mspace{15mu} 0 < \delta_0 \le \delta/2.
	\end{equation*}
Let $0 < T' \le T_0$, $(\sigma, \psi) \in \mathbb{R} \times \Map(I, E)$, and $v \in B$.
Then for all $\beta \in \varGamma_{0, \boldsymbol{0}}(T', \delta_0)$, we have
	\begin{align*}
		\bigl\| (A_{\sigma, \psi}^v\beta)(\sigma + \cdot) - \bar{\psi}(\cdot) \bigr\|_{C[0, T']}
		&= \sup_{s \in [0, T']} \|sv + \beta(s)\|_E \\
		&\le T'\|v\|_E + \|\beta\|_\infty \\
		&\le T_0M + \delta_0.
	\end{align*}
Since $T_0M + \delta_0 \le (\delta/2) + (\delta/2) = \delta$,
this shows $A_{\sigma, \psi}^v\beta \in \varGamma_{\sigma, \psi}(T', \delta)$.
\end{proof}

\begin{proposition}\label{prop:comparison of spaces of prolongations}
Let $B$ be a bounded subset of $E$ and $T, \delta \in (0, \infty)$.
Then there exist $0 < T_0 \le T$ and $0 < \delta_0 \le \delta/2$ such that
for all $0 < T' \le T_0$, all $(\sigma, \psi) \in \mathbb{R} \times \Map(I, E)$, and all $v \in B$,
	\begin{equation*}
		\varGamma_{\sigma, \psi}^1(T', \delta_0, v) \subset \varGamma_{\sigma, \psi}(T', \delta)
	\end{equation*}
holds.
\end{proposition}

\begin{proof}
By Lemma~\ref{lem:addition and spaces of prolongations},
we choose $0 < T_0 \le T$ and $0 < \delta_0 \le \delta/2$ so that
	\begin{equation*}
		A_{\sigma, \psi}^v(\varGamma_{0, \boldsymbol{0}}(T', \delta_0)) \subset \varGamma_{\sigma, \psi}(T', \delta)
	\end{equation*}
holds for all $0 < T' \le T_0$, $(\sigma, \psi) \in \mathbb{R} \times \Map(I, E)$, and $v \in B$.
Then we have
	\begin{equation*}
		\varGamma_{\sigma, \psi}^1(T', \delta_0, v)
		= A_{\sigma, \psi}^v(\varGamma_{0, \boldsymbol{0}}^1(T', \delta_0, 0))
		\subset A_{\sigma, \psi}^v(\varGamma_{0, \boldsymbol{0}}(T', \delta_0)).
	\end{equation*}
Therefore, the conclusion is obtained.
\end{proof}

\subsubsection{Rectangles by prolongations and $C^1$-prolongations}

The following lemma states the relationship between rectangles by prolongations and $C^1$-prolongations.

\begin{lemma}\label{lem:comparison of rectangles}
Let $B \subset E$ be a bounded set, and let $T, \delta > 0$.
Then there exist $0 < T_0 \le T$ and $0 < \delta_0 \le \delta/2$ such that
for all $(\sigma, \psi) \in \mathbb{R} \times \Map(I, E)$ and all $v \in B$,
	\begin{equation*}
		\varLambda_{\sigma, \psi}^1(T_0, \delta_0, v) \subset \varLambda_{\sigma, \psi}(T_0, \delta)
	\end{equation*}
holds.
\end{lemma}

\begin{proof}
From Proposition~\ref{prop:comparison of spaces of prolongations},
we choose $0 < T_0 \le T$ and $0 < \delta_0 \le \delta/2$ so that
	\begin{equation*}
		\varGamma_{\sigma, \psi}^1(T', \delta_0, v) \subset \varGamma_{\sigma, \psi}(T', \delta)
	\end{equation*}
holds
for all $0 < T' \le T_0$, $(\sigma, \psi) \in \mathbb{R} \times \Map(I, E)$, and $v \in B$.

Let $(\sigma, \psi) \in \mathbb{R} \times \Map(I, E)$, $v \in B$,
and $(t, \phi) \in \varLambda_{\sigma, \psi}^1(T_0, \delta_0, v)$.
Then $t = \sigma + \tau$ for some $\tau \in [0, T_0]$, and
$\phi = I_{\sigma + \tau}\gamma$ for some $\gamma \in \varGamma^1_{\sigma, \psi}(\tau, \delta_0, v)$.
This implies $(t, \phi) \in \varLambda_{\sigma, \psi}(T_0, \delta)$
in view of $\varGamma_{\sigma, \psi}^1(\tau, \delta_0, v) \subset \varGamma_{\sigma, \psi}(\tau, \delta)$.
\end{proof}

\begin{remark}
From Lemma~\ref{lem:comparison of rectangles},
	\begin{equation*}
		\{\mspace{2mu} \text{Neighborhoods by prolongations} \mspace{2mu}\}
		\subset
		\{\mspace{2mu} \text{Neighborhoods by $C^1$-prolongations} \mspace{2mu}\}.
	\end{equation*}
The inclusion is also true for uniform neighborhoods by prolongations and $C^1$-prolongations.
\end{remark}

As the following lemma shows, the translations act on rectangles by prolongations as the change of base points.

\begin{lemma}\label{lem:translation on rectangles}
Let $(\sigma, \psi) \in \mathbb{R} \times \Map(I, E)$, $T, \delta > 0$, and $v \in E$.
Then the following statements hold:
\begin{enumerate}
\item[\emph{(i)}]
	$\tau_{\sigma, \psi}^0 \colon \varLambda_{0, \boldsymbol{0}}(T, \delta) \to \varLambda_{\sigma, \psi}(T, \delta)$
is an well-defined bijection.
\item[\emph{(ii)}]
	$\tau_{\sigma, \psi}^v \colon
	\varLambda_{0, \boldsymbol{0}}^1(T, \delta, 0) \to \varLambda_{\sigma, \psi}^1(T, \delta, v)$
is an well-defined bijection.
\end{enumerate}
\end{lemma}

\begin{proof}
Let $(t, \phi) \in \varLambda_{0, \boldsymbol{0}}(T, \delta)$.
Then $t \in [0, T]$, and $\phi = I_t\beta$ for some $\beta \in \varGamma_{0, \boldsymbol{0}}(t, \delta)$.
Therefore,
	\begin{align*}
		\tau_{\sigma, \psi}^0(t, \phi)
		&= \bigl( \sigma + t, I_t\bar{\psi} + \phi \bigr) \\
		&= \bigl( \sigma + t, I_t \bigl[ \bar{\psi} + \beta \bigr] \bigr) \\
		&= \bigl( \sigma + t, I_{\sigma + t}[A_{\sigma, \psi}^0\beta] \bigr).
	\end{align*}
Since $A_{\sigma, \psi}^0\beta \in \varGamma_{\sigma, \psi}(t, \delta)$,
we have $\tau_{\sigma, \psi}^0(t, \phi) \in \varLambda_{\sigma, \psi}(T, \delta)$.
The bijectivity of the restricted map follows by the bijectivity of
	\begin{equation*}
		A_{\sigma_0, \psi_0}^0 \colon
		\varGamma_{0, \boldsymbol{0}}(t, \delta)
		\to
		\varGamma_{\sigma, \psi}(t, \delta)
	\end{equation*}
for every $t \in [0, T]$.

(ii) We omit the proof because the similar proof to (i) is valid.
\end{proof}

The next proposition shows the thinness of the rectangles by prolongations and $C^1$-prolongations.

\begin{proposition}\label{prop:rectangle by prolongation, cpt support}
Suppose $|I| > 0$.
Then for all $0 < T < |I|$ and all $0 \le \delta \le \infty$,
	\begin{align*}
		\varLambda_{0, \boldsymbol{0}}(T, \delta)
		&= \bigcup_{\tau \in [0, T]} \{\mspace{2mu} (\tau, \phi) \in [0, T] \times C_\mathrm{c}(I, E) :
			\text{$\supp(\phi) \subset [-\tau, 0]$, $\|\phi\|_\infty \le \delta$} \mspace{2mu}\}, \\
		\varLambda_{0, \boldsymbol{0}}^1(T, \delta, 0)
		&= \bigcup_{\tau \in [0, T]} \{\mspace{2mu} (\tau, \phi) \in [0, T] \times C^1_\mathrm{c}(I, E) :
			\text{$\supp(\phi) \subset [-\tau, 0]$, $\|\phi\|_{C^1} \le \delta$} \mspace{2mu}\}
	\end{align*}
hold.
\end{proposition}

\begin{proof}
We only show the expression of $\varLambda_{0, \boldsymbol{0}}(T, \delta)$
because the same proof is valid for $\varLambda_{0, \boldsymbol{0}}^1(T, \delta, 0)$.

($\subset$)
Let $(\tau, \phi) \in \varLambda_{0, \boldsymbol{0}}(T, \delta)$.
Then $\tau \in [0, T]$, and $\phi = I_\tau\beta$ for some $\beta \in \varGamma_{0, \boldsymbol{0}}(T, \delta)$.
Therefore, we have $\supp(\phi) = \supp(I_\tau\beta) \subset [-\tau, 0]$
and $\|\phi\|_\infty = \|I_\tau\beta\|_\infty = \|\beta\|_\infty \le \delta$.

($\supset$)
Let $\tau \in [0, T]$ and $\phi \in C_\mathrm{c}(I, E)$
satisfying $\supp(\phi) \subset [-\tau, 0]$ and $\|\phi\|_\infty \le \delta$.
We define $\beta \colon [0, \tau] + I \to E$ by
	\begin{equation*}
		\beta(t) =
		\begin{cases}
			\phi(t - \tau), & 0 \le t \le \tau, \\
			0, & t \in I.
		\end{cases}
	\end{equation*}
Then $\beta$ is a prolongation of $\boldsymbol{0}$ and $\|\beta\|_\infty \le \delta$.
Since $\phi = I_\tau\beta$, we have $(\tau, \phi) \in \varLambda_{0, \boldsymbol{0}}(T, \delta)$.

This completes the proof.
\end{proof}

\subsection{Characterizations of regulation by prolongations}

Let $H \subset \Map(I, E)$ be a history space.
In this subsection, we give some characterizations of the properties of regulation by prolongations and $C^1$-prolongations.

\subsubsection{Compact past interval}

\begin{theorem}\label{thm:r < infty, regulation}
Let $0 \le r < \infty$ and $H \subset \Map(I^r, E)$ be a history space.
We assume that $\boldsymbol{0}$ is closed under prolongations in $H$.
Then the following properties are equivalent:
\begin{enumerate}
\item[\emph{(a)}] $H$ is regulated by prolongations.
\item[\emph{(b)}] The topology of $H$ is coarser than (or equal to) the topology of uniform convergence on $C(I^r, E)$.
\end{enumerate}
\end{theorem}

In the above statement, the inclusion
	\begin{equation*}
		C(I^r, E) \subset H
	\end{equation*}
for $r < \infty$ (see Proposition~\ref{prop:closedness under prolongations}) is fundamental.
For the proof of Theorem~\ref{thm:r < infty, regulation}, we use the following fact:
For linear topological spaces $(X, \tau)$ and $(X, \tau')$,
$\tau$ is finer than (or equal to) $\tau'$ if and only if
for every neighborhood $N'$ of $0$ with respect to $\tau'$,
there exists a neighborhood $N$ of $0$ with respect to $\tau$ such that $N \subset N'$.
Therefore, the property (b) is equivalent to the continuity of the inclusion map
	\begin{equation*}
		i \colon C(I^r, E)_\mathrm{u} \to H.
	\end{equation*}

\begin{proof}[\textbf{Proof of Theorem~\ref{thm:r < infty, regulation}}]
(a) $\Rightarrow$ (b):
Let $N$ be a neighborhood of $\boldsymbol{0}$ in $H$.
By the assumption, there is $\delta > 0$ such that
	\begin{equation*}
		\varLambda_{0, \boldsymbol{0}}(r + 1, \delta) \subset [0, r + 1] \times N.
	\end{equation*}
For each $\phi \in C(I^r, E)$, we consider the map $\gamma_\phi \colon [0, r + 1] + I^r \to E$ defined by
	\begin{equation*}
		\gamma_\phi(t) =
		\begin{cases}
			0 & (-r \le t \le 0), \\
			t\phi(-r) & (0 \le t \le 1), \\
			\phi(t - (r + 1)) & (1 \le t \le r + 1).
		\end{cases}
	\end{equation*}
Then $\gamma_\phi$ is a prolongation of $\boldsymbol{0}$ and satisfies
	\begin{equation*}
		I^r_{r + 1}\gamma_\phi = \phi, \mspace{15mu} \|\gamma_\phi\|_\infty \le \|\phi\|_\infty.
	\end{equation*}
Therefore, $\|\phi\|_\infty \le \delta$ implies
	\begin{equation*}
		(r + 1, \phi) = (r + 1, I^r_{r + 1}\gamma_\phi) \in \varLambda_{0, \boldsymbol{0}}(r + 1, \delta).
	\end{equation*}
By combining this and $\varLambda_{0, \boldsymbol{0}}(r + 1, \delta) \subset [0, r + 1] \times N$,
we have $\phi \in N$.
This shows
	\begin{equation*}
		\{\mspace{2mu} \phi \in C(I^r, E) : \|\phi\|_\infty \le \delta \mspace{2mu}\} \subset N,
	\end{equation*}
and we obtain (b).

(b) $\Rightarrow$ (a):
Let $T > 0$ and $N$ be a neighborhood of $\boldsymbol{0}$ in $H$.
By the assumption, there is $\delta > 0$ such that
	\begin{equation*}
		\{\mspace{2mu} \phi \in C(I^r, E) : \|\phi\|_\infty \le \delta \mspace{2mu}\} \subset N.
	\end{equation*}
For every $\tau \in [0, T]$ and every prolongation $\gamma \colon [0, \tau] + I^r \to E$ of $\boldsymbol{0}$,
	\begin{equation*}
		I^r_\tau\gamma \in C(I^r, E), \mspace{15mu} \|I^r_\tau\gamma\|_\infty \le \|\gamma\|_{C[0, \tau]}.
	\end{equation*}
This shows $\varLambda_{0, \boldsymbol{0}}(T, \delta) \subset [0, T] \times N$.
\end{proof}

The following is a $C^1$-version of Theorem~\ref{thm:r < infty, regulation}.
The result $C^1(I^r, E) \subset H$ for $r < \infty$
(see Proposition~\ref{prop:closedness under C^1-prolongations}) is fundamental,
and the property (b) in Theorem~\ref{thm:r < infty, C^1-regulation} is equivalent to the continuity of the inclusion
	\begin{equation*}
		i \colon \bigl( C^1(I^r, E), \|\cdot\|_{C^1} \bigr) \to H.
	\end{equation*}

\begin{theorem}\label{thm:r < infty, C^1-regulation}
Let $0 \le r < \infty$ and $H \subset \Map(I^r, E)$ be a history space.
We assume that $\boldsymbol{0}$ is closed under $C^1$-prolongations with $0$-derivative in $H$.
Then the following properties are equivalent:
\begin{enumerate}
\item[\emph{(a)}] $H$ is regulated by $C^1$-prolongations.
\item[\emph{(b)}] The topology of $H$ is coarser than (or equal to) the topology of uniform $C^1$-convergence on $C^1(I^r, E)$.
\end{enumerate}
\end{theorem}

\begin{proof}
(a) $\Rightarrow$ (b):
Let $N$ be a neighborhood of $\boldsymbol{0}$ in $H$.
By the assumption, there is $\delta > 0$ such that
	\begin{equation*}
		\varLambda_{0, \boldsymbol{0}}^1(r + 1, \delta, 0) \subset [0, r + 1] \times N.
	\end{equation*}
For each $\phi \in C^1(I^r, E)$, we consider the $\gamma_\phi \colon [0, r + 1] + I^r \to E$ defined by
	\begin{equation*}
		\gamma_\phi(t) =
		\begin{cases}
			0 & (-r \le t \le 0) \\
			t^2[-\phi'(-r) + 3\phi(-r)] + t^3[\phi'(-r) - 2\phi(-r)] & (0 \le t \le 1) \\
			\phi(t - (r + 1)) & (1 \le t \le r + 1).
		\end{cases}
	\end{equation*}
Then $\gamma_\phi$ is a $C^1$-prolongation of $\boldsymbol{0}$ with $0$-derivative
(see Proposition~\ref{prop:C^1 subset H}) and $I^r_{r + 1}\gamma_\phi = \phi$.
Since for all $t \in [0, 1]$
	\begin{align*}
		&\|\gamma_\phi(t)\|_E \le 2\|\phi'(-r)\|_E + 5\|\phi(-r)\|_E \le 5\|\phi\|_{C^1}, \\
		&\|(\gamma_\phi)'(t)\|_E \le 5\|\phi'(-r)\|_E + 12\|\phi(-r)\|_E \le 12\|\phi\|_{C^1},
	\end{align*}
we have $\|\gamma_\phi\|_{C^1[0, r + 1]} \le 12\|\phi\|_{C^1}$.
Therefore, $\|\phi\|_{C^1} \le \delta/12$ implies
	\begin{equation*}
		(r + 1, \phi) = (r + 1, I^r_{r + 1}\gamma_\phi) \in \varLambda_{0, \boldsymbol{0}}^1(r + 1, \delta, 0).
	\end{equation*}
By combining this and $\varLambda_{0, \boldsymbol{0}}^1(r + 1, \delta, 0) \subset [0, r + 1] \times N$,
we have $\phi \in N$.
This shows
	\begin{equation*}
		\{\mspace{2mu} \phi \in C^1(I^r, E) : \|\phi\|_\infty \le \delta/12 \mspace{2mu}\} \subset N,
	\end{equation*}
and we obtain (b).

(b) $\Rightarrow$ (a):
Let $T > 0$ and $N$ be a neighborhood of $\boldsymbol{0}$ in $H$.
By assumption, there is $\delta > 0$ such that
	\begin{equation*}
		\{\mspace{2mu} \phi \in C^1(I^r, E) : \|\phi\|_{C^1} \le \delta \mspace{2mu}\} \subset N.
	\end{equation*}
For every $\tau \in [0, T]$
and every $C^1$-prolongation $\gamma \colon [0, \tau] + I^r \to E$ of $\boldsymbol{0}$ with $0$-derivative,
	\begin{equation*}
		I^r_\tau\gamma \in C^1(I^r, E), \mspace{15mu} \|I^r_\tau\gamma\|_{C^1} \le \|\gamma\|_{C^1[0, \tau]}.
	\end{equation*}
This shows $\varLambda_{0, \boldsymbol{0}}^1(T, \delta, 0) \subset [0, T] \times N$.
\end{proof}

\subsubsection{Whole past interval}

\begin{theorem}[\cite{Nishiguchi 2017}]
\label{thm:r = infty, regulation}
Let $H \subset \Map(I^\infty, E)$ be a history space.
We assume that $\boldsymbol{0}$ is closed under prolongations in $H$.
Then the following properties are equivalent:
\begin{enumerate}
\item[\emph{(a)}] $H$ is regulated by prolongations.
\item[\emph{(b)}] For each $R > 0$,
the topology of $H$ is coarser than (or equal to) the topology of uniform convergence on
	\begin{equation*}
		\{\mspace{2mu} \phi \in C_\mathrm{c}(I^\infty, E) : \supp(\phi) \subset [-R, 0] \mspace{2mu} \}.
	\end{equation*}
\end{enumerate}
\end{theorem}

\begin{proof}
(a) $\Rightarrow$ (b):
Fix $R > 0$.
Let $N$ be a neighborhood of $\boldsymbol{0}$ in $H$.
By the assumption, there is $\delta > 0$ such that
	\begin{equation*}
		\varLambda_{0, \boldsymbol{0}}(R, \delta) \subset [0, R] \times N.
	\end{equation*}
For each $\phi \in C_\mathrm{c}(I^\infty, E)$ with $\supp(\phi) \subset [-R, 0]$,
let $\gamma_\phi \colon [0, R] + I^\infty \to E$ be a prolongation of $\boldsymbol{0}$
given by
	\begin{equation*}
		\gamma_\phi(t) = \phi(t - R) \mspace{20mu} (t \in [0, R] + I^\infty).
	\end{equation*}
Then
	\begin{equation*}
		I^\infty_R\gamma_\phi = \phi, \mspace{20mu} \|\gamma_\phi\|_{C[0, R]} = \|\phi\|_{\infty}.
	\end{equation*}
Therefore, $\|\phi\|_\infty \le \delta$ implies
	\begin{equation*}
		(R, \phi)
		= (R, I^\infty_R\gamma_\phi)
		\in \varLambda_{0, \boldsymbol{0}}(R, \delta).
	\end{equation*}
By combining this and $\varLambda_{0, \boldsymbol{0}}(R, \delta) \subset [0, R] \times N$, we have $\phi \in N$.
This shows that (b) holds.

(b) $\Rightarrow$ (a):
Let $T > 0$ and $N$ be a neighborhood of $\boldsymbol{0}$ in $H$.
By applying the property (b) as $R = T$, there is $\delta > 0$ such that
for all $\phi \in C_\mathrm{c}(I^\infty, E)$,
	\begin{equation*}
		\supp(\phi) \subset [-T, 0] \mspace{10mu} \text{and} \mspace{10mu} \|\phi\|_\infty \le \delta \imply \phi \in N.
	\end{equation*}

For every $\tau \in [0, T]$ and every prolongation $\gamma \colon [0, \tau] + I^\infty \to E$ of $\boldsymbol{0}$,
	\begin{equation*}
		I^\infty_\tau\gamma \in C_\mathrm{c}(I^\infty, E), \mspace{15mu}
		\supp(I^\infty_\tau\gamma) \subset [-\tau, 0], \mspace{15mu}
		\|I^\infty_\tau\gamma\|_\infty = \|\gamma\|_{C[0, \tau]}.
	\end{equation*}
This shows $\varLambda_{0, \boldsymbol{0}}(T, \delta) \subset [0, T] \times N$.
\end{proof}

The condition (b) in Theorem~\ref{thm:r = infty, regulation} is equivalent to the following property:
For each $R > 0$, the inclusion
	\begin{equation*}
		i \colon \{\mspace{2mu} \phi \in C_\mathrm{c}(I^\infty, E) : \supp(\phi) \subset [-R, 0] \mspace{2mu} \} \to H
	\end{equation*}
is continuous with respect to the topology of uniform convergence.
This is a one of the hypotheses of phase spaces used by Schumacher~\cite{Schumacher 1978}.

The following is a $C^1$-version of Theorem~\ref{thm:r = infty, regulation}.
We omit the proof because the essentially same argument of that proof is valid
($\varLambda_{0, \boldsymbol{0}}(R, \delta)$ should be replaced with $\varLambda_{0, \boldsymbol{0}}^1(R, \delta, 0)$).
The condition (b) in Theorem~\ref{thm:r = infty, C^1-regulation} is equivalent to the following property:
For each $R > 0$, the inclusion
	\begin{equation*}
		i \colon
		\bigl\{ \mspace{2mu} \phi \in C_\mathrm{c}^1(I^\infty, E) : \supp(\phi) \subset [-R, 0] \mspace{2mu} \bigr\} \to H
	\end{equation*}
is continuous with respect to the topology of uniform $C^1$-convergence.

\begin{theorem}\label{thm:r = infty, C^1-regulation}
Let $H \subset \Map(I^\infty, E)$ be a history space.
We assume that $\boldsymbol{0}$ is closed under $C^1$-prolongations with $0$-derivative in $H$.
Then the following properties are equivalent:
\begin{enumerate}
\item[\emph{(a)}] $H$ is regulated by $C^1$-prolongations.
\item[\emph{(b)}] For each $R > 0$,
the topology of $H$ is coarser than (or equal to) the topology of uniform $C^1$-convergence on
	\begin{equation*}
		\{\mspace{2mu} \phi \in C^1_\mathrm{c}(I^\infty, E) : \supp(\phi) \subset [-R, 0] \mspace{2mu} \}.
	\end{equation*}
\end{enumerate}
\end{theorem}

\subsection{Relationships between neighborhoods and prolongations}

Let $H \subset \Map(I, E)$ be a history space.
In this subsection, we examine relationships between neighborhoods and neighborhoods by prolongations.
The following theorem was obtained in \cite{Nishiguchi 2017}.

\begin{theorem}[\cite{Nishiguchi 2017}]\label{thm:nbd by prolongations}
Suppose that $H$ is prolongable and regulated by prolongations.
Let $W$ be a subset of $\mathbb{R} \times H$ and $(\sigma_0, \psi_0) \in \mathbb{R} \times H$.
Then the following statements hold:
\begin{enumerate}
\item[\emph{(i)}] For every neighborhood $W$ of $(\sigma_0, \psi_0)$ in $[\sigma_0, +\infty) \times H$,
$W$ is a neighborhood by prolongations of $(\sigma_0, \psi_0)$.
\item[\emph{(ii)}] Furthermore, we assume that the semiflow
	\begin{equation*}
		\mathbb{R}_+ \times H \ni (t, \phi) \mapsto S_0(t)\phi \in H
	\end{equation*}
is continuous.
Then every neighborhood $W$ of $(\sigma_0, \psi_0)$ in $\mathbb{R} \times H$
is a uniform neighborhood by prolongations of $(\sigma_0, \psi_0)$.
\end{enumerate}
\end{theorem}

This theorem is extended in the following way
for a history space $H$ which is $C^1$-prolongable and regulated by $C^1$-prolongations.

\begin{theorem}\label{thm:nbd by C^1-prolongations}
Suppose that $H$ is $C^1$-prolongable and regulated by $C^1$-prolongations.
Let $W$ be a subset of $\mathbb{R} \times H$ and $(\sigma_0, \psi_0) \in \mathbb{R} \times H$.
Then the following statements hold:
\begin{enumerate}
\item[\emph{(i)}] For every neighborhood $W$ of $(\sigma_0, \psi_0)$ in $[\sigma_0, +\infty) \times H$,
$W$ is a neighborhood by $C^1$-prolongations of $(\sigma_0, \psi_0)$.
\item[\emph{(ii)}] Furthermore, we assume that
	\begin{equation*}
		\mathbb{R}_+ \times E \times H \ni (t, v, \phi) \mapsto S(t)(v, \phi) \in H
	\end{equation*}
is continuous.
Then every neighborhood $W$ of $(\sigma_0, \psi_0)$ in $\mathbb{R} \times H$
is a uniform neighborhood by $C^1$-prolongations of $(\sigma_0, \psi_0)$.
\end{enumerate}
\end{theorem}

\begin{proof}
(i) Fix $v \in E$ and let $W$ be a neighborhood of $(\sigma_0, \psi_0)$ in $[\sigma_0, +\infty) \times H$.
By the continuity of $\tau_{\sigma_0, \psi_0}^v$ at $(0, \boldsymbol{0})$,
there exist $T > 0$ and a neighborhood $N$ of $\boldsymbol{0}$ in $H$ such that
$\tau_{\sigma_0, \psi_0}^v([0, T] \times N) \subset W$.
Since $H$ is regulated by $C^1$-prolongations, there is $\delta > 0$ such that
	\begin{equation*}
		\varLambda_{0, \boldsymbol{0}}^1(T, \delta, 0) \subset [0, T] \times N.
	\end{equation*}
Therefore, we have
	\begin{equation*}
		\varLambda_{\sigma_0, \psi_0}^1(T, \delta, v)
		= \tau_{\sigma_0, \psi_0}^v(\varLambda_{0, \boldsymbol{0}}^1(T, \delta, 0))
		\subset W
	\end{equation*}
from Lemma~\ref{lem:translation on rectangles}.
This means that $W$ is a neighborhood by $C^1$-prolongations of $(\sigma_0, \psi_0)$.

(ii) Let $W$ be a neighborhood of $(\sigma_0, \psi_0)$ in $\mathbb{R} \times H$.
Fix $v_0 \in E$.
We consider the map
	\begin{equation*}
		\tau(\cdot) \colon E \times (\mathbb{R} \times H) \times (\mathbb{R}_+ \times H) \to \mathbb{R} \times H
	\end{equation*}
defined by
	\begin{equation*}
		\tau \bigl( v, (\sigma, \psi), (t, \phi) \bigr) = \tau_{\sigma, \psi}^v(t, \phi).
	\end{equation*}
Then $\tau(\cdot)$ is continuous by the continuity of
$\mathbb{R}_+ \times E \times H \ni (t, v, \phi) \mapsto S(t)(v, \phi) \in H$.
Since
	\begin{equation*}
		\tau \bigl( v_0, (\sigma_0, \psi_0), (0, \boldsymbol{0}) \bigr) = (\sigma_0, \psi_0),
	\end{equation*}
there are a neighborhood $V_0$ of $v_0$ in $E$, a neighborhood $W_0$ of $(\sigma_0, \psi_0)$ in $\mathbb{R} \times H$,
$T > 0$, and a neighborhood $N$ of $\boldsymbol{0}$ in $H$ such that
	\begin{equation*}
		\tau \bigl( V_0 \times W_0 \times ([0, T] \times N) \bigr) \subset W.
	\end{equation*}
This means that for all $(v, \sigma, \psi) \in V_0 \times W_0$, $\tau_{\sigma, \psi}^v([0, T] \times N) \subset W$ holds.
In the same way as the proof of (i), we have
	\begin{equation*}
		\bigcup_{(\sigma, \psi, v) \in W_0 \times V_0} \varLambda_{\sigma, \psi}^1(T, \delta, v)
		= \bigcup_{(\sigma, \psi, v) \in W_0 \times V_0} \tau_{\sigma, \psi}^v(\varLambda_{0, \boldsymbol{0}}^1(T, \delta, 0))
		\subset W.
	\end{equation*}
This asserts the conclusion.
\end{proof}

\begin{remark}\label{rmk:continuity and equi-continuity}
From Corollaries~\ref{cor:continuity of global semiflows} and \ref{cor:locally equi-continuous C_0-semigroup},
the following property holds:
Under the assumption that $H$ is $C^1$-prolongable,
	\begin{equation*}
		\mathbb{R}_+ \times E \times H \ni (t, v, \phi) \mapsto S(t)(v, \phi) \in H
	\end{equation*}
is continuous if and only if for every $T > 0$,
$(S(t))_{t \in [0, T]}$ is equi-continuous at $(0, \boldsymbol{0})$.
\end{remark}

\section{Properties of Lipschitz conditions}\label{sec:Lip conditions}

Let $I$ be a past interval, $E = (E, \|\cdot\|_E)$ be a Banach space,
$H \subset \Map(I, E)$ be a history space, and $F \colon \mathbb{R} \times H \supset \dom(F) \to E$ be a map.
In this section, we investigate properties of various Lipschitz conditions
introduced in Subsection~\ref{subsec:Lip conditions}.

The next lemma will be used in the following subsections.

\begin{lemma}\label{lem:Lipschitz const for rectangles by prolongations}
Let $W_0 \subset \mathbb{R} \times \Map(I, E)$ be a subset, $B \subset E$ be a bounded set,
and $(I_k)_{k = 1}^\infty$ be a sequence of past intervals contained in $I$.
If
	\begin{equation*}
		\sup_{(\sigma, \psi) \in W_0} \lip(\psi|_{I_k}) < \infty
			\mspace{20mu}
		(\forall k \ge 1),
	\end{equation*}
then there exists a sequence $(M_k)_{k = 1}^\infty$ of positive numbers such that
for all $T > 0$ and all $0 < \delta < 1$,
	\begin{equation*}
		\sup
		\bigl\{ \mspace{2mu} \lip(\phi|_{I_k}) :
			(t, \phi) \in \textstyle\bigcup_{(\sigma, \psi, v) \in W_0 \times B} \varLambda_{\sigma, \psi}^1(T, \delta, v)
		\mspace{2mu} \bigr\} \le M_k
		\mspace{20mu} (\forall k \ge 1)
	\end{equation*}
holds.
\end{lemma}

\begin{proof}
Choose $M' > 0$ so that $B \subset \bar{B}_E(0; M')$.
Let $T, \delta > 0$, $(\sigma, \psi, v) \in W_0 \times B$, and $(t, \phi) \in \varLambda_{\sigma, \psi}^1(T, \delta, v)$.
Then $t = \sigma + \tau$ for some $\tau \in [0, T]$,
and $\phi = I_{\sigma + \tau}\gamma$ for some $\gamma \in \varGamma^1_{\sigma, \psi}(\tau, \delta, v)$.
Since
	\begin{equation*}
		\lip(\phi|_{I_k})
		\le \lip \bigl( \gamma|_{[\sigma, \sigma + \tau] + I_k} \bigr)
		\le \lip(\psi|_{I_k}) + \sup_{u \in [\sigma, \sigma + \tau]} \|\gamma'(u)\|_E
	\end{equation*}
and
	\begin{align*}
		\sup_{u \in [\sigma, \sigma + \tau]} \|\gamma'(u)\|_E
		&\le \sup_{u \in [\sigma, \sigma + \tau]} \|\gamma'(u) - v\|_E + \|v\|_E \\
		&\le \bigl\| \gamma(\sigma + \cdot) - \psi^{\wedge v}(\cdot) \bigr\|_{C^1[0, \tau]} + \|v\|_E,
	\end{align*}
we have
	\begin{equation*}
		\lip(\phi|_{I_k})
		\le \sup_{(\sigma, \psi) \in W_0} \lip(\psi|_{I_k}) + 1 + M'.
	\end{equation*}
Therefore, the conclusion is obtained by choosing $M_k := \sup_{(\sigma, \psi) \in W_0} \lip(\psi|_{I_k}) + 1 + M'$.
\end{proof}

We note that when $I_k \equiv I'$ $(k \ge 1)$ for some $I' \subset I$,
one can choose $M_k \equiv M'$ $(k \ge 1)$ for some $M' > 0$.

\subsection{Lipschitz conditions about prolongations}

\begin{proposition}\label{prop:comparison of Lip about prolongations}
Suppose that $H$ is closed under prolongations.
Let $(\sigma_0, \psi_0) \in \dom(F)$.
If $F$ is locally Lipschitzian about prolongations at $(\sigma_0, \psi_0)$,
then $F$ is locally Lipschitzian about $C^1$-prolongations at $(\sigma_0, \psi_0)$.
\end{proposition}

\begin{proof}
Choose $T, \delta, L > 0$ so that $L$-\eqref{eq:Lipschitzian} holds
for all $(t, \phi_1), (t, \phi_2) \in \varLambda_{\sigma_0, \psi_0}(T, \delta) \cap \dom(F)$.
Applying Lemma~\ref{lem:comparison of rectangles} as $B = \{F(\sigma_0, \psi_0)\}$,
we choose $0 < T_0 \le T$ and $0 < \delta_0 \le \delta/2$ so that
	\begin{equation*}
		\varLambda_{\sigma_0, \psi_0}^1(T_0, \delta_0; F) \subset \varLambda_{\sigma_0, \psi_0}(T_0, \delta).
	\end{equation*}
Therefore, $L$-\eqref{eq:Lipschitzian} holds
for all $(t, \phi_1), (t, \phi_2) \in \varLambda_{\sigma_0, \psi_0}^1(T_0, \delta_0; F) \cap \dom(F)$.
This shows the conclusion.
\end{proof}

The following is a uniform version of Proposition~\ref{prop:comparison of Lip about prolongations}.

\begin{proposition}\label{prop:comparison of uniform Lip about prolongations}
Suppose that $H$ is closed under prolongations.
Let $(\sigma_0, \psi_0) \in \dom(F)$.
If $F$ is uniformly locally Lipschitzian about prolongations at $(\sigma_0, \psi_0)$,
and if $F$ is locally bounded at $(\sigma_0, \psi_0)$,
then $F$ is uniformly locally Lipschitzian about $C^1$-prolongations at $(\sigma_0, \psi_0)$.
\end{proposition}

\begin{proof}
By the local boundedness of $F$ at $(\sigma_0, \psi_0)$,
there is a neighborhood $W$ of $(\sigma_0, \psi_0)$ in $\dom(F)$ and $M > 0$ such that
	\begin{equation*}
		\sup_{(\sigma, \psi) \in W} \|F(\sigma, \psi)\|_E \le M.
	\end{equation*}
Since $F$ is uniformly locally Lipschitzian about prolongations at $(\sigma_0, \psi_0)$,
we choose a neighborhood $W_0$ in $\dom(F)$ so that $L$-\eqref{eq:Lipschitzian} holds
for all $(\sigma, \psi) \in W_0$ and all $(t, \phi_1), (t, \phi_2) \in \varLambda_{\sigma, \psi}(T, \delta) \cap \dom(F)$.
We may assume $W_0 \subset W$ by choosing small $W_0$.
Applying Lemma~\ref{lem:comparison of rectangles} as
$B := \{\mspace{2mu} F(\sigma, \psi) : (\sigma, \psi) \in W_0 \mspace{2mu}\}$,
there are $0 < T_0 \le T$ and $0 < \delta_0 \le \delta/2$ such that
	\begin{equation*}
		\bigcup_{(\sigma, \psi) \in W_0}\varLambda_{\sigma, \psi}^1(T_0, \delta_0; F)
		\subset \bigcup_{(\sigma, \psi) \in W_0} \varLambda_{\sigma, \psi}(T_0, \delta).
	\end{equation*}
Therefore, the conclusion follows.
\end{proof}

We note that the continuity of $F$ at $(\sigma_0, \psi_0)$ is sufficient
for the local boundedness of $F$ at $(\sigma_0, \psi_0)$
in Proposition~\ref{prop:comparison of uniform Lip about prolongations}.

\subsection{Lipschitz conditions about memories}

In this subsection, we investigate relationships
between the Lipschitz conditions about memories and the Lipschitz conditions about prolongations.

\begin{theorem}[\cite{Nishiguchi 2017}]\label{thm:Lip about memories}
Let $(\sigma_0, \psi_0) \in \dom(F)$.
Suppose that $H$ is prolongable and regulated by prolongations.
If $F$ is locally Lipschitzian about memories at $(\sigma_0, \psi_0)$,
then $F$ is locally Lipschitzian about prolongations at $(\sigma_0, \psi_0)$.
Furthermore, if the semiflow
	\begin{equation*}
		\mathbb{R}_+ \times H \ni (t, \phi) \mapsto S_0(t)\phi \in H
	\end{equation*}
is continuous,
then $F$ is uniformly locally Lipschitzian about prolongations at $(\sigma_0, \psi_0)$.
\end{theorem}

We generalize this theorem as follows.

\begin{proposition}\label{prop:loc Lip about Lip-memories}
Let $(\sigma_0, \psi_0) \in \dom(F)$.
Suppose that $H$ is $C^1$-prolongable and regulated by $C^1$-prolongations.
If $F$ is locally Lipschitzian about Lip-memories at $(\sigma_0, \psi_0)$,
and if there exists $R_0 > 0$ such that $[-R_0, 0] \subset I$ and $\lip \bigl( \psi_0|_{[-R_0, 0]}\bigr) < \infty$,
then $F$ is locally Lipschitzian about $C^1$-prolongations at $(\sigma_0, \psi_0)$.
\end{proposition}

\begin{proof}
By applying Lemma~\ref{lem:Lipschitz const for rectangles by prolongations} as
	\begin{equation*}
		W_0 = \{(\sigma_0, \psi_0)\},
			\mspace{10mu}
		B = \{F(\sigma_0, \psi_0)\},
			\mspace{10mu} \text{and} \mspace{10mu}
		I_k \equiv [-R_0, 0],
	\end{equation*}
we choose $M > 0$ so that for all sufficiently small $T, \delta > 0$,
	\begin{equation*}
		\sup \bigl\{ \mspace{2mu} \lip \bigl( \phi|_{[-R_0, 0]} \bigr) :
		(t, \phi) \in \varLambda_{\sigma_0, \psi_0}^1(T, \delta; F)
		\mspace{2mu} \bigr\}
		\le M
	\end{equation*}
holds.
Since $F$ is locally Lipschitzian about Lip-memories at $(\sigma_0, \psi_0)$,
there is a neighborhood $W$ of $(\sigma_0, \psi_0)$ in $\mathbb{R} \times H$ and $R, L > 0$ for this $M > 0$ such that
$R < R_0$ and $L$-\eqref{eq:Lipschitzian} holds for all $(t, \phi_1), (t, \phi_2) \in W \cap \dom(F)$
satisfying the following conditions:
\begin{enumerate}
\item[(i)] $\supp(\phi_1 - \phi_2) \subset [-R, 0] \subset I$, and
\item[(ii)] $\lip \bigl( \phi_1|_{[-R, 0]} \bigr), \lip \bigl( \phi_2|_{[-R, 0]} \bigr) \le M$.
\end{enumerate}
By choosing sufficiently small $T, \delta > 0$,
we may assume $T < R$, and
	\begin{equation*}
		\varLambda_{\sigma_0, \psi_0}^1(T, \delta; F) \subset W
	\end{equation*}
from Theorem~\ref{thm:nbd by C^1-prolongations}.

Let $(t, \phi_1), (t, \phi_2) \in \varLambda_{\sigma_0, \psi_0}^1(T, \delta; F) \cap \dom(F)$.
Then
	\begin{equation*}
		\supp(\phi_1 - \phi_2) \subset [-T, 0] \subset [-R, 0],
	\end{equation*}
and we also have
	\begin{equation*}
		\lip \bigl( \phi_i|_{[-R, 0]} \bigr)
		\le \lip \bigl( \phi_i|_{[-R_0, 0]} \bigr)
		\le M
		\mspace{20mu} (i = 1, 2).
	\end{equation*}
Therefore, the above conditions (i) and (ii) are satisfied.
This implies that $F$ is locally Lipschitzian about $C^1$-prolongations at $(\sigma_0, \psi_0)$.
\end{proof}

The following is a corollary of Proposition~\ref{prop:loc Lip about Lip-memories}.

\begin{corollary}\label{cor:loc Lip about Lip-memories}
Suppose that $H$ is $C^1$-prolongable and regulated by $C^1$-prolongations.
Let
	\begin{equation*}
		H_0
		:= \bigl\{ \mspace{2mu} \phi \in H :
			\text{$\exists R_0 > 0$ s.t.\ $\lip \bigl( \phi|_{[-R_0, 0]} \bigr) < \infty$}
		\mspace{2mu} \bigr\}.
	\end{equation*}
We consider $F_0 \colon \mathbb{R} \times H_0 \supset \dom(F_0) \to E$
which is the restriction of $F$ to $\dom(F_0) := \dom(F) \cap (\mathbb{R} \times H_0)$.
If $F$ is locally Lipschitzian about Lip-memories,
then $F_0$ is locally Lipschitzian about $C^1$-prolongations.
\end{corollary}

By the definition of $H_0$, the following hold:
\begin{itemize}
\item $H_0$ is a linear topological subspace of $H$, where the subspace topology of $H$ is given.
\item $H_0$ is closed under $C^1$-prolongations.
\end{itemize}
Therefore, $H_0$ is also $C^1$-prolongable and regulated by $C^1$-prolongations.

\subsection{Almost local Lipschitz}\label{subsec:almost loc Lip}

When $H \supset C^{0, 1}_\mathrm{loc}(I, E)$,
let $F^{0, 1} \colon \mathbb{R} \times C^{0, 1}_\mathrm{loc}(I, E) \supset \dom \bigl( F^{0, 1} \bigr) \to E$ be the restriction of $F$ to
	\begin{equation*}
		\dom \bigl( F^{0, 1} \bigr) := \dom(F) \cap \left( \mathbb{R} \times C^{0, 1}_\mathrm{loc}(I, E) \right).
	\end{equation*}

\begin{theorem}\label{thm:almost local Lip for I = I^r}
Suppose $I = I^r$ for some $r > 0$,
and $H$ is prolongable and regulated by prolongations.
If $F$ is almost locally Lipschitzian at $(\sigma_0, \psi_0) \in \dom \bigl( F^{0, 1} \bigr)$,
then $F^{0, 1}$ is locally Lipschitzian about $C^1$-prolongations at $(\sigma_0, \psi_0)$.
\end{theorem}

\begin{proof}
Fix $(\sigma_0, \psi_0) \in \dom \bigl( F^{0, 1} \bigr)$.
By applying Lemma~\ref{lem:Lipschitz const for rectangles by prolongations} as
	\begin{equation*}
		W_0 = \{(\sigma_0, \psi_0)\},
			\mspace{10mu}
		B = \{F(\sigma_0, \psi_0)\},
			\mspace{10mu} \text{and} \mspace{10mu}
		I_k \equiv I = I^r,
	\end{equation*}
we choose $M > 0$ so that for all sufficiently small $T, \delta > 0$,
	\begin{equation*}
		\sup \bigl\{ \mspace{2mu} \lip(\phi) :
		(t, \phi) \in \varLambda_{\sigma_0, \psi_0}^1(T, \delta; F) \mspace{2mu} \bigr\}
		\le M
	\end{equation*}
holds.
Since $F$ is almost locally Lipschitzian at $(\sigma_0, \psi_0)$,
there are a neighborhood $W$ of $(\sigma_0, \psi_0)$ in $\mathbb{R} \times H$ and $L > 0$ such that
$L$-\eqref{eq:Lipschitzian} holds for all $(t, \phi_1), (t, \phi_2) \in W \cap \dom(F)$
satisfying $\lip(\phi_i) \le M$ $(i = 1, 2)$.
By choosing sufficiently small $T, \delta > 0$, we may assume
	\begin{equation*}
		\varLambda_{\sigma_0, \psi_0}^1(T, \delta; F) \subset W
	\end{equation*}
from Theorem~\ref{thm:nbd by C^1-prolongations}.
This shows that for all
	\begin{equation*}
		(t, \phi_1), (t, \phi_2)
		\in \varLambda_{\sigma_0, \psi_0}^1(T, \delta; F) \cap \dom(F)
		= \varLambda_{\sigma_0, \psi_0}^1(T, \delta; F) \cap \dom \bigl( F^{0, 1} \bigr),
	\end{equation*}
$L$-\eqref{eq:Lipschitzian} holds.
\end{proof}

When $H = C(I^r, E)_\mathrm{u}$, it is unfortunate that the restricted map
	\begin{equation*}
		F_0 \colon \mathbb{R} \times \bigl( C^{0, 1}(I^r, E), \|\cdot\|_\infty \bigr) \supset \dom(F_0) \to E
	\end{equation*}
is not necessarily uniformly locally Lipschitzian about $C^1$-prolongations
under the assumption of Theorem~\ref{thm:almost local Lip for I = I^r}.
The reason is that the Lipschitz constant is unbounded
in a neighborhood with respect to the topology of uniform convergence,
and therefore, the assumption $\sup_{(\sigma, \psi) \in W_0} \lip(\psi) < \infty$
in Lemma~\ref{lem:Lipschitz const for rectangles by prolongations} is not satisfied.

\begin{theorem}\label{thm:almost local Lip for I = I^infty}
Suppose $I = I^\infty$, $H \supset C^{0, 1}_\mathrm{loc}(I^\infty, E)$,
and $H$ is prolongable and regulated by prolongations.
If $F$ is almost locally Lipschitzian at $(\sigma_0, \psi_0) \in \dom \bigl( F^{0, 1} \bigr)$,
then $F^{0, 1}$ is locally Lipschitzian about $C^1$-prolongations at $(\sigma_0, \psi_0)$.
\end{theorem}

\begin{proof}
By applying Lemma~\ref{lem:Lipschitz const for rectangles by prolongations} as
	\begin{equation*}
		W_0 = \{(\sigma_0, \psi_0)\},
			\mspace{10mu}
		B = \{F(\sigma_0, \psi_0)\},
			\mspace{10mu} \text{and} \mspace{10mu}
		I_k = [-k, 0],
	\end{equation*}
we choose a sequence $(M_k)_{k = 1}^\infty$ of positive numbers so that for all sufficiently small $T, \delta > 0$,
	\begin{equation*}
		\sup \bigl\{ \mspace{2mu} \lip \bigl( \phi|_{[-k, 0]} \bigr) :
		(t, \phi) \in \varLambda_{\sigma_0, \psi_0}^1(T, \delta; F) \mspace{2mu} \bigr\}
		\le M_k \mspace{20mu} (\forall k \ge 1)
	\end{equation*}
holds.
Since $F$ is almost locally Lipschitzian at $(\sigma_0, \psi_0) \in \dom \bigl( F^{0, 1} \bigr)$,
there are a neighborhood $W$ of $(\sigma_0, \psi_0)$ in $\mathbb{R} \times H$ and $L > 0$ such that
$L$-\eqref{eq:Lipschitzian} holds for all $(t, \phi_1), (t, \phi_2) \in W \cap \dom(F)$
satisfying the conditions
\begin{enumerate}
\item[(i)] $\supp(\phi_1 - \phi_2)$ is compact, and
\item[(ii)] $\lip \bigl( \phi_1|_{[-k, 0]} \bigr), \lip \bigl( \phi_2|_{[-k, 0]} \bigr) \le M_k$ for all $k \ge 1$.
\end{enumerate}
By choosing sufficiently small $T, \delta > 0$, we may assume
	\begin{equation*}
		\varLambda_{\sigma_0, \psi_0}^1(T, \delta; F) \subset W
	\end{equation*}
from Theorem~\ref{thm:nbd by C^1-prolongations}.
This shows that for all
	\begin{equation*}
		(t, \phi_1), (t, \phi_2)
		\in \varLambda_{\sigma_0, \psi_0}^1(T, \delta; F) \cap \dom(F)
		= \varLambda_{\sigma_0, \psi_0}^1(T, \delta; F) \cap \dom \bigl( F^{0, 1} \bigr),
	\end{equation*}
$L$-\eqref{eq:Lipschitzian} holds.
\end{proof}

\begin{remark}
In general, the property of almost local Lipschitz does not imply the uniformly local Lipschitz.
The reason is that for a neighborhood $W_0$ of $(\sigma_0, \psi_0)$ in $\dom(F)$,
	\begin{equation*}
		\sup_{(\sigma, \psi) \in W_0} \lip(\psi|_{I_k}) < \infty
	\end{equation*}
does not hold, and therefore, Lemma~\ref{lem:Lipschitz const for rectangles by prolongations} cannot be applied.
\end{remark}

Summarizing the above theorems, we arrive the following conclusion:
If $F$ is almost locally Lipschitzian on $\dom \bigl( F^{0, 1} \bigr)$,
then $F^{0, 1}$ is locally Lipschitzian about $C^1$-prolongations.

\subsection{Examples for infinite retardation}

There is an advantage of the notions of Lipschitz about memories in the property that
we only need to choose histories which have same tail.
The efficiency derived from this property can be seen
in Propositions~\ref{prop:metric for cpt-open top} and \ref{prop:gauge for cpt-open top}.
See also Subsection~\ref{subsec:constancy about memories} for this efficiency.

\begin{proposition}\label{prop:metric for cpt-open top}
Let $I = I^\infty$ and $H = C(I^\infty, E)_\mathrm{co}$.
We consider a metric $\rho$ on $C(I^\infty, E)$ defined by
	\begin{equation*}
		\rho(\phi, \psi)
		= \sum_{k = 1}^\infty \frac{1}{2^k} \cdot \frac{\|\phi - \psi\|_{C[-k, 0]}}{1 + \|\phi - \psi\|_{C[-k, 0]}}.
	\end{equation*}
If $F$ is Lipschitzian with respect to $\rho$, i.e, there exists $L > 0$ such that
	\begin{equation*}
		\|F(t, \phi_1) - F(t, \phi_2)\|_E \le L \cdot \rho(\phi_1, \phi_2)
	\end{equation*}
holds for all $(t, \phi_1), (t, \phi_2) \in \dom(F)$,
then $F$ is Lipschitzian about memories.
\end{proposition}

\begin{proof}
Let $f \colon \mathbb{R}_+ \to \mathbb{R}$ be a function given by $f(t) = t/(1 + t)$.
Then $f$ is monotonically increasing, and $f(t) \le t$ for all $t \ge 0$.
Let $k_0$ be a positive integer.
For every $\phi_1, \phi_2 \in C(I^\infty, E)$, $\supp(\phi_1 - \phi_2) \subset [-k_0, 0]$ implies
	\begin{equation*}
		\|\phi_1 - \phi_2\|_{C[-k, 0]} \le \|\phi_1 - \phi_2\|_{C[-k_0, 0]} = \|\phi_1 - \phi_2\|_\infty < \infty
	\end{equation*}
holds for all $k \ge 1$.
Therefore, we have
	\begin{equation*}
		\rho(\phi_1, \phi_2)
		= \sum_{k = 1}^\infty \frac{1}{2^k} f(\|\phi_1 - \phi_2\|_{C[-k, 0]})
		\le \|\phi_1 - \phi_2\|_\infty \sum_{k = 1}^\infty \frac{1}{2^k}
		= \|\phi_1 - \phi_2\|_\infty.
	\end{equation*}
This shows that $L$-\eqref{eq:Lipschitzian} holds
for all $(t, \phi_1), (t, \phi_2) \in \dom(F)$ satisfying $\supp(\phi_1 - \phi_2) \subset [-k_0, 0]$.
\end{proof}

See \cite[Section 4]{Hale 1965} for the treatment of RFDEs with infinite retardation
by a metric which generates the topology of uniform convergence on compact sets.
See also \cite[Section 3]{Kolmanovskii--Nosov 1986}.

\begin{proposition}\label{prop:gauge for cpt-open top}
Let $I = I^\infty$, $H = C(I^\infty, E)_\mathrm{co}$,
and $\{\rho_k\}_{k = 1}^\infty$ be a gauge (i.e., a set of pseudo-metrics) on $C(I^\infty, E)$ given by
	\begin{equation*}
		\rho_k(\phi, \psi) := \|\phi - \psi\|_{C[-k, 0]} \mspace{20mu} (k \ge 1).
	\end{equation*}
If $F$ is Lipschitzian with respect to the gauge $\{\rho_k\}_{k = 1}^\infty$,
i.e, there exists a sequence $\{L_k\}_{k = 1}^\infty$ of positive numbers such that
	\begin{equation*}
		\|F(t, \phi_1) - F(t, \phi_2)\|_E \le L_k \cdot \rho_k(\phi_1, \phi_2)
	\end{equation*}
holds for all $(t, \phi_1), (t, \phi_2) \in \dom(F)$ and all $k \ge 1$,
then $F$ is Lipschitzian about memories.
\end{proposition}

\begin{proof}
Let $k_0$ be a positive integer.
Then for every $\phi_1, \phi_2 \in C(I^\infty, E)$, $\supp(\phi_1 - \phi_2) \subset [-k_0, 0]$ implies
$\rho_{k_0}(\phi_1, \phi_2) = \|\phi_1 - \phi_2\|_\infty$.
This shows that $L_{k_0}$-\eqref{eq:Lipschitzian} holds
for all $(t, \phi_1), (t, \phi_2) \in \dom(F)$ satisfying $\supp(\phi_1 - \phi_2) \subset [-k_0, 0]$.
\end{proof}

\section{Existence and uniqueness}\label{sec:existence and uniqueness}

Let $I$ be a past interval, $E = (E, \|\cdot\|_E)$ be a Banach space,
$H \subset \Map(I, E)$ be a history space, and $F \colon \mathbb{R} \times H \supset \dom(F) \to E$ be a map
throughout this section.

The purpose of this section is to extend the results about the local existence and uniqueness
obtained in \cite{Nishiguchi 2017} under the following assumptions:
\begin{itemize}
\item $H$ is prolongable and regulated by prolongations.
\item $F$ is continuous and locally Lipschitzian about prolongations.
\item $\dom(F)$ is open in $\mathbb{R} \times H$.
\end{itemize}

\subsection{Approaches for extensions of previous results}

For the purpose stated above, two approaches of extensions will be investigated:
\begin{description}
\item[(E1)] Suppose that
	\begin{enumerate}
	\item[(i)] $H$ is $C^1$-prolongable and regulated by $C^1$-prolongations,
	\item[(ii)] $F$ is continuous, and
	\item[(iii)] $\dom(F)$ is a neighborhood by $C^1$-prolongations of every $(t_0, \phi_0) \in \dom(F)$.
	\end{enumerate}
\item[(E2)] Suppose that
	\begin{enumerate}
	\item[(i)] $H$ is closed under $C^1$-prolongations, and
	\item[(ii)] there exists an extension $\bigl( \bar{H}, \bar{F} \bigr)$ of $(H, F)$ with the following properties:
		\begin{itemize}
		\item $\bar{H}$ is prolongable and regulated by prolongations.
		\item $\bar{F}$ is continuous.
		\item $\dom \bigl( \bar{F} \bigr)$ is a neighborhood by prolongations 
		of every $(t_0, \phi_0) \in \dom \bigl( \bar{F} \bigr)$.
		\end{itemize}
	\end{enumerate}
\end{description}
Here we call $\bigl( \bar{H}, \bar{F} \bigr)$ an \textit{extension} of $(H, F)$ if
\begin{itemize}
\item $\bar{H} \subset \Map(I, E)$ is a history space containing $H$, and
\item $\bar{F} \colon \mathbb{R} \times \bar{H} \supset \dom \bigl( \bar{F} \bigr) \to E$ is an extension of $F$ satisfying
	\begin{equation*}
		\dom(F) = \dom \bigl( \bar{F} \bigr) \cap (\mathbb{R} \times H).
	\end{equation*}
\end{itemize}

\begin{remark}
$H$ is not necessarily a linear topological subspace of $\bar{H}$
in the approach (E2).
For the above approaches (E1) and (E2), we will consider the Lipschitz condition for $F$.
Therefore, (E2) is not included in (E1).
\end{remark}

For an extension $\bigl( \bar{H}, \bar{F} \bigr)$ of $(H, F)$, we consider the family of systems of equations
	\begin{equation}\label{eq:extended IVP}
		\left\{
		\begin{alignedat}{2}
			\dot{x}(t) &= \bar{F}(t, I_tx), & \mspace{20mu} & t \ge t_0, \\
			I_{t_0}x &= \phi_0, & & (t_0, \phi_0) \in \dom \bigl( \bar{F} \bigr)
		\end{alignedat}
		\right. \tag{$\bar{*}$}
	\end{equation}
with a parameter $(t_0, \phi_0)$.
The following lemma combines IVPs~\eqref{eq:IVP} and \eqref{eq:extended IVP}.

\begin{lemma}\label{lem:sol of extended IVP}
Let $(t_0, \phi_0) \in \dom(F)$ and $x \colon J + I \to E$ be a prolongation of $(t_0, \phi_0)$.
Suppose that
(i) $H$ is closed under $C^1$-prolongations, and
(ii) $\bigl( \bar{H}, \bar{F} \bigr)$ is an extension of $(H, F)$.
If $\bar{H}$ is prolongable and $\bar{F}$ is continuous,
then the following properties are equivalent.
\begin{enumerate}
\item[\emph{(a)}] $x$ is a $C^1$-solution of $\eqref{eq:IVP}_{t_0, \phi_0}$.
\item[\emph{(b)}] $x$ is a solution of $\eqref{eq:extended IVP}_{t_0, \phi_0}$.
\end{enumerate}
\end{lemma}

\begin{proof}
(a) $\Rightarrow$ (b):
For a $C^1$-solution $x$ of $\eqref{eq:IVP}_{t_0, \phi_0}$,
$x$ is also a solution of $\eqref{eq:extended IVP}_{t_0, \phi_0}$
because $\bar{F}$ is an extension of $F$.

(b) $\Rightarrow$ (a):
Suppose that $x$ is a solution of $\eqref{eq:extended IVP}_{t_0, \phi_0}$.
Then $x|_J$ is of class $C^1$ by the prolongability of $\bar{H}$ and the continuity of $\bar{F}$.
That is, $x$ is a $C^1$-prolongation of $(t_0, \phi_0)$.
This shows that $(t, I_tx) \in \mathbb{R} \times H$ holds for all $t \in J$ by the assumption (i).
Then we have
	\begin{equation*}
		(t, I_tx) \in \dom(F)
			\mspace{10mu} \text{and} \mspace{10mu}
		(x|_J)'(t) = F(t, I_tx)
	\end{equation*}
for all $t \in J$.
This means that $x$ is a $C^1$-solution of $\eqref{eq:IVP}_{t_0, \phi_0}$.
\end{proof}

We see that the two approaches (E1) and (E2) can be treated in a unified way
by using the following notions.

\begin{definition}[cf.\ \cite{Kappel--Schappacher 1980}]
Suppose that $H$ is closed under $C^1$-prolongations.
We say that $F$ is \textit{continuous along $C^1$-prolongations} at $(\sigma_0, \psi_0) \in \dom(F)$
if for all $C^1$-prolongation $\gamma \colon J + I \to E$ of $(\sigma_0, \psi_0)$,
	\begin{equation*}
		(t, I_t\gamma) \in \dom(F) \mspace{20mu} (\forall t \in J)
	\end{equation*}
implies the continuity of $J \ni t \mapsto F(t, I_t\gamma) \in E$.
When $F$ is continuous along $C^1$-prolongations at every $(\sigma_0, \psi_0) \in \dom(F)$,
we simply say that $F$ is continuous along $C^1$-prolongations.
\end{definition}

\begin{definition}[cf.\ \cite{Kappel--Schappacher 1980}]
Suppose that $H$ is closed under $C^1$-prolongations.
We say that $F$ is \textit{locally bounded about $C^1$-prolongations} at $(\sigma_0, \psi_0) \in \dom(F)$
if there exist $T, \delta > 0$ such that $F$ is bounded on $\varLambda_{\sigma_0, \psi_0}^1(T, \delta; F) \cap \dom(F)$.
When $F$ is locally bounded about $C^1$-prolongations at every $(\sigma_0, \psi_0) \in \dom(F)$,
we simply say that $F$ is locally bounded about $C^1$-prolongations.
\end{definition}

The continuity along prolongations and the local boundedness about prolongations can be introduced
in the similar way.

\begin{remark}
Kappel \& Schappacher~\cite[(i) and (ii) of (A2)]{Kappel--Schappacher 1980} used the following conditions:
Let $(\sigma_0, \psi_0) \in \dom(F)$ and $T, \delta > 0$ be given.
\begin{itemize}
\item For every $\gamma \in \varGamma_{\sigma_0, \psi_0}(T, \delta)$,
the function $[t_0, t_0 + T] \ni t \mapsto F(t, I_t\gamma) \in E$ is integrable.
\item There exists a function $m(\cdot) \colon [t_0, t_0 + T] \to \mathbb{R}_+$ such that
for every $\gamma \in \varGamma_{\sigma_0, \psi_0}(T, \delta)$,
	\begin{equation*}
		\|F(t, I_t\gamma)\|_E \le m(t) \mspace{20mu} (\forall t \in [t_0, t_0 + T])
	\end{equation*}
holds.
\end{itemize}
The former property may be called the \textit{local integrability along prolongations}.
\end{remark}

We introduce the following condition for a unification of (E1) and (E2):
\begin{description}
\item[(E0)] Suppose that
	\begin{enumerate}
	\item[(i)] $H$ is closed under $C^1$-prolongations,
	\item[(ii)] $F$ is continuous along $C^1$-prolongations,
	\item[(iii)] $F$ is locally bounded about $C^1$-prolongations, and 
	\item[(iv)] $\dom(F)$ is a neighborhood by $C^1$-prolongations of every $(t_0, \phi_0)$.
	\end{enumerate}
\end{description}

The succeeding lemmas show that
	\begin{equation*}
		\text{(E1) or (E2) $\Longrightarrow$ (E0)}.
	\end{equation*}

\begin{lemma}\label{lem:(E1) implies (E0)}
Let $(\sigma_0, \psi_0) \in \dom(F)$.
Suppose that $H$ is $C^1$-prolongable and regulated by $C^1$-prolongations.
If $F$ is continuous, and if $\dom(F)$ is a neighborhood by $C^1$-prolongations of $(\sigma_0, \psi_0)$,
then the following properties hold:
\begin{enumerate}
\item $F$ is continuous along $C^1$-prolongations at $(\sigma_0, \psi_0)$.
\item $F$ is locally bounded about $C^1$-prolongations at $(\sigma_0, \psi_0)$.
\end{enumerate}
\end{lemma}

\begin{proof}
1. This is a consequence of the $C^1$-prolongability of $H$ and the continuity of $H$.

2. By the continuity of $F$,
there are a neighborhood $W$ of $(\sigma_0, \psi_0)$ in $\mathbb{R} \times H$ and $M > 0$ such that
	\begin{equation*}
		\sup_{(t, \phi) \in W \cap \dom(F)} \|F(t, \phi)\|_E \le M.
	\end{equation*}
From Theorem~\ref{thm:nbd by C^1-prolongations},
$W$ is a neighborhood by $C^1$-prolongations of $(\sigma_0, \psi_0)$.
Then there are $T, \delta > 0$ such that
	\begin{equation*}
		\varLambda_{\sigma_0, \psi_0}^1(T, \delta; F) \subset W.
	\end{equation*}
Therefore, $F$ is bounded on $\varLambda_{\sigma_0, \psi_0}^1(T, \delta; F) \cap \dom(F)$.
\end{proof}

\begin{lemma}\label{lem:(E2) implies (E0)}
Suppose that
(i) $H$ is closed under $C^1$-prolongations,
(ii) $\bigl( \bar{H}, \bar{F} \bigr)$ is an extension of $(H, F)$, and
(iii) $\bar{H}$ is prolongable and regulated by prolongations.
If $\bar{F}$ is continuous, and if $\dom \bigl( \bar{F} \bigr)$ is a neighborhood by prolongations
of every $(\sigma_0, \psi_0) \in \dom \bigl( \bar{F} \bigr)$,
then the following properties hold:
\begin{enumerate}
\item $F$ is continuous along $C^1$-prolongations.
\item $\dom(F)$ is a neighborhood by $C^1$-prolongations of every $(\sigma_0, \psi_0) \in \dom(F)$.
\item $F$ is locally bounded about $C^1$-prolongations.
\end{enumerate}
\end{lemma}

\begin{proof}
Fix $(\sigma_0, \psi_0) \in \dom(F)$.

1. Let $\gamma \colon J + I \to E$ be a $C^1$-prolongation of $(\sigma_0, \psi_0)$ satisfying
	\begin{equation*}
		(t, I_t\gamma) \in \dom(F) \mspace{20mu} (\forall t \in J).
	\end{equation*}
Then $F(t, I_t\gamma) = \bar{F}(t, I_t\gamma)$ for all $t \in J$.
Therefore, the map
	\begin{equation*}
		J \ni t \mapsto F(t, I_t\gamma) \in E
	\end{equation*}
is continuous by the prolongability of $H$ and by the continuity of $\bar{F}$.

2. There are $T, \delta > 0$ such that
	\begin{equation*}
		\varLambda_{\sigma_0, \psi_0}^1(T, \delta; F) \subset \dom \bigl( \bar{F} \bigr).
	\end{equation*}
By the assumption (i), this shows
	\begin{align*}
		\varLambda_{\sigma_0, \psi_0}^1(T, \delta; F)
		&= \varLambda_{\sigma_0, \psi_0}^1(T, \delta; F) \cap (\mathbb{R} \times H) \\
		&\subset \dom \bigl( \bar{F} \bigr) \cap (\mathbb{R} \times H) \\
		&= \dom(F).
	\end{align*}

3. By the continuity of $\bar{F}$,
there are a neighborhood $W$ of $(\sigma_0, \psi_0)$ in $\mathbb{R} \times \bar{H}$ and $M > 0$ such that
	\begin{equation*}
		\sup_{(t, \phi) \in W \cap \dom(\bar{F})} \bigl\| \bar{F}(t, \phi) \bigr\|_E \le M.
	\end{equation*}
We note that $\bar{H}$ is regulated by $C^1$-prolongations
(see Remark~\ref{rmk:comparison of regulations by prolongations}).
From Theorem~\ref{thm:nbd by C^1-prolongations},
$W$ is a neighborhood by $C^1$-prolongations of $(\sigma_0, \psi_0)$.
Therefore, there are $T, \delta > 0$ such that $\varLambda_{\sigma_0, \psi_0}^1(T, \delta; F) \subset W$.
Then we have
	\begin{equation*}
		\varLambda_{\sigma_0, \psi_0}^1(T, \delta; F) \cap \dom(F)
		\subset \varLambda_{\sigma_0, \psi_0}^1(T, \delta; F) \cap \dom \bigl( \bar{F} \bigr)
		\subset W \cap \dom \bigl( \bar{F} \bigr).
	\end{equation*}
Since $F = \bar{F}$ on $\dom(F)$,
this shows that $F$ is bounded on $\varLambda_{\sigma_0, \psi_0}^1(T, \delta; F) \cap \dom(F)$.
\end{proof}

\subsection{Equivalent integral equation}

To transform $\eqref{eq:IVP}_{t_0, \phi_0}$ an integral equation for each $(t_0, \phi_0) \in \dom(F)$,
we use the following fact for the Riemann integration of Banach space-valued functions.

\begin{fact}
Let $a < b$ be real numbers and $X$ be a Banach space.
For every continuous map $g \colon [a, b] \to X$, the following properties hold:
\begin{enumerate}
\item[\emph{(i)}] $g$ is Riemann integrable.
\item[\emph{(ii)}] The indefinite integral $G \colon [a, b] \to X$ of $g$ is an anti-derivative of $g$.
\item[\emph{(iii)}] Any anti-derivative $G \colon[a, b] \to X$ of $g$ satisfies
	\begin{equation*}
		G(x) - G(a) = \int_a^x g(t) \mspace{2mu} \mathrm{d}t
			\mspace{20mu}
		(\forall x \in [a, b]).
	\end{equation*}
\end{enumerate}
\end{fact}

Here a function $G \colon [a, b] \to X$ is called an anti-derivative of $g$ if $G' = g$ holds.
Both of the above (ii) and (iii) represent the fundamental theorem of calculus
for the Riemann integration of Banach space-valued functions.
The proofs of (ii) and (iii) are standard.
For the proof of (i), we refer the reader to Gordon~\cite{Gordon 1991}.

\begin{lemma}\label{lem:integral eq}
Suppose that $H$ is closed under $C^1$-prolongations.
Let $(t_0, \phi_0) \in \dom(F)$ and $x \colon J + I \to E$ be a $C^1$-prolongation of $(t_0, \phi_0)$.
If $F$ is continuous along $C^1$-prolongations at $(t_0, \phi_0)$,
then the following properties are equivalent:
\begin{enumerate}
\item[\emph{(a)}] $x$ is a solution of $\eqref{eq:IVP}_{t_0, \phi_0}$.
\item[\emph{(b)}] $(t, I_tx) \in \dom(F)$ holds for all $t \in J$, and $x$ satisfies
	\begin{equation*}
		x(t) = \phi_0(0) + \int_{t_0}^t F(u, I_ux) \mspace{2mu} \mathrm{d}u
		\mspace{20mu}
		(\forall t \in J).
	\end{equation*}
\end{enumerate}
\end{lemma}

\begin{proof}
(a) $\Rightarrow$ (b):
The assumption means that
$x$ is an anti-derivative of the continuous map $t \mapsto F(t, I_tx)$ on any compact sub-interval.
Therefore, the integral equation is obtained by the second fundamental theorem of calculus.

(b) $\Rightarrow$ (a):
Since $x$ is a $C^1$-prolongation of $(t_0, \phi_0)$,
the integrand $t \mapsto F(t, I_tx)$ is continuous.
Therefore, (a) follows by the first fundamental theorem of calculus.
\end{proof}

In view of Lemma~\ref{lem:integral eq}, we introduce the following transformation.

\begin{notation}\label{notation:transformations for integral eq}
Let $(\sigma, \psi) \in \dom(F)$.
We define a transformation $\mathcal{T}_{\sigma, \psi}$
for any prolongation $\gamma \colon J_\sigma + I \to E$ of $(\sigma, \psi)$ as follows:
$I_\sigma[\mathcal{T}_{\sigma, \psi}\gamma] = \psi$ and
	\begin{equation*}
		\mathcal{T}_{\sigma, \psi}\gamma(t)
		= \psi(0) + \int_{\sigma}^t F(u, I_u\gamma) \mspace{2mu} \mathrm{d}u
			\mspace{20mu}
		(t \in J_\sigma).
	\end{equation*}
We also define transformations $\mathcal{S}_{\sigma, \psi}^0$ and $\mathcal{S}_{\sigma, \psi}^1$
for any prolongation of $\boldsymbol{0}$ by
	\begin{equation*}
		\mathcal{S}_{\sigma, \psi}^* = N_{\sigma, \psi}^v \circ \mathcal{T}_{\sigma, \psi} \circ A_{\sigma, \psi}^v,
			\mspace{20mu}
		\text{where $v =
		\begin{cases}
			0 & \text{for $* = 0$}, \\
			F(\sigma, \psi) & \text{for $* = 1$}.
		\end{cases}$}
	\end{equation*}
We define the map $F_{\sigma, \psi}^* \colon \mathbb{R} \times H \supset \dom \bigl( F_{\sigma, \psi}^* \bigr) \to E$ by
	\begin{equation*}
		\dom \bigl( F_{\sigma, \psi}^* \bigr) = \bigl( \tau_{\sigma, \psi}^v \bigr)^{-1}(\dom(F)),
			\mspace{15mu}
		F_{\sigma, \psi}^* = F \circ \tau_{\sigma, \psi}^v.
	\end{equation*}
in the same rule for $*$ and $v$.
\end{notation}

\begin{lemma}
Let $(\sigma, \psi) \in \dom(F)$.
Then we have the following expressions of $\mathcal{S}_{\sigma, \psi}^0$ and $\mathcal{S}_{\sigma, \psi}^1$:
For any prolongation $\beta \colon J_0 + I \to E$ of $\boldsymbol{0}$,
$I_0[\mathcal{S}_{\sigma, \psi}^*\beta] = \boldsymbol{0}$ and
	\begin{equation*}
		\mathcal{S}_{\sigma, \psi}^*\beta(s)
		= \int_0^s [F_{\sigma, \psi}^*(u, I_u\beta) - v] \mspace{2mu} \mathrm{d}u
			\mspace{20mu}
		(s \in J_0), \\
	\end{equation*}
where
	\begin{equation*}
		v =
		\begin{cases}
			0 & \text{for $* = 0$}, \\
			F(\sigma, \psi) & \text{for $* = 1$}.
		\end{cases}
	\end{equation*}
\end{lemma}

\begin{proof}
Fix $v \in E$.
Let $\beta \colon J_0 + I \to E$ be a prolongation of $\boldsymbol{0}$, and let
	\begin{equation*}
		\gamma := \mathcal{T}_{\sigma, \psi} \bigl( A_{\sigma, \psi}^v\beta \bigr).
	\end{equation*}
By definition,
	\begin{equation*}
		N_{\sigma, \psi}^v\gamma(s) = \gamma(\sigma + s) - \psi^{\wedge v}(s)
		\mspace{20mu}
		(\forall s \in J_0 + I).
	\end{equation*}
Then, for all $\theta \in I$, we have
	\begin{align*}
		I_0[N_{\sigma, \psi}^v\gamma](\theta)
		&= N_{\sigma, \psi}^v\gamma(\theta) \\
		&= \gamma(\sigma + \theta) - \psi^{\wedge v}(\theta) \\
		&= (I_\sigma\gamma - \psi)(\theta) \\
		&= 0.
	\end{align*}
For all $s \in J_0$, we have
	\begin{align*}
		\gamma(\sigma + s)
		&= \psi(0) + \int_{\sigma}^{\sigma + s} F \bigl( u, I_u[A_{\sigma, \psi}^v\beta] \bigr) \mspace{2mu} \mathrm{d}u \\
		&= \psi(0) + \int_0^s F \bigl( \sigma + u, I_{\sigma + u}[A_{\sigma, \psi}^v\beta] \bigr) \mspace{2mu} \mathrm{d}u,
	\end{align*}
where
	\begin{align*}
		I_{\sigma + u}[A_{\sigma, \psi}^v\beta]
		&= I_{\sigma + u}[\sigma + J_0 \ni t \mapsto \beta(t - \sigma) + \psi^{\wedge v}(t - \sigma)] \\
		&= I_u\beta + I_u\psi^{\wedge v}.
	\end{align*}
Therefore, we have
	\begin{align*}
		N_{\sigma, \psi}^v\gamma(s)
		&= -sv + \int_0^s F \bigl( \sigma + u, I_u\psi^{\wedge v} + I_u\beta \bigr) \mspace{2mu} \mathrm{d}u \\
		&= \int_0^s [F \circ \tau_{\sigma, \psi}^v(u, I_u\beta) - v] \mspace{2mu} \mathrm{d}u.
	\end{align*}
Thus, the expressions are obtained.
\end{proof}

\subsection{Local existence with Lipschitz condition}

\begin{lemma}\label{lem:Lip condition and local boundedness for C^1-prolongations}
Let $(t_0, \phi_0) \in \dom(F)$.
Suppose that
(i) $H$ is closed under $C^1$-prolongations,
(ii) $F$ is continuous along $C^1$-prolongations, and
(iii) $\dom(F)$ is a neighborhood by $C^1$-prolongations of $(t_0, \phi_0)$.
If $F$ is locally Lipschitzian about $C^1$-prolongations at $(t_0, \phi_0)$,
then $F$ is locally bounded about $C^1$-prolongations at $(t_0, \phi_0)$.
\end{lemma}

\begin{proof}
By the assumption (iii) and by the Lipschitz condition,
there are $T, \delta, L > 0$ such that
\begin{itemize}
\item $\varLambda_{t_0, \phi_0}^1(T, \delta; F) \subset \dom(F)$,
\item $L$-\eqref{eq:Lipschitzian} holds for all $(t, \phi_1), (t, \phi_2) \in \varLambda_{t_0, \phi_0}^1(T, \delta; F)$.
\end{itemize}
Let $\beta \in \varLambda_{t_0, \phi_0}^1(T, \delta, 0)$.
Then for all $s \in [0, T]$, we have
	\begin{align*}
		\bigl\| F_{t_0, \phi_0}^1(s, I_s\beta) \bigr\|_E
		&\le \bigl\| F_{t_0, \phi_0}^1(s, I_s\beta)
				- F_{t_0, \phi_0}^1(s, I_s\bar{\boldsymbol{0}}) \bigr\|_E
				+ \bigl\| F_{t_0, \phi_0}^1(s, I_s\bar{\boldsymbol{0}}) \bigr\|_E \\
		&\le L \cdot \|I_s\beta\|_\infty + \bigl\| F_{t_0, \phi_0}^1(s, I_s\bar{\boldsymbol{0}}) \bigr\|_E \\
		&\le L\delta + \sup_{u \in [0, T]} \bigl\| F_{t_0, \phi_0}^1(u, I_u\bar{\boldsymbol{0}}) \bigr\|_E,
	\end{align*}
where
	\begin{equation*}
		\sup_{u \in [0, T]} \bigl\| F_{t_0, \phi_0}^1(u, I_u\bar{\boldsymbol{0}}) \bigr\|_E < \infty
	\end{equation*}
by the assumption (ii).
This shows that $F$ is bounded on $\varLambda_{t_0, \phi_0}^1(T, \delta; F)$.
Therefore, the conclusion holds.
\end{proof}

\begin{lemma}\label{lem:contraction on C^1-prolongation space}
Let $(t_0, \phi_0) \in \dom(F)$.
Suppose that
(i) $H$ is closed under $C^1$-prolongations,
(ii) $F$ is continuous along $C^1$-prolongations, and
(iii) $\dom(F)$ is a neighborhood by $C^1$-prolongations of $(t_0, \phi_0)$.
If $F$ is locally Lipschitzian about $C^1$-prolongations at $(t_0, \phi_0)$,
then there exists $\delta > 0$ such that for all sufficiently small $T > 0$,
	\begin{equation*}
		\mathcal{T}_{t_0, \phi_0} \colon \varGamma_{t_0, \phi_0}^1(T, \delta; F) \to \varGamma_{t_0, \phi_0}^1(T, \delta; F)
	\end{equation*}
is an well-defined contraction with respect to the metrics $\rho^0$ and $\rho^1$.
\end{lemma}

\begin{proof}
From Lemma~\ref{lem:Lip condition and local boundedness for C^1-prolongations},
$F$ is locally bounded about $C^1$-prolongations at $(t_0, \phi_0)$.
Then by combining the assumption (iii) and the Lipschitz condition,
there are $T_1, \delta, M, L > 0$ such that
\begin{itemize}
\item $\varLambda_{t_0, \phi_0}^1(T_1, \delta; F) \subset \dom(F)$,
\item
	$\sup
	\bigl\{ \mspace{2mu}
		\|F(t, \phi)\|_E : (t, \phi) \in \varLambda_{t_0, \phi_0}^1(T_1, \delta; F)
	\mspace{2mu} \bigr\}
	\le M$,
\item $L$-\eqref{eq:Lipschitzian} holds for all $(t, \phi_1), (t, \phi_2) \in \varLambda_{t_0, \phi_0}^1(T_1, \delta; F)$.
\end{itemize}

We show that by choosing sufficiently small $T > 0$ for the above $\delta > 0$,
	\begin{equation*}
		\mathcal{S}_{t_0, \phi_0}^1 \colon
		\varGamma_{0, \boldsymbol{0}}^1(T, \delta, 0) \to \varGamma_{0, \boldsymbol{0}}^1(T, \delta, 0)
	\end{equation*}
becomes an well-defined contraction with respect to $\rho^0$ and $\rho^1$.
Then the conclusion is obtained in view of the following diagram:
	\begin{equation*}
		\xymatrix{
			\varGamma_{t_0, \phi_0}^1(T, \delta; F)
				\ar[r]^{\mathcal{T}_{t_0, \phi_0}}
				\ar[d]_{N_{t_0, \phi_0}^{F(t_0, \phi_0)}}
				\ar@{}[dr]|\circlearrowleft
			& \varGamma_{t_0, \phi_0}^1(T, \delta; F)
				\ar[d]^{N_{t_0, \phi_0}^{F(t_0, \phi_0)}} \\
			\varGamma_{0, \boldsymbol{0}}^1(T, \delta, 0)
				\ar[r]_{\mathcal{S}_{t_0, \phi_0}^1}
			& \varGamma_{0, \boldsymbol{0}}^1(T, \delta, 0)
		}
	\end{equation*}
By the assumption (ii), there is $T_2 > 0$ such that
	\begin{equation*}
		\sup_{u \in [0, T_2]} \bigl\| F_{t_0, \phi_0}^1(u, I_u\bar{\boldsymbol{0}}) - F(t_0, \phi_0) \bigr\|_E
		\le \delta/4.
	\end{equation*}
We choose $T > 0$ so that $T \le \min\{T_1, T_2, \delta/4M, 1/4L\}$.

\begin{flushleft}
\textbf{Step 1. Well-definedness}
\end{flushleft}

Let $\beta \in \varGamma_{0, \boldsymbol{0}}^1(T, \delta, 0)$.
Then for all $s \in [0, T]$, we have
	\begin{align*}
		\bigl\| \mathcal{S}_{t_0, \phi_0}^1\beta(s) \bigr\|_E
		&\le \int_0^s
				\bigl\| F_{t_0, \phi_0}^1(u, I_u\beta) - F(t_0, \phi_0) \bigr\|_E
			\mspace{2mu} \mathrm{d}u \\
		&\le \int_0^T
				\left[ \bigl\| F_{t_0, \phi_0}^1(u, I_u\beta) \bigr\|_E + \|F(t_0, \phi_0)\|_E \right]
			\mspace{2mu} \mathrm{d}u  \\
		&\le 2MT
	\end{align*}
and
	\begin{align*}
		\bigl\| (\mathcal{S}_{t_0, \phi_0}^1\beta)'(s) \bigr\|_E
		&= \bigl\| F_{t_0, \phi_0}^1(s, I_s\beta) - F(t_0, \phi_0) \bigr\|_E \\
		&\le \bigl\| F_{t_0, \phi_0}^1(s, I_s\beta)
				- F_{t_0, \phi_0}^1(s, I_s\bar{\boldsymbol{0}}) \bigr\|_E \\
		&\mspace{60mu} + \bigl\| F_{t_0, \phi_0}^1(s, I_s\bar{\boldsymbol{0}}) - F(t_0, \phi_0) \bigr\|_E \\
		&\le L \cdot \|I_s\beta\|_\infty + (\delta/4) \\
		&\le \{LT + (1/4)\} \delta.
	\end{align*}
Therefore,
	\begin{align*}
		\bigl\| \mathcal{S}_{t_0, \phi_0}^1\beta \bigr\|_{C^1[0, T]}
		&\le 2MT + \{LT + (1/4)\} \delta \\
		&\le (\delta/2) + (\delta/2),
	\end{align*}
which shows
$\mathcal{S}^1_{t_0, \phi_0}\beta \in \varGamma_{0, \boldsymbol{0}}^1(T, \delta, 0)$.

\begin{flushleft}
\textbf{Step 2. Contraction}
\end{flushleft}

Let $\beta_1, \beta_2 \in \varGamma_{0, \boldsymbol{0}}^1(T, \delta, 0)$.
For all $s \in [0, T]$, we have
	\begin{align*}
		\bigl\| \mathcal{S}^1_{t_0, \phi_0}\beta_1(s) - \mathcal{S}^1_{t_0, \phi_0}\beta_2(s) \bigr\|_E
		&\le \int_0^s \bigl\| F_{t_0, \phi_0}^1(u, I_u\beta_1) - F_{t_0, \phi_0}^1(u, I_u\beta_2) \bigr\|_E \mspace{2mu} \mathrm{d}u \\
		&\le L \cdot \int_0^T \|I_u\beta_1 - I_u\beta_2\|_\infty \mspace{2mu} \mathrm{d}u \\
		&\le LT \cdot \sup_{u \in [0, T]} \|\beta_1(u) - \beta_2(u)\|_E
	\end{align*}
and
	\begin{align*}
		\bigl\| (\mathcal{S}^1_{t_0, \phi_0}\beta_1)'(s) - (\mathcal{S}^1_{t_0, \phi_0}\beta_2)'(s) \bigr\|_E
		&\le \bigl\| F_{t_0, \phi_0}^1(s, I_s\beta_1) - F_{t_0, \phi_0}^1(s, I_s\beta_2) \bigr\|_E \\
		&\le L \cdot \|I_s\beta_1 - I_s\beta_2\|_\infty \\
		&\le LT \cdot \sup_{u \in [0, T]} \|\beta_1'(u) - \beta_2'(u)\|_E.
	\end{align*}
Therefore, we have
	\begin{align*}
		&\rho^0 \bigl( \mathcal{S}^1_{t_0, \phi_0}\beta_1, \mathcal{S}^1_{t_0, \phi_0}\beta_2 \bigr)
		\le LT \cdot \rho^0(\beta_1, \beta_2)
		\le \frac{1}{4} \cdot \rho^0(\beta_1, \beta_2), \\
		&\rho^1 \bigl( \mathcal{S}^1_{t_0, \phi_0}\beta_1, \mathcal{S}^1_{t_0, \phi_0}\beta_2 \bigr)
		\le 2LT \cdot \rho^1(\beta_1, \beta_2)
		\le \frac{1}{2} \cdot \rho^1(\beta_1, \beta_2),
	\end{align*}
which show that $\mathcal{S}^1_{t_0, \phi_0}$ is a contraction with respect to $\rho^0$ and $\rho^1$.

This completes the proof.
\end{proof}

\begin{theorem}\label{thm:local existence with Lip}
Let $(t_0, \phi_0) \in \dom(F)$.
Suppose that
(i) $H$ is closed under $C^1$-prolongations,
(ii) $F$ is continuous along $C^1$-prolongations, and
(iii) $\dom(F)$ is a neighborhood by $C^1$-prolongations of $(t_0, \phi_0)$.
If $F$ is locally Lipschitzian about $C^1$-prolongations at $(t_0, \phi_0)$,
then there exist $T > 0$ such that $\eqref{eq:IVP}_{t_0, \phi_0}$ has a $C^1$-solution $x \colon [t_0, t_0 + T] + I \to E$.
\end{theorem}

\begin{proof}
Applying Lemma~\ref{lem:contraction on C^1-prolongation space},
there are $T, \delta > 0$ such that
$\mathcal{T}_{t_0, \phi_0}$ is a contraction on the metric space
	\begin{equation*}
		\left( \varGamma_{t_0, \phi_0}^1(T, \delta; F), \rho^1 \right),
	\end{equation*}
which is complete from Proposition~\ref{prop:completeness of space of C^1-prolongations}.
Therefore, by the Banach fixed point theorem,
$\mathcal{T}_{t_0, \phi_0}$ has a unique fixed point $\gamma_* \in \varGamma_{t_0, \phi_0}^1(T, \delta; F)$.
Lemma~\ref{lem:integral eq} shows that
$\gamma_* \colon [t_0, t_0 + T] + I \to E$ is a solution of $\eqref{eq:IVP}_{t_0, \phi_0}$.
\end{proof}

The following treat the existence theorem under the assumptions (E1) and (E2).
These are corollaries of Theorem~\ref{thm:local existence with Lip}
in view of Lemmas~\ref{lem:(E1) implies (E0)} and \ref{lem:(E2) implies (E0)}.

\begin{corollary}\label{cor:local existence with Lip, (E1)}
Let $(t_0, \phi_0) \in \dom(F)$.
Suppose that
(i) $H$ is $C^1$-prolongable and regulated by $C^1$-prolongations,
(ii) $F$ is continuous, and
(iii) $\dom(F)$ is a neighborhood by $C^1$-prolongations of $(t_0, \phi_0)$.
If $F$ is locally Lipschitzian about $C^1$-prolongations at $(t_0, \phi_0)$,
then there exist $T > 0$ such that $\eqref{eq:IVP}_{t_0, \phi_0}$ has a $C^1$-solution $x \colon [t_0, t_0 + T] + I \to E$.
\end{corollary}

\begin{corollary}\label{cor:local existence with Lip, (E2)}
Let $(t_0, \phi_0) \in \dom(F)$.
Suppose that
(i) $H$ is closed under $C^1$-prolongations,
(ii) $\bigl( \bar{H}, \bar{F} \bigr)$ is an extension of $(H, F)$ with the properties that
	\begin{itemize}
	\item $\bar{H}$ is prolongable and regulated by prolongations,
	\item $\bar{F}$ is continuous, and
	\item $\dom \bigl( \bar{F} \bigr)$ is a neighborhood by prolongations
	of $(t_0, \phi_0)$.
	\end{itemize}
If $F$ is locally Lipschitzian about $C^1$-prolongations at $(t_0, \phi_0)$,
then there exist $T > 0$ such that $\eqref{eq:IVP}_{t_0, \phi_0}$ has a $C^1$-solution $x \colon [t_0, t_0 + T] + I \to E$.
\end{corollary}

The following example shows that
the topology of $H$ is less related to the existence of solutions.

\begin{example}\label{exa:const delay and L^p-topology}
For $r > 0$ and a continuous map $f \colon \mathbb{R} \times E \times E \to E$,
we consider a DDE
	\begin{equation*}
		\dot{x}(t) = f(t, x(t), x(t - r)).
	\end{equation*}
For $1 \le p < \infty$ and $\int_{-r}^0 \|\phi(\theta)\|_E^p \mspace{2mu} \mathrm{d}u < \infty$,
we introduce the notation
	\begin{equation*}
		|\phi|_{L^p \times E}
		:= \left( \int_{-r}^0 \|\phi(\theta)\|_E^p \mspace{2mu} \mathrm{d}\theta + \|\phi(0)\|_E^p \right)^{\frac{1}{p}}.
	\end{equation*}
Let $H := (C([-r, 0], E), |\cdot|_{L^p \times E})$ and $F \colon \mathbb{R} \times H \to E$ be a map given by
	\begin{equation*}
		F(t, \phi) := f(t, \phi(0), \phi(-r)).
	\end{equation*}
Then $F \colon \mathbb{R} \times H \to E$ is not continuous.
In this case, an extension $\bigl( \bar{H}, \bar{F} \bigr)$ of $(H, F)$ can be chosen as
	\begin{equation*}
		\bar{H} = C([-r, 0], E)_\mathrm{u}
			\mspace{10mu} \text{and} \mspace{10mu}
		\bar{F}(t, \phi) := f(t, \phi(0), \phi(-r)).
	\end{equation*}
Then the assumptions of Corollary~\ref{cor:local existence with Lip, (E2)} are satisfied.
Furthermore, if $f$ is locally Lipschitzian, then $F$ is locally Lipschitzian about $C^1$-prolongations.
\end{example}

\subsection{Local uniqueness}

\begin{theorem}\label{thm:local uniqueness for C^1-sols with Lip}
Let $(t_0, \phi_0) \in \dom(F)$.
Suppose that
(i) $H$ is closed under $C^1$-prolongations, and
(ii) $F$ is continuous along $C^1$-prolongations.
If $F$ is locally Lipschitzian about $C^1$-prolongations at $(t_0, \phi_0)$,
then for every $C^1$-solutions $x_i \colon J_i + I \to E$ $(i = 1, 2)$ of $\eqref{eq:IVP}_{t_0, \phi_0}$,
there exists $T > 0$ such that
	\begin{equation*}
		x_1|_{[t_0, t_0 + T]} = x_2|_{[t_0, t_0 + T]}.
	\end{equation*}
\end{theorem}

\begin{proof}
By the Lipschitz condition for $F$, there are $T, \delta, L > 0$ such that $L$-\eqref{eq:Lipschitzian} holds
for all $(t, \phi_1), (t, \phi_2) \in \varLambda_{t_0, \phi_0}^1(T, \delta; F) \cap \dom(F)$.
Since
	\begin{align*}
		\Bigl\| x_i(t_0 + \cdot) - \phi_0^{\wedge F(t_0, \phi_0)}(\cdot) \Bigr\|_{C^1[0, T]}
		&\le \|x_i(t_0 + \cdot) - \phi_0(0)\|_{C[0, T]} + T \cdot \|F(t_0, \phi_0)\|_E \\
		&\mspace{100mu} + \|x_i'(t_0 + \cdot) - F(t_0, \phi_0)\|_{C[0, T]},
	\end{align*}
we may assume
	\begin{equation*}
		x_i|_{[t_0, t_0 + T] + I} \in \varLambda_{t_0, \phi_0}^1(T, \delta; F)
			\mspace{10mu} \text{and} \mspace{10mu}
		LT < 1
	\end{equation*}
by choosing small $T > 0$.
Therefore, for all $t \in [t_0, t_0 + T]$, we have
	\begin{align*}
		\|x_1(t) - x_2(t)\|_E
		&\le \int_{t_0}^t \|F(u, I_ux_1) - F(u, I_ux_2)\|_E \mspace{2mu} \mathrm{d}u \\
		&\le L \cdot \int_{t_0}^{t_0 + T} \|I_ux_1 - I_ux_2\|_\infty \mspace{2mu} \mathrm{d}u \\
		&\le LT \cdot \sup_{u \in [t_0, t_0 + T]} \|x_1(u) - x_2(u)\|_E
	\end{align*}
from Lemma~\ref{lem:integral eq}.
This shows $\sup_{u \in [t_0, t_0 + T]} \|x_1(u) - x_2(u)\|_E = 0$,
and the conclusion holds.
\end{proof}

An alternative proof of the local uniqueness is possible
when $\dom(F)$ is a neighborhood by $C^1$-prolongations of $(t_0, \phi_0)$.

\begin{proposition}\label{prop:local uniqueness for C^1-sols with contraction}
Let $(t_0, \phi_0) \in \dom(F)$.
Suppose that
(i) $H$ is closed under $C^1$-prolongations,
(ii) $F$ is continuous along $C^1$-prolongations, and 
(iii) $\dom(F)$ is a neighborhood by $C^1$-prolongations of $(t_0, \phi_0)$.
If $F$ is locally Lipschitzian about $C^1$-prolongations at $(t_0, \phi_0)$,
then for every $C^1$-solutions $x_i \colon J_i + I \to E$ $(i = 1, 2)$ of $\eqref{eq:IVP}_{t_0, \phi_0}$,
there exists $T > 0$ such that
	\begin{equation*}
		x_1|_{[t_0, t_0 + T]} = x_2|_{[t_0, t_0 + T]}.
	\end{equation*}
\end{proposition}

\begin{proof}
Since
	\begin{align*}
		\Bigl\| x_i(t_0 + \cdot) - \phi_0^{\wedge F(t_0, \phi_0)}(\cdot) \Bigr\|_{C^1[0, T]}
		&\le \|x_i(t_0 + \cdot) - \phi_0(0)\|_{C[0, T]} + T \cdot \|F(t_0, \phi_0)\|_E \\
		&\mspace{100mu} + \|x_i'(t_0 + \cdot) - F(t_0, \phi_0)\|_{C[0, T]},
	\end{align*}
we may assume
	\begin{equation*}
		x_i|_{[t_0, t_0 + T] + I} \in \varLambda_{t_0, \phi_0}^1(T, \delta; F)
	\end{equation*}
by choosing sufficiently small $T > 0$.
Then the conclusion holds because $x_i|_{[t_0, t_0 + T] + I}$ are fixed points
of $\mathcal{T}_{t_0, \phi_0} \colon \varGamma_{t_0, \phi_0}^1(T, \delta; F) \to \varGamma_{t_0, \phi_0}^1(T, \delta; F)$,
which becomes a contraction from Lemma~\ref{lem:contraction on C^1-prolongation space}.
\end{proof}

The following are corollaries of Theorem~\ref{thm:local uniqueness for C^1-sols with Lip}
from Lemmas~\ref{lem:(E1) implies (E0)} and \ref{lem:(E2) implies (E0)}.

\begin{corollary}\label{cor:local uniqueness (E1)}
Let $(t_0, \phi_0) \in \dom(F)$.
Suppose that
(i) $H$ is $C^1$-prolongable and regulated by $C^1$-prolongations, and
(ii) $F$ is continuous.
If $F$ is locally Lipschitzian about $C^1$-prolongations at $(t_0, \phi_0)$,
then for every $C^1$-solutions $x_i \colon J_i + I \to E$ $(i = 1, 2)$ of $\eqref{eq:IVP}_{t_0, \phi_0}$,
there exists $T > 0$ such that
	\begin{equation*}
		x_1|_{[t_0, t_0 + T]} = x_2|_{[t_0, t_0 + T]}.
	\end{equation*}
\end{corollary}

\begin{corollary}\label{cor:local uniqueness (E2)}
Let $(t_0, \phi_0) \in \dom(F)$.
Suppose that
(i) $H$ is closed under $C^1$-prolongations,
(ii) $\left( \bar{H}, \bar{F} \right)$ is an extension of $(H, F)$ with the properties that
	\begin{itemize}
	\item $\bar{H}$ is prolongable and regulated by prolongations, and
	\item $\bar{F}$ is continuous.
	\end{itemize}
If $F$ is locally Lipschitzian about $C^1$-prolongations at $(t_0, \phi_0)$,
then for every $C^1$-solutions $x_i \colon J_i + I \to E$ $(i = 1, 2)$ of $\eqref{eq:IVP}_{t_0, \phi_0}$,
there exists $T > 0$ such that
	\begin{equation*}
		x_1|_{[t_0, t_0 + T]} = x_2|_{[t_0, t_0 + T]}.
	\end{equation*}
\end{corollary}

\subsection{Local existence without Lipschitz condition}

Here we study the local existence without the Lipschitz condition in the case $E = \mathbb{R}^n$ for some $n \ge 1$.
We consider the following conditions about the continuity and boundedness of $F$.

\begin{definition}\label{dfn:PC}
We say that $F$ is \textit{continuous about prolongation} at $(t_0, \phi_0) \in \dom(F)$
if there exist $T_0, \delta > 0$ such that for each fixed $\beta_0 \in \varGamma_{0, \boldsymbol{0}}(T_0, \delta)$,
	\begin{equation}\label{eq:PC}
		\int_0^{T_0}
			\bigl\| F_{t_0, \phi_0}^0(u, I_u\beta) - F_{t_0, \phi_0}^0(u, I_u\beta_0) \bigr\|_E
		\mspace{2mu} \mathrm{d}u
		\to 0
		\tag{PC}
	\end{equation}
as $\beta \to \beta_0$ in $\varGamma_{0, \boldsymbol{0}}(T_0, \delta)$.
We simply say that $F$ is continuous about prolongations
when $F$ is continuous about prolongations at every $(t_0, \phi_0) \in \dom(F)$.
\end{definition}

We write $\eqref{eq:PC}_{t_0, \phi_0}$
when the base point $(t_0, \phi_0)$ is specified.

\begin{remark}
When $F$ is locally Lipschitzian about prolongations at $(t_0, \phi_0)$,
the convergence $\eqref{eq:PC}_{t_0, \phi_0}$ is obvious by the following reason:
there exist $T_0, \delta, L > 0$ such that
for all $\beta_1, \beta_2 \in \varGamma_{0, \boldsymbol{0}}(T_0, \delta)$ and all $u \in [0, T_0]$,
	\begin{align*}
		\bigl\| F_{t_0, \phi_0}^0(u, I_u\beta_1) - F_{t_0, \phi_0}^0(u, I_u\beta_2) \bigr\|_E
		&\le L \cdot \|I_u\beta_1 - I_u\beta_2\|_\infty \\
		&\le L \cdot \rho^0(\beta_1, \beta_2),
	\end{align*}
which shows that
	\begin{equation*}
		\int_0^{T_0}
			\bigl\| F_{t_0, \phi_0}^0(u, I_u\beta) - F_{t_0, \phi_0}^0(u, I_u\beta_0) \bigr\|_E
		\mspace{2mu} \mathrm{d}u
		\le LT_0 \cdot \rho^0(\beta, \beta_0)
	\end{equation*}
holds.
\end{remark}

Definition~\ref{dfn:PC} is motivated by the following proposition.

\begin{proposition}[\cite{Nishiguchi 2017}]\label{prop:continuity about prolongations}
Let $(t_0, \phi_0) \in \dom(F)$.
Suppose that
(i) $H$ is prolongable and regulated by prolongations,
(ii) $F$ is continuous, and
(iii) there exist $T_0, \delta > 0$ such that $\varLambda_{t_0, \phi_0}(T_0, \delta) \subset \dom(F)$.
Then for each fixed $\beta_0 \in \varGamma_{0, \boldsymbol{0}}(T_0, \delta)$,
	\begin{equation*}
		\sup_{u \in [0, T]} \bigl\| F_{t_0, \phi_0}^0(u, I_u\beta) - F_{t_0, \phi_0}^0(u, I_u\beta_0) \bigr\|_E
		\to 0
	\end{equation*}
as $\beta \to \beta_0$ in $\varGamma_{0, \boldsymbol{0}}(T_0, \delta)$.
\end{proposition}

See Appendix~\ref{subsec:proofs in Section 5} for the proof.
See also Proposition~\ref{prop:continuity about prolongations, uniform} for the related result.

\begin{lemma}\label{lem:PC implies LB}
Let $(t_0, \phi_0) \in \dom(F)$.
Suppose that
(i) $H$ is closed under prolongations,
(ii) $F$ is continuous along prolongations, and
(iii) $\dom(F)$ is a neighborhood by prolongations of $(t_0, \phi_0)$.
If $F$ is continuous about prolongation at $(t_0, \phi_0)$,
then there exist $T_0, \delta > 0$ such that for all $s_1, s_2 \in [0, T_0]$,
	\begin{equation*}
		\lim_{|s_1 - s_2| \to 0}
			\int_{s_2}^{s_1} \bigl\| F_{t_0, \phi_0}^0(u, I_u\beta) \bigr\|_E \mspace{2mu} \mathrm{d}u
		= 0
		\tag{LB}
	\end{equation*}
holds uniformly in $\beta \in \varGamma_{0, \boldsymbol{0}}(T_0, \delta)$.
\end{lemma}

\begin{proof}
By the assumption (iii) and by the condition for $F$,
we choose $T_0, \delta_0 > 0$ so that
\begin{itemize}
\item $\varLambda_{t_0, \phi_0}(T_0, \delta_0) \subset \dom(F)$,
\item for each fixed $\beta_0 \in \varGamma_{0, \boldsymbol{0}}(T_0, \delta_0)$,
$\eqref{eq:PC}_{t_0, \phi_0}$ follows
as $\beta \to \beta_0$ in $\varGamma_{0, \boldsymbol{0}}(T_0, \delta_0)$.
\end{itemize}

Let $\ep > 0$ be given.
From $\eqref{eq:PC}_{t_0, \phi_0}$,
there is $0 < \delta < \delta_0$ such that
for all $\beta \in \varGamma_{0, \boldsymbol{0}}(T_0, \delta)$,
	\begin{equation*}
		\int_0^{T_0}
			\bigl\| F_{t_0, \phi_0}^0(u, I_u\beta) - F_{t_0, \phi_0}^0(u, I_u\bar{\boldsymbol{0}})\bigr\|_E
		\mspace{2mu} \mathrm{d}u
		\le \ep/2.
	\end{equation*}
The assumption (ii) indicates that there is $r > 0$ such that
for all $s_1, s_2 \in [0, T_0]$,
	\begin{equation*}
		|s_1 - s_2| < r
			\imply
		\left|
			\int_{s_2}^{s_1} \bigl\| F_{t_0, \phi_0}^0 (u, I_u\bar{\boldsymbol{0}}) \bigr\|_E \mspace{2mu} \mathrm{d}u
		\right|
		\le \ep/2.
	\end{equation*}
By combining these inequalities,
for all $s_1, s_2 \in [0, T_0]$ and all $\beta \in \varGamma_{0, \boldsymbol{0}}(T_0, \delta)$,
$|s_1 - s_2| < r$ implies
	\begin{align*}
		\left| \int_{s_2}^{s_1} \bigl\| F_{t_0, \phi_0}^0(u, I_u\beta) \bigr\|_E \mspace{2mu} \mathrm{d}u \right|
		&\le \int_0^{T_0}
			\bigl\| F_{t_0, \phi_0}^0(u, I_u\beta) - F_{t_0, \phi_0}^0(u, I_u\bar{\boldsymbol{0}})\bigr\|_E
		\mspace{2mu} \mathrm{d}u \\
		&\mspace{40mu} + \left| \int_{s_2}^{s_1}
			\bigl\| F_{t_0, \phi_0}^0 (u, I_u\bar{\boldsymbol{0}}) \bigr\|_E
		\mspace{2mu} \mathrm{d}u \right| \\
		&\le (\ep/2) + (\ep/2) \\
		&= \ep.
	\end{align*}
Therefore, the conclusion is obtained.
\end{proof}

\begin{proposition}[cf.\ \cite{Kappel--Schappacher 1980}]
\label{prop:compactness of transformation on prolongation space}
Let $E = \mathbb{R}^n$ and $(t_0, \phi_0) \in \dom(F)$.
Suppose that
(i) $H$ is closed under prolongations,
(ii) $F$ is continuous along prolongations,
(iii) $F$ is locally bounded about prolongations, and
(iv) $\dom(F)$ is a neighborhood by prolongations of $(t_0, \phi_0)$.
If $F$ is continuous about prolongation at $(t_0, \phi_0)$,
then for all sufficiently small $T > 0$,
	\begin{equation*}
		\mathcal{S}_{t_0, \phi_0}^0
		\colon \varGamma_{0, \boldsymbol{0}}(T, \delta) \to \varGamma_{0, \boldsymbol{0}}(T, \delta)
	\end{equation*}
is an well-defined compact map.
\end{proposition}

We give the proof in Appendix~\ref{subsec:proofs in Section 5}.
See also \cite[Theorem 2.2]{Kappel--Schappacher 1980} for this proof.

To obtain a solution by using Proposition~\ref{prop:compactness of transformation on prolongation space},
the Schauder fixed point theorem is used.

\begin{fact}[Schauder fixed point theorem]\label{fact:Schauder}
Let $C$ be a convex subset of a normed linear space.
Then each compact map from $C$ to $C$ has at least one fixed point.
\end{fact}

We refer the reader to Granas \& Dugundji~\cite{Granas--Dugundji} for fixed point theory.

We give the proofs of the succeeding lemmas and theorem in Appendix~\ref{subsec:proofs in Section 5}
to keep this paper self-contained.

\begin{proposition}[\cite{Nishiguchi 2017}]
\label{prop:compactness of transformation on prolongation space by continuity}
Let $E = \mathbb{R}^n$ and $(t_0, \phi_0) \in \dom(F)$.
Suppose that
(i) $H$ is prolongable and regulated by prolongations,
(ii) $F$ is continuous, and
(iii) $\dom(F)$ is a neighborhood by prolongations of $(t_0, \phi_0)$.
Then there exists $\delta > 0$ such that for all sufficiently small $T > 0$,
	\begin{equation*}
		\mathcal{S}_{t_0, \phi_0}^0
		\colon \varGamma_{0, \boldsymbol{0}}(T, \delta) \to \varGamma_{0, \boldsymbol{0}}(T, \delta)
	\end{equation*}
is an well-defined compact map with respect to $\rho^0$.
\end{proposition}

The following is a version of the Peano existence theorem for RFDEs with history space $H$.

\begin{theorem}[\cite{Nishiguchi 2017}]
\label{thm:Peano thm}
Let $E = \mathbb{R}^n$ and $(t_0, \phi_0) \in \dom(F)$.
Suppose that
(i) $H$ is prolongable and regulated by prolongations,
(ii) $F$ is continuous, and
(iii) $\dom(F)$ is a neighborhood by prolongations of $(t_0, \phi_0)$.
Then there exits $T > 0$ such that $\eqref{eq:IVP}_{t_0, \phi_0}$ has a solution $x \colon [t_0, t_0 + T] + I \to E$.
\end{theorem}

We finally arrive the local existence result without the Lipschitz condition for $F$.

\begin{corollary}\label{cor:local existence with extension}
Let $E = \mathbb{R}^n$ and $(t_0, \phi_0) \in \dom(F)$.
Suppose that
(i) $H$ is closed under $C^1$-prolongations,
(ii) $\bigl( \bar{H}, \bar{F} \bigr)$ is an extension of $(H, F)$,
(iii) $\bar{H}$ is prolongable and regulated by prolongations,
(iv) $\bar{F}$ is continuous, and
(v) $\dom \bigl( \bar{F} \bigr)$ is a neighborhood by prolongations
	of $(t_0, \phi_0)$.
Then there exists $T > 0$ such that $\eqref{eq:IVP}_{t_0, \phi_0}$ has a solution $x \colon [t_0, t_0 + T] + I \to E$.
\end{corollary}

\begin{proof}
Applying Theorem~\ref{thm:Peano thm}, $\eqref{eq:extended IVP}_{t_0, \phi_0}$ has a solution $x$.
From Lemma~\ref{lem:sol of extended IVP}, $x$ is a solution of $\eqref{eq:IVP}_{t_0, \phi_0}$.
\end{proof}

\section{Mechanisms for continuity of solution processes}\label{sec:mechanisms}

The purpose of this section is to find mechanisms for continuity of solution processes
by generalizing the following result obtained in \cite[Theorem B]{Nishiguchi 2017}.

The main results of this section are
Theorems~\ref{thm:maximal WP, C^1} and \ref{thm:maximal WP without uniform Lip, extension}.

\subsection{Maximal well-posedness with uniform Lipschitz condition}

\begin{proposition}\label{prop:continuity about C^1-prolongations}
Let $(t_0, \phi_0) \in \dom(F)$ and $v_0 \in E$.
Suppose that
(i) $H$ is $C^1$-prolongable and regulated by $C^1$-prolongations,
(ii) $F$ is continuous, and
(iii) there exist a neighborhood $W$ of $(t_0, \phi_0)$ in $\dom(F)$,
a neighborhood $V$ of $v_0$ in $E$, and $T_0, \delta > 0$ such that
	\begin{equation*}
		\bigcup_{(\sigma, \psi, v) \in W \times V} \varLambda_{\sigma, \psi}^1(T_0, \delta, v) \subset \dom(F).
	\end{equation*}
If the map
	\begin{equation*}
		\mathbb{R}_+ \times E \times H \ni (t, v, \phi) \mapsto S(t)(v, \phi) \in H
	\end{equation*}
is continuous,
then for every $(\sigma_0, \psi_0) \in W$, every $0 < T \le T_0$,
and every $\beta_0 \in \varGamma_{0, \boldsymbol{0}}^1(T, \delta, 0)$,
we have
	\begin{equation*}
		\sup_{u \in [0, T]}
		\bigl\| F \circ \tau_{\sigma, \psi}^v(u, I_u\beta) - F \circ \tau_{\sigma_0, \psi_0}^{v_0}(u, I_u\beta_0) \bigr\|_E
		\to 0
	\end{equation*}
as $(\sigma, \psi, v) \to (\sigma_0, \psi_0, v_0)$ in $W \times V$ and 
as $\rho^1(\beta, \beta_0) \to 0$ in $\varGamma_{0, \boldsymbol{0}}^1(T, \delta, 0)$.
\end{proposition}

\begin{proof}
Fix $(\sigma_0, \psi_0) \in W$, $0 < T \le T_0$, and $\beta_0 \in \varGamma_{0, \boldsymbol{0}}^1(T, \delta, 0)$.
By the $C^1$-prolongability of $H$,
	\begin{equation*}
		[0, T] \ni u \mapsto \tau_{\sigma_0, \psi_0}^{v_0}(u, I_u\beta_0) \in \mathbb{R} \times H
	\end{equation*}
is continuous.
Therefore,
	\begin{equation*}
		K
		:= \left\{ \mspace{2mu} \tau_{\sigma_0, \psi_0}^{v_0}(u, I_u\beta_0) : u \in [0, T] \mspace{2mu} \right\}
	\end{equation*}
is a compact set of $\mathbb{R} \times H$.

Let $\ep > 0$.
By the continuity of $F$,
there are $a > 0$ and a neighborhood $N$ of $\boldsymbol{0}$ in $H$ such that
for all $(t_1, \phi_1) \in \dom(F)$ and all $(t_2, \phi_2) \in K$,
	\begin{equation*}
		|t_1 - t_2| < a \mspace{10mu} \text{and} \mspace{10mu} \phi_1 - \phi_2 \in N
			\imply
		\|F(t_1, \phi_1) - F(t_2, \phi_2)\|_E \le \ep.
	\end{equation*}
We choose a neighborhood $N'$ of $\boldsymbol{0}$ in $H$ so that $N' + N' \subset N$.
We also choose $r > 0$ so that
	\begin{equation*}
		\varLambda_{0, \boldsymbol{0}}^1(T, r, 0) \subset N'
	\end{equation*}
since $H$ is regulated by $C^1$-prolongations.
For all $(\sigma, \psi, v, \beta) \in \mathbb{R} \times H \times E \times \varGamma_{0, \boldsymbol{0}}^1(T, \delta, 0)$
and all $u \in [0, T]$,
	\begin{align*}
		\tau_{\sigma, \psi}^v(u, I_u\beta) - \tau_{\sigma_0, \psi_0}^{v_0}(u, I_u\beta_0)
		&= \Bigl( \sigma - \sigma_0, I_u[\beta - \beta_0] + I_u[\psi^{\wedge v} - \psi_0^{\wedge v_0}] \Bigr) \\
		&= \bigl( \sigma - \sigma_0, I_u[\beta - \beta_0] + S(u)(v - v_0, \psi - \psi_0) \bigr).
	\end{align*}
Therefore, there are a neighborhood $W'$ of $(\sigma_0, \psi_0)$ in $W$,
a neighborhood $V'$ of $v_0$ in $V$ such that
for all $(\sigma, \psi) \in W'$, all $v \in V'$,
all $\beta \in \varGamma_{0, \boldsymbol{0}}^1(T, \delta, 0)$ satisfying $\rho^1(\beta, \beta_0) \le r$,
and all $u \in [0, T]$,
	\begin{align*}
		\tau_{\sigma, \psi}^v(u, I_u\beta) - \tau_{\sigma_0, \psi_0}^{v_0}(u, I_u\beta_0)
		&\in (-a, a) \times (N' + N') \\
		&\subset (-a, a) \times N.
	\end{align*}
This shows the conclusion.
\end{proof}

\begin{theorem}\label{thm:uniform contraction on C^1-prolongations}
Let $(t_0, \phi_0) \in \dom(F)$.
Suppose that
(i) $H$ is $C^1$-prolongable and regulated by $C^1$-prolongations,
(ii) $F$ is continuous, and
(iii) $\dom(F)$ is a uniform neighborhood by $C^1$-prolongations of $(t_0, \phi_0)$.
If $F$ is uniformly locally Lipschitzian about $C^1$-prolongations at $(t_0, \phi_0)$, and if
	\begin{equation*}
		\mathbb{R}_+ \times E \times H \ni (t, v, \phi) \mapsto S(t)(v, \phi) \in H
	\end{equation*}
is continuous,
then there exist a neighborhood $W$ of $(t_0, \phi_0)$ in $\dom(F)$ and $\delta > 0$ such that
the following statements hold for all sufficiently small $T > 0$:
\begin{enumerate}
\item For every $(\sigma, \psi) \in W$,
	\begin{equation*}
		\mathcal{S}^1_{\sigma, \psi} \colon
		\varGamma_{0, \boldsymbol{0}}^1(T, \delta, 0) \to \varGamma_{0, \boldsymbol{0}}^1(T, \delta, 0)
	\end{equation*}
is well-defined.
\item $\bigl( \mathcal{S}^1_{\sigma, \psi} \bigr)_{(\sigma, \psi) \in W}$ is a uniform contraction
on $\varGamma_{0, \boldsymbol{0}}^1(T, \delta, 0)$ with respect to $\rho^0$ and $\rho^1$.
\item The map
	\begin{equation*}
		W \times \varGamma_{0, \boldsymbol{0}}^1(T, \delta, 0) \ni (\sigma, \psi, \beta) \mapsto
		\mathcal{S}^1_{\sigma, \psi}\beta \in \varGamma_{0, \boldsymbol{0}}^1(T, \delta, 0)
	\end{equation*}
is continuous with respect to $\rho^0$ and $\rho^1$.
\end{enumerate}
\end{theorem}

\begin{proof}
By the continuity of $F$,
there are a neighborhood $W_0$ of $(t_0, \phi_0)$ in $\mathbb{R} \times H$ and $M > 0$ such that
	\begin{equation*}
		\sup_{(t, \phi) \in W_0 \cap \dom(F)} \|F(t, \phi)\|_E \le M.
	\end{equation*}
From Theorem~\ref{thm:nbd by C^1-prolongations},
$W_0$ is a uniform neighborhood by $C^1$-prolongations of $(t_0, \phi_0)$.
Therefore, $W_0 \cap \dom(F)$ is also a uniform neighborhood by $C^1$-prolongations of $(t_0, \phi_0)$.
By combining this property and the Lipschitz condition,
there are a neighborhood $W_1$ of $(t_0, \phi_0)$ in $W_0$ and $T_1, \delta, L > 0$ such that
\begin{itemize}
\item $\bigcup_{(\sigma, \psi) \in W_1} \varLambda_{\sigma, \psi}^1(T_1, \delta; F) \subset W_0 \cap \dom(F)$,
\item $L$-\eqref{eq:Lipschitzian} holds
for all $(\sigma, \psi) \in W_1$ and all $(t, \phi_1), (t, \phi_2) \in \varLambda_{\sigma, \psi}^1(T_1, \delta; F)$.
\end{itemize}
We note that the continuity of $F$ at $(t_0, \phi_0)$ is also used here.

By the assumptions (i) and (ii), $F$ is continuous along $C^1$-prolongations.
By combining this and Proposition~\ref{prop:continuity about C^1-prolongations},
there are a neighborhood $W_2$ of $(t_0, \phi_0)$ in $\dom(F)$ and $T_2 > 0$ with the following properties:
\begin{itemize}
\item $\sup_{u \in [0, T_2]} \bigl\| F_{t_0, \phi_0}^1(u, I_u\bar{\boldsymbol{0}}) - F(t_0, \phi_0) \bigr\|_E \le \delta/8$,
\item For all $(\sigma, \psi) \in W_2$,
	\begin{equation*}
		\sup_{u \in [0, T_2]}
		\bigl\| F_{\sigma, \psi}^1(u, I_u\bar{\boldsymbol{0}}) - F_{t_0, \phi_0}^1(u, I_u\bar{\boldsymbol{0}}) \bigr\|_E
		\le \delta/8.
	\end{equation*}
\end{itemize}
Let
	\begin{equation*}
		0 < T \le \min\{T_1, T_2, \delta/4M, 1/8L\}
			\mspace{10mu} \text{and} \mspace{10mu}
		W := W_1 \cap W_2.
	\end{equation*}

1. Fix $(\sigma, \psi) \in W$.
Let $\beta \in \varGamma_{0, \boldsymbol{0}}^1(T, \delta, 0)$.
We now check $\mathcal{S}_{\sigma, \psi}^1\beta \in \varGamma_{0, \boldsymbol{0}}^1(T, \delta, 0)$.
By the assumptions (i) and (ii),
$\mathcal{S}_{\sigma, \psi}^1\beta \colon [0, T] + I \to E$ is a $C^1$-prolongation of $\boldsymbol{0}$
satisfying $(\mathcal{S}_{\sigma, \psi}^1\beta)'(0) = 0$.
For all $s \in [0, T]$, we have
	\begin{align*}
		\bigl\| \mathcal{S}_{\sigma, \psi}^1\beta(s) \bigr\|_E
		&\le \int_0^s
				\bigl\| F_{\sigma, \psi}^1(u, I_u\beta) - F(\sigma, \psi)] \bigr\|_E
			\mspace{2mu} \mathrm{d}u \\
		&\le \int_0^T
				\left[ \bigl\| F_{\sigma, \psi}^1(u, I_u\beta) \bigr\|_E + \|F(\sigma, \psi)\|_E \right]
			\mspace{2mu} \mathrm{d}u  \\
		&\le 2MT
	\end{align*}
and
	\begin{align*}
		\bigl\| (\mathcal{S}_{\sigma, \psi}^1\beta)'(s) \bigr\|_E
		&= \bigl\| F_{\sigma, \psi}^1(s, I_s\beta) - F(\sigma, \psi) \bigr\|_E \\
		&\le \bigl\| F_{\sigma, \psi}^1(s, I_s\beta) - F_{\sigma, \psi}^1(s, I_s\bar{\boldsymbol{0}}) \bigr\|_E
			+ \bigl\| F_{\sigma, \psi}^1(s, I_s\bar{\boldsymbol{0}}) - F_{t_0, \phi_0}^1(s, I_s\bar{\boldsymbol{0}}) \bigr\|_E \\
			&\mspace{60mu} + \bigl\| F_{t_0, \phi_0}^1(s, I_s\bar{\boldsymbol{0}}) - F(t_0, \phi_0) \bigr\|_E
				+ \|F(t_0, \phi_0) - F(\sigma, \psi)\|_E \\
		&\le L \cdot \|I_s\beta\|_\infty + (\delta/8) + (\delta/8) + (\delta/8) \\
		&\le \{LT + (3/8)\} \delta.
	\end{align*}
Therefore,
	\begin{align*}
		\bigl\| \mathcal{S}_{\sigma, \psi}^1\beta \bigr\|_{C^1[0, T]}
		&\le 2MT + \{LT + (3/8)\} \delta \\
		&\le (\delta/2) + (\delta/2),
	\end{align*}
which shows
$\mathcal{S}_{\sigma, \psi}^1\beta \in \varGamma_{0, \boldsymbol{0}}^1(T, \delta, 0)$.

2. We omit the proof because the proof is obtained by replacing $(t_0, \phi_0)$ to $(\sigma, \psi) \in W$
in Step 2 of the proof of Lemma~\ref{lem:contraction on C^1-prolongation space}.
This is possible by the assumption of the uniformly local Lipschitz about $C^1$-prolongations.

3. From the statement 2 of this theorem and from Lemma~\ref{lem:continuity of uniform contraction},
it is enough to show the continuity of
	\begin{equation*}
		W \ni (\sigma, \psi) \mapsto \mathcal{S}_{\sigma, \psi}^1\beta \in \varGamma_{0, \boldsymbol{0}}^1(T, \delta, 0)
	\end{equation*}
at each $(\sigma_0, \psi_0) \in W$ for each fixed $\beta \in \varGamma_{0, \boldsymbol{0}}^1(T, \delta, 0)$
with respect to $\rho^0$ and $\rho^1$.
For all $(\sigma, \psi) \in W$ and all $s \in [0, T]$, we have
	\begin{align*}
		\bigl\| \mathcal{S}_{\sigma, \psi}^1\beta(s) - \mathcal{S}_{\sigma_0, \psi_0}^1\beta(s) \bigr\|_E
		&\le \int_0^T
				\bigl\| F_{\sigma, \psi}^1(u, I_u\beta) - F_{\sigma_0, \psi_0}^1(u, I_u\beta) \bigr\|_E
			\mspace{2mu} \mathrm{d}u \\
		&\mspace{60mu} + \int_0^T \|F(\sigma, \psi) - F(\sigma_0, \psi_0)\|_E \mspace{2mu} \mathrm{d}u \\
		&\le 2T \cdot \sup_{u \in [0, T]} \bigl\| F_{\sigma, \psi}^1(u, I_u\beta) - F_{\sigma_0, \psi_0}^1(u, I_u\beta) \bigr\|_E
	\end{align*}
and
	\begin{align*}
		\bigl\| (\mathcal{S}_{\sigma, \psi}^1\beta)'(s) - (\mathcal{S}_{\sigma_0, \psi_0}^1\beta)'(s) \bigr\|_E
		&\le \sup_{s \in [0, T]} \bigl\| F_{\sigma, \psi}^1(u, I_u\beta) - F_{\sigma_0, \psi_0}^1(u, I_u\beta) \bigr\|_E \\
		&\mspace{60mu} + \|F(\sigma, \psi) - F(\sigma_0, \psi_0)\|_E \\
		&\le 2 \cdot \sup_{u \in [0, T]} \bigl\| F_{\sigma, \psi}^1(u, I_u\beta) - F_{\sigma_0, \psi_0}^1(u, I_u\beta) \bigr\|_E.
	\end{align*}
Therefore, the convergence
	\begin{align*}
		\rho^0 \bigl( \mathcal{S}_{\sigma, \psi}^1\beta, \mathcal{S}_{\sigma_0, \psi_0}^1\beta \bigr)
		&\le \rho^1 \bigl( \mathcal{S}_{\sigma, \psi}^1\beta, \mathcal{S}_{\sigma_0, \psi_0}^1\beta \bigr) \\
		&= 2(T + 1) \cdot
			\sup_{u \in [0, T]} \bigl\| F_{\sigma, \psi}^1(u, I_u\beta) - F_{\sigma_0, \psi_0}^1(u, I_u\beta) \bigr\|_E \\
		&\to 0
	\end{align*}
as $(\sigma, \psi) \to (\sigma_0, \psi_0)$ follows
from Proposition~\ref{prop:continuity about C^1-prolongations}.

This completes the proof.
\end{proof}

Theorem~\ref{thm:maximal WP, C^1} stated in Introduction is the result about the maximal well-posedness.

\begin{proof}[\textbf{Proof of Theorem~\ref{thm:maximal WP, C^1}}]
The local existence and local uniqueness for $C^1$-solutions of IVP~\eqref{eq:IVP}
follows by Corollaries~\ref{cor:local existence with Lip, (E1)} and \ref{cor:local uniqueness (E1)}.
From Proposition~\ref{prop:maximal C^1-solutions and solution process}, the following statements hold :
\begin{itemize}
\item For each $(t_0, \phi_0) \in \dom(F)$, $\eqref{eq:IVP}_{t_0, \phi_0}$ has the unique maximal $C^1$-solution
	\begin{equation*}
		x_F(\cdot; t_0, \phi_0) \colon [t_0, t_0 + T_F(t_0, \phi_0)) + I \to E,
	\end{equation*}
where $0 < T_F(t_0, \phi_0) \le \infty$.
\item The solution process $\mathcal{P}_F$ defined by \eqref{eq:sol process}
given in Subsection~\ref{subsec:maximal well-posedness} is a maximal process in $\dom(F)$.
\end{itemize}

We now show that $\mathcal{P}_F$ is a continuous maximal process in $\dom(F)$.
For this purpose, we use Corollary~\ref{cor:continuity of maximal processes}
and Theorems~\ref{thm:equi-continuity from continuity} and \ref{thm:continuity from local equi-continuity}.

\begin{flushleft}
\textbf{Step 1. Continuity of orbits}
\end{flushleft}

This follows by the $C^1$-prolongability of $H$.

\begin{flushleft}
\textbf{Step 2. Lower semi-continuity of escape time function}
\end{flushleft}

Fix $(t_0, \phi_0) \in \dom(F)$.
Applying Theorem~\ref{thm:uniform contraction on C^1-prolongations},
we choose a neighborhood $W$ of $(t_0, \phi_0)$ in $\dom(F)$ and $T, \delta > 0$ so that
$\bigl( \mathcal{S}^1_{\sigma, \psi} \bigr)_{(\sigma, \psi) \in W}$ is an well-defined uniform contraction
on $\varGamma_{0, \boldsymbol{0}}^1(T, \delta, 0)$.
Therefore, $\mathcal{S}^1_{\sigma, \psi}$ has the unique fixed point
	\begin{equation*}
		\eta(\cdot; \sigma, \psi) \in \varGamma_{0, \boldsymbol{0}}^1(T, \delta, 0)
	\end{equation*}
for each $(\sigma, \psi) \in W$.
Let $(\sigma, \psi) \in W$.
Then
	\begin{equation*}
		\xi(\cdot; \sigma, \psi)
		:= A_{\sigma, \psi}^{F(\sigma, \psi)}[\eta(\cdot; \sigma, \psi)]
		\in \varGamma_{\sigma, \psi}^1(T, \delta; F)
	\end{equation*}
is the unique fixed point of
	$\mathcal{T}_{\sigma, \psi} \colon
	\varGamma_{\sigma, \psi}^1(T, \delta; F) \to \varGamma_{\sigma, \psi}^1(T, \delta; F)$,
i.e., a $C^1$-solution of $\eqref{eq:IVP}_{\sigma, \psi}$ from Lemma~\ref{lem:integral eq}.
By the maximality of $x_F(\cdot; \sigma, \psi)$, we have $T < T_F(\sigma, \psi)$.
Since this holds for every $(\sigma, \psi) \in W$,
	\begin{equation*}
		[0, T] \times W \subset \dom(\mathcal{P}_F)
	\end{equation*}
is derived.

\begin{flushleft}
\textbf{Step 3. Equi-continuity}
\end{flushleft}

Fix $(t_0, \phi_0) \in \dom(F)$.
We choose $W, T, \delta$ in Step 2.
For every $(\tau, \sigma, \psi) \in [0, T] \times W$, we have
	\begin{align*}
		\mathcal{P}_F(\tau, \sigma, \psi)
		&= I_{\sigma + \tau}[x_F(\cdot; \sigma, \psi)] \\
		&= I_{\sigma + \tau}[\xi(\cdot; \sigma, \psi)] \\
		&= I_\tau[\eta(\cdot; \sigma, \psi)] + I_\tau\psi^{\wedge F(\sigma, \psi)}.
	\end{align*}
This shows that for each fixed $(\sigma_0, \psi_0) \in W$, we have
	\begin{align*}
		&\mathcal{P}_F(\tau, \sigma, \psi) - \mathcal{P}_F(\tau, \sigma_0, \psi_0) \\
		&= I_\tau[\eta(\cdot; \sigma, \psi) - \eta(\cdot; \sigma_0, \psi_0)]
			+ S(\tau) \bigl( F(\sigma, \psi) - F(\sigma_0, \psi_0), \psi - \psi_0 \bigr).
	\end{align*}
Then the equi-continuity of $(\mathcal{P}_F(\tau, \cdot)|_W)_{\tau \in [0, T]}$ follows by the following reasons:
\begin{itemize}
\item Applying Theorem~\ref{thm:uniform contraction thm} (uniform contraction theorem)
in view of Theorem~\ref{thm:uniform contraction on C^1-prolongations},
we have the convergence
	\begin{equation*}
		\rho^1 \bigl( \eta(\cdot; \sigma, \psi), \eta(\cdot; \sigma_0, \psi_0) \bigr) \to 0
		\mspace{20mu}
		\text{as $(\sigma, \psi) \to (\sigma_0, \psi_0)$}.
	\end{equation*}
This implies the uniform convergence of the first term since $H$ is regulated by $C^1$-prolongations.
\item The uniform convergence of the second term follows by the continuity of $F$ and by the continuity of
$\mathbb{R}_+ \times E \times H \ni (t, v, \phi) \mapsto S(t)(v, \phi) \in H$.
\end{itemize}
This completes the proof.
\end{proof}

\subsection{Maximal well-posedness without uniform Lipschitz condition}

\begin{proposition}\label{prop:continuity about prolongations, uniform}
Let $(t_0, \phi_0) \in \dom(F)$.
Suppose that
(i) $H$ is prolongable and regulated by prolongations,
(ii) $F$ is continuous, and
(iii) there exist a neighborhood $W$ of $(t_0, \phi_0)$ in $\dom(F)$ and $T_0, \delta > 0$ such that
	\begin{equation*}
		\bigcup_{(\sigma, \psi) \in W} \varLambda_{\sigma, \psi}(T_0, \delta) \subset \dom(F).
	\end{equation*}
If the semifow
	\begin{equation*}
		\mathbb{R}_+ \times H \ni (t, \phi) \mapsto S_0(t)\phi \in H
	\end{equation*}
is continuous,
then for every $(\sigma_0, \psi_0) \in W$, every $0 < T \le T_0$,
and every $\beta_0 \in \varGamma_{0, \boldsymbol{0}}(T, \delta)$,
we have
	\begin{equation*}
		\sup_{u \in [0, T]}
		\bigl\| F_{\sigma, \psi}^0(u, I_u\beta) - F_{\sigma_0, \psi_0}^0(u, I_u\beta_0) \bigr\|_E
		\to 0
	\end{equation*}
as $(\sigma, \psi, \beta) \to (\sigma_0, \psi_0, \beta_0)$ in $W \times \varGamma_{0, \boldsymbol{0}}(T, \delta)$. 
\end{proposition}

See Propositions~\ref{prop:continuity about prolongations} and \ref{prop:continuity about C^1-prolongations}
for the related results.
We give the proof in Appendix~\ref{subsec:proofs in Section 6}
because the proof is similar as that of Proposition~\ref{prop:continuity about C^1-prolongations}.

\begin{proposition}\label{prop:compactness of transformation on prolongation space, uniform}
Let $E = \mathbb{R}^n$ and $(t_0, \phi_0) \in \dom(F)$.
Suppose that
(i) $H$ is prolongable and regulated by prolongations,
(ii) $F$ is continuous, and
(iii) $\dom(F)$ is a uniform neighborhood by prolongations of $(t_0, \phi_0)$.
If the semifow
	\begin{equation*}
		\mathbb{R}_+ \times H \ni (t, \phi) \mapsto S_0(t)\phi \in H
	\end{equation*}
is continuous, then there exists a neighborhood $W$ of $(t_0, \phi_0)$ in $\dom(F)$ and $\delta > 0$ such that
for all sufficiently small $T > 0$,
	\begin{equation*}
		W \times \varGamma_{0, \boldsymbol{0}}(T, \delta) \ni (\sigma, \psi, \beta)
			\mapsto
		\mathcal{S}_{\sigma, \psi}^0\beta \in \varGamma_{0, \boldsymbol{0}}(T, \delta)
	\end{equation*}
is an well-defined compact map.
\end{proposition}

The proof is similar as that of Proposition~\ref{prop:compactness of transformation on prolongation space}.
Therefore, we give the proof in Appendix~\ref{subsec:proofs in Section 6}.

\begin{theorem}\label{thm:continuity and uniqueness}
Let $E = \mathbb{R}^n$ and $(t_0, \phi_0) \in \dom(F)$.
Suppose that
(i) $H$ is prolongable and regulated by prolongations,
(ii) $F$ is continuous, and
(iii) $\dom(F)$ is a uniform neighborhood by prolongations of $(t_0, \phi_0)$.
If $\eqref{eq:IVP}_{t_0, \phi_0}$ has a unique solution, and if the semiflow
	\begin{equation*}
		\mathbb{R}_+ \times H \ni (t, \phi) \mapsto S_0(t)\phi \in H
	\end{equation*}
is continuous,
then there exist a neighborhood $W$ of $(t_0, \phi_0)$ in $\dom(F)$ and $T, \delta > 0$ such that
the following statements hold:
\begin{enumerate}
\item For each $(\sigma, \psi) \in W$, $\eqref{eq:IVP}_{\sigma, \psi}$ has a solution
	\begin{equation*}
		\xi(\cdot; \sigma, \psi) \in \varGamma_{\sigma, \psi}(T, \delta) \subset \dom(F).
	\end{equation*}
\item Let $\eta(\cdot; \sigma, \psi) := N_{\sigma, \psi}^0[\xi(\cdot; \sigma, \psi)]$ for each $(\sigma, \psi) \in W$.
Then we have
	\begin{equation*}
		\rho^0(\eta(\cdot; \sigma, \psi), \eta(\cdot; t_0, \phi_0)) \to 0
	\end{equation*}
as $(\sigma, \psi) \to (t_0, \phi_0)$ in $W$.
\end{enumerate}
\end{theorem}

\begin{proof}
1. From Proposition~\ref{prop:compactness of transformation on prolongation space, uniform},
we choose a neighborhood $W$ of $(t_0, \phi_0)$ in $\dom(F)$ and $T, \delta > 0$ so that
	\begin{equation*}
		W \times \varGamma_{0, \boldsymbol{0}}(T, \delta) \ni (\sigma, \psi, \beta)
			\mapsto
		\mathcal{S}_{\sigma, \psi}^0\beta \in \varGamma_{0, \boldsymbol{0}}(T, \delta)
	\end{equation*}
is an well-defined compact map.
By the Schauder fixed point theorem (see Fact~\ref{fact:Schauder}),
	\begin{equation*}
		\mathcal{S}_{\sigma, \psi}^0
		\colon \varGamma_{0, \boldsymbol{0}}(T, \delta) \to \varGamma_{0, \boldsymbol{0}}(T, \delta)
	\end{equation*}
has a fixed point
	\begin{equation*}
		\eta(\cdot; \sigma, \psi) \in \varGamma_{0, \boldsymbol{0}}(T, \delta)
	\end{equation*}
for each $(\sigma, \psi) \in W$,
which belongs to some compact set $K$ of $\varGamma_{0, \boldsymbol{0}}(T, \delta)$.
Let
	\begin{equation*}
		\xi(\cdot; \sigma, \psi)
		:= A_{\sigma, \psi}^0\eta(\cdot; \sigma, \psi)
		\in \varGamma_{\sigma, \psi}(T, \delta)
	\end{equation*}
for each $(\sigma, \psi) \in W$, which is a solution of $\eqref{eq:IVP}_{\sigma, \psi}$.

2. We choose a net $(\sigma_\alpha, \psi_\alpha)_{\alpha \in A}$ in $W$ which converges to $(t_0, \phi_0)$.
Since $(\eta(\cdot; \sigma_\alpha, \psi_\alpha))_{\alpha \in A}$ is a net in the compact set $K$,
there is a subnet
	\begin{equation*}
		\left( \sigma_{h(\beta)}, \psi_{h(\beta)} \right)_{\beta \in B}
	\end{equation*}
of $(\sigma_\alpha, \psi_\alpha)_{\alpha \in A}$ such that
$\left( \eta \bigl( \cdot; \sigma_{h(\beta)}, \psi_{h(\beta)} \bigr) \right)_{\beta \in B}$ converges to
some $\eta \in K \subset \varGamma_{0, \boldsymbol{0}}(T, \delta)$.
By the point-wise convergence, for all $s \in [0, T]$, we have
	\begin{align*}
		\eta(s)
		&= \lim_{\beta} \eta \bigl( s; \sigma_{h(\beta)}, \psi_{h(\beta)} \bigr) \\
		&= \lim_{\beta}
				\int_0^s
				F_{\sigma_{h(\beta)}, \psi_{h(\beta)}}^0 \bigl( u, I_u\eta(\cdot; \sigma_{h(\beta)}, \psi_{h(\beta)} \bigr) \bigr)
				\mspace{2mu} \mathrm{d}u,
	\end{align*}
which equals to
	\begin{equation*}
		\int_0^s F_{t_0, \phi_0}^0(u, I_u\eta) \mspace{2mu} \mathrm{d}u
	\end{equation*}
from Proposition~\ref{prop:continuity about prolongations, uniform}.
Therefore, $\eta$ is a fixed point of $\mathcal{S}_{t_0, \phi_0}^0$, that is,
	\begin{equation*}
		\xi := A_{t_0, \phi_0}^0\eta
	\end{equation*}
is a solution of $\eqref{eq:IVP}_{t_0, \phi_0}$.
By the uniqueness, we have $\xi(\cdot; t_0, \phi_0) = \xi$,
which is equivalent to
	\begin{equation*}
		\eta = N_{t_0, \phi_0}^0[\xi(\cdot; t_0, \phi_0)] = \eta(\cdot; t_0, \phi_0).
	\end{equation*}

The above argument shows that
every subnet of $(\eta(\cdot; \sigma_\alpha, \psi_\alpha))_{\alpha \in A}$ converges to $\eta(\cdot; t_0, \phi_0)$,
which is independent from the choice of a subnet.
Therefore, $(\eta(\cdot; \sigma_\alpha, \psi_\alpha))_{\alpha \in A}$ converges to $\eta(\cdot; t_0, \phi_0)$.
This shows the conclusion.
\end{proof}

\begin{remark}
Let $(\sigma, \psi) \in W$.
Then for all $s \in [0, T] + I$, we have
	\begin{equation*}
		\xi(\sigma + s; \sigma, \psi) = \eta(s; \sigma, \psi) + \bar{\psi}(s).
	\end{equation*}
Therefore, for all $s \in [0, T]$,
	\begin{align*}
		&\|\xi(\sigma + s; \sigma, \psi) - \xi(t_0 + s; t_0, \phi_0)\|_E \\
		&\mspace{40mu} \le \|\eta(s; \sigma, \psi) - \eta(s; t_0, \phi_0)\|_E + \|\psi(0) - \phi_0(0)\|_E.
	\end{align*}
Thus, under an additional assumption of the continuity of the substitution
	\begin{equation*}
		H \ni \phi \mapsto \phi(0) \in E,
	\end{equation*}
we obtain the convergence
	\begin{align*}
		\sup_{s \in [0, T]} \left\| \xi(\sigma + s; \sigma, \psi) - \xi(t_0 + s; t_0, \phi_0) \right\|_E \to 0
	\end{align*}
as $(\sigma, \psi) \to (t_0, \phi_0)$.
\end{remark}

The continuity of $H \ni \phi \mapsto \phi(0) \in E$ was assumed in \cite{Hale--J.Kato 1978} and \cite{Kato 1978}.
In general, the continuity of the substitution is unrelated to
the continuity of the semiflow $\mathbb{R}_+ \times H \ni (t, \phi) \mapsto S_0(t)\phi \in H$.

\begin{example}
Let $I = I^r$ $(0 < r < \infty)$, and we consider a seminorm $p$ given by
	\begin{equation*}
		p(\phi) = \|\phi(0)\|_E + \|\phi(-r)\|_E
			\mspace{20mu}
		(\phi \in C(I^r, E)).
	\end{equation*}
Let $H = (C(I^r, E), p)$ be a seminormed space.
Then the following properties hold:
\begin{itemize}
\item $H \ni \phi \mapsto \phi(0) \in E$ is continuous.
\item For each $0 < t < r$, the time-$t$ map $S_0(t) \colon H \to H$ is not continuous.
\end{itemize}
\end{example}

See also \cite[Example 3]{Nishiguchi 2017}.

In the following theorems,
it becomes clear that the continuity of the substitution $H \ni \phi \mapsto \phi(0) \in E$ is unnecessary,
but the continuity of the semiflow is important for the maximal well-posedness.

Theorem~\ref{thm:maximal WP without uniform Lip} is included in the following theorem.

\begin{theorem}\label{thm:maximal WP without uniform Lip, extension}
Let $E = \mathbb{R}^n$.
Suppose that
(i) $H$ is closed under $C^1$-prolongations, and
(ii) there exists an extension $\left( \bar{H}, \bar{F} \right)$ of $(H, F)$ with the following properties:
	\begin{itemize}
	\item $\bar{H}$ is prolongable and regulated by prolongations.
	\item $H$ is a topological subspace of $\bar{H}$.
	\item $\bar{F}$ is continuous.
	\item $\dom \bigl( \bar{F} \bigr)$ is a neighborhood by prolongations 
	of every $(t_0, \phi_0) \in \dom \bigl( \bar{F} \bigr)$.
	\end{itemize}
If $F$ is locally Lipschitzian about $C^1$-prolongations, and if
	\begin{equation*}
		\mathbb{R}_+ \times \bar{H} \ni (t, \phi) \mapsto S_0(t)\phi \in \bar{H}
	\end{equation*}
is continuous, then IVP~\eqref{eq:IVP} is maximally well-posed for $C^1$-solutions.
\end{theorem}

We give the proof in Appendix~\ref{subsec:proofs in Section 6}
because this is similar as that of Theorem~\ref{thm:maximal WP, C^1}.
We note that the assumption that $F$ is locally Lipschitzian about $C^1$-prolongations can be replaced with
the assumption of the uniqueness of solutions in view of Theorem~\ref{thm:continuity and uniqueness}.

\begin{remark}\label{rmk:H and bar{H}}
The following properties hold under the assumptions:
\begin{enumerate}
\item The semiflow $\mathbb{R}_+ \times H \ni (t, \phi) \mapsto S_0(t)\phi \in H$ is continuous.
\item $H$ is $C^1$-prolongable.
\end{enumerate}
\end{remark}

\section{Applications to state-dependent DDEs}\label{sec:state-dependent DDEs}

Throughout this section,
let $I$ be a past interval, $E = (E, \|\cdot\|_E)$ be a Banach space, and $H \subset \Map(I, E)$ be a history space.
Let
	\begin{equation*}
		f \colon \mathbb{R} \times E \times E \supset \dom(f) \to E
			\mspace{10mu} \text{and} \mspace{10mu}
		\tau \colon \mathbb{R} \times H \supset D \to \mathbb{R}_+
	\end{equation*}
be maps, where $D := \dom(\tau)$.
We suppose that $\tau$ satisfies
	\begin{equation*}
		\tau(t, \phi) \in -I \mspace{20mu} (\forall (t, \phi) \in D).
	\end{equation*}
This means that the following cases are considered:
\begin{itemize}
\item Case 1: $I = I^\infty$. $\tau$ is an unbounded function.
\item Case 2: $I = I^r$ for some $r > 0$. $\tau$ is bounded.
\end{itemize}

We consider a class of state-dependent DDEs of the form~\eqref{eq:sdDDE}
	\begin{equation*}
		\dot{x}(t) = f \bigl( t, x(t), x(t - \tau(t, I_tx)) \bigr).
	\end{equation*}
Then $\tau$ is the delay functional of DDE~\eqref{eq:sdDDE}.
For the given delay functional $\tau$, we introduce the map
	\begin{equation*}
		\rho_\tau \colon \mathbb{R} \times H \supset D \to \mathbb{R} \times E \times E, \mspace{20mu}
		\rho_\tau(t, \phi) = \bigl( t, \phi(0), \phi(-\tau(t, \phi)) \bigr).
	\end{equation*}
DDE~\eqref{eq:sdDDE} can be written as an RFDE with history space $H$
	\begin{equation*}
		\dot{x}(t) = F_{f, \tau}(t, I_tx),
	\end{equation*}
where $F_{f, \tau} \colon \mathbb{R} \times H \supset \dom \bigl( F_{f, \tau} \bigr) \to E$ is the history functional
defined by
	\begin{equation*}
		\dom \bigl( F_{f, \tau} \bigr) = \rho_\tau^{-1}(\dom(f)), \mspace{15mu}
		F_{f, \tau}(t, \phi) = f \circ \rho_\tau(t, \phi).
	\end{equation*}

\begin{lemma}
Let $(\sigma_0, \psi_0) \in \dom \bigl( F_{f, \tau} \bigr)$.
Suppose that $\dom(f)$ is a neighborhood of $\rho_\tau(\sigma_0, \psi_0)$ in $\mathbb{R} \times E \times E$.
If $\rho_\tau$ is continuous,
then $\dom \bigl( F_{f, \tau} \bigr)$ is a neighborhood of $(\sigma_0, \psi_0)$ in $D$.
\end{lemma}

\begin{proof}
We choose an open neighborhood $U$ of $\rho_\tau(\sigma_0, \psi_0)$ in $\mathbb{R} \times E \times E$
so that $U \subset \dom(f)$.
Since $\rho_\tau$ is continuous,
$\rho_\tau^{-1}(U)$ is an open neighborhood of $(\sigma_0, \psi_0)$ in $D$.
In view of
	\begin{equation*}
		\rho_\tau^{-1}(U) \subset\rho_\tau^{-1}(\dom(f)) = \dom \bigl( F_{f, \tau} \bigr),
	\end{equation*}
the conclusion is obtained.
\end{proof}

\subsection{Spaces of continuous maps, and history spaces}\label{subsec:spaces of continuous maps}

By using the evaluation map
	\begin{equation*}
		\ev_H \colon H \times I \to E,
			\mspace{15mu}
		\ev_H(\phi, \theta) = \phi(\theta),
	\end{equation*}
$\rho_\tau$ is expressed by
	\begin{equation*}
		\rho_\tau(t, \phi)
		= \bigl( t, \ev_H(\phi, 0), \ev_H(\phi, -\tau(t, \phi)) \bigr).
	\end{equation*}
Therefore, the continuity of $\tau$ and $\ev_H$ imply that of $\rho_\tau$.

The continuity of the evaluation map $\ev_H$ also implies the following:
\begin{itemize}
\item $H \subset C(I, E)$, and
\item the topology of $H$ is finer than the compact-open topology.
\end{itemize}
We refer the reader to \cite[Section 7]{Kelley 1955} for the relationship
between the continuity of the evaluation map and the compact-open topology.
Therefore, we arrive the following proposition.

\begin{proposition}
Let $I = I^r$ for some $r > 0$.
Then the following properties are equivalent:
\begin{enumerate}
\item[\emph{(a)}] $H$ is closed under prolongations and regulated by prolongations,
and the evaluation map $\ev_H$ is continuous.
\item[\emph{(b)}] $H = C(I^r, E)_\mathrm{u}$.
\end{enumerate}
\end{proposition}

\begin{proof}
(b) $\Rightarrow$ (a) is obvious.
We show (a) $\Rightarrow$ (b).
The continuity of $\ev_H$ implies $H \subset C(I^r, E)$.
This shows $H = C(I^r, E)$ from Proposition~\ref{prop:closedness under prolongations}.
The continuity of $\ev_H$ also implies that the topology of $H$ is finer than the topology of uniform convergence.
Therefore, the topology of $H$ is exactly equal to the topology of uniform convergence
from Theorem~\ref{thm:r < infty, regulation}.
\end{proof}

The next proposition ensures the continuity of
	\begin{equation*}
		\mathbb{R}_+ \times E \times H \ni (t, v, \phi) \mapsto S(t)(v, \phi) \in H
	\end{equation*}
for the history spaces which consist of continuous maps.
See also Remark~\ref{rmk:continuity and equi-continuity}.

\begin{proposition}
Let
	\begin{equation*}
		H :=
		\begin{cases}
			C(I^r, E)_\mathrm{u}, & \text{when $I = I^r$ for some $r > 0$}, \\
			C(I^\infty, E)_\mathrm{co}, & \text{when $I = I^\infty$}.
		\end{cases}
	\end{equation*}
Then for every $T > 0$, $(S(t))_{t \in [0, T]}$ is equi-continuous at $(0, \boldsymbol{0})$.
\end{proposition}

\begin{proof}
Fix $T > 0$.

\textbf{Case 1:} $I = I^r$ for some $r > 0$.

For all $t \in [0, T]$ and all $(v, \phi) \in E \times H$, we have
	\begin{align*}
		\|S(t)(v, \phi)\|_\infty
		&= \sup_{\theta \in [-r, 0]} \bigl\| \phi^{\wedge v}(t + \theta) \bigr\|_E \\
		&\le \|\phi\|_\infty + T\|v\|_E \\
		&\le (1 + T) \cdot (\|v\|_E + \|\phi\|_\infty).
	\end{align*}
This shows the conclusion.

\textbf{Case 2:} $I = I^\infty$.

Let $N$ be a neighborhood of $\boldsymbol{0}$ in $H$.
We choose an integer $k \ge 1$ and $\ep > 0$ so that
	\begin{equation*}
		\{\mspace{2mu} \phi \in H : \|\phi\|_{C[-k, 0]} < \ep \mspace{2mu}\} \subset N.
	\end{equation*}
Then for all $t \in [0, T]$ and all $(v, \phi) \in E \times H$,
	\begin{equation*}
		\|\phi\|_{C[-k, 0]} < \ep/2 \mspace{10mu} \text{and} \mspace{10mu} \|v\|_E < \ep/(2T)
	\end{equation*}
imply
	\begin{align*}
		\|S(t)(v, \phi)\|_{C[-k, 0]}
		&= \sup_{\theta \in [-k, 0]} \bigl\| \phi^{\wedge v}(t + \theta) \bigr\|_E \\
		&\le \|\phi\|_{C[-k, 0]} + T\|v\|_E \\
		&< (\ep/2) + T \cdot (\ep/2T) \\
		&= \ep.
	\end{align*}
This shows $S(t)(v, \phi) \in N$.
Therefore, the conclusion holds.
\end{proof}

\subsection{Lipschitz conditions for state-dependent DDEs}

In this subsection, we suppose
	\begin{equation*}
		H =
		\begin{cases}
			C(I^r, E)_\mathrm{u}, & \text{when $I = I^r$ for some $r > 0$}, \\
			C(I^\infty, E)_\mathrm{co}, & \text{when $I = I^\infty$}
		\end{cases}
	\end{equation*}
in view of Subsection~\ref{subsec:spaces of continuous maps}.
Then under the assumption of the continuity of $\tau$, $\rho_\tau$ is continuous.
The above can be shortly written by
	\begin{equation*}
		H = C(I, E)_\mathrm{co}
			\mspace{20mu}
		\text{for $I = I^r$ or $I = I^\infty$}
	\end{equation*}
because the compact-open topology of $C(I, E)$ equals to the topology of uniform convergence when $I$ is compact.

\begin{lemma}\label{lem:Lip condition for sd delay}
Let $(\sigma_0, \psi_0) \in \dom \bigl( F_{f, \tau} \bigr)$.
Suppose that $\tau$ is continuous.
If $f$ is locally Lipschitzian,
then there exist a neighborhood $W$ of $(\sigma_0, \psi_0)$ in $\mathbb{R} \times H$ and $L_f > 0$ such that
for all $(t, \phi_1), (t, \phi_2) \in W \cap \dom \bigl( F_{f, \tau} \bigr)$,
	\begin{align*}
		\bigl\| F_{f, \tau}(t, \phi_1) - F_{f, \tau}(t, \phi_2) \bigr\|_E
		&\le L_f \cdot (\|\phi_1(0) - \phi_2(0)\|_E \\
		&\mspace{80mu} + \|\phi_1(-\tau(t, \phi_1)) - \phi_2(-\tau(t, \phi_2))\|_E)
	\end{align*}
holds.
\end{lemma}

\begin{proof}
Let $(\sigma_0, x_0, y_0) := \rho_\tau(\sigma_0, \psi_0)$.
Since $f$ is locally Lipschitzian,
there are a neighborhood $U$ of $(\sigma_0, x_0, y_0)$ in $\mathbb{R} \times E \times E$ and $L_f > 0$ such that
for all $(t, x_1, y_1), (t, x_2, y_2) \in U \cap \dom(f)$,
	\begin{equation*}
		\|f(t, x_1, y_1) - f(t, x_2, y_2)\|_E
		\le L_f \cdot (\|x_1 - y_1\|_E + \|x_2 - y_2\|_E).
	\end{equation*}
By the continuity of $\rho_\tau$,
there is a neighborhood $W$ of $(\sigma_0, \psi_0)$ in $\mathbb{R} \times H$ such that $\rho_\tau(W \cap D) \subset U$.
Therefore, $(t, \phi) \in W \cap \dom \bigl( F_{f, \tau} \bigr)$ implies
	\begin{equation*}
		\rho_\tau(t, \phi) \in \dom(f), \mspace{15mu} \rho_\tau(t, \phi) \in \rho_\tau(W \cap D).
	\end{equation*}
This shows the conclusion.
\end{proof}

\subsubsection{General delay functionals}

For a general delay functional $\tau$, a difficulty arise by the following type estimation:
For $(t, \phi_1), (t, \phi_2) \in D$,
	\begin{align*}
		\|\phi_1(-\tau(t, \phi_1)) - \phi_2(-\tau(t, \phi_2))\|_E
		&\le \|\phi_1(-\tau(t, \phi_1)) - \phi_2(-\tau(t, \phi_1))\|_E \\
		&\mspace{60mu} + \|\phi_2(-\tau(t, \phi_1)) - \phi_2(-\tau(t, \phi_2))\|_E.
	\end{align*}
Here we have
	\begin{equation*}
		\|\phi_1(-\tau(t, \phi_1)) - \phi_2(-\tau(t, \phi_1))\|_E \le \|\phi_1 - \phi_2\|_\infty.
	\end{equation*}
However, the Lipschitz continuity of $\phi_2$ is necessary for the estimation of the second term.

\begin{proposition}\label{prop:almost loc Lip for sd delay}
Let $(\sigma_0, \psi_0) \in \dom \bigl( F_{f, \tau} \bigr)$.
Suppose that $\tau$ is continuous.
If $f$ is locally Lipschitzian, and if $\tau$ is almost locally Lipschitzian at $(\sigma_0, \psi_0)$,
then $F_{f, \tau}$ is almost locally Lipschitzian at $(\sigma_0, \psi_0)$.
\end{proposition}

The proof for $I = I^r$ $(r > 0)$ can be found
in \cite[Proposition 1.1]{Mallet-Paret--Nussbaum--Paraskevopoulos 1994}.
We give the proof for the case $I = I^\infty$.
See Definitions~\ref{dfn:almost locally Lip for I^r} and \ref{dfn:almost locally Lip for I^infty}
for the notion of almost local Lipschitz.

\begin{proof}[\textbf{Proof of Proposition~\ref{prop:almost loc Lip for sd delay}}]
Let $(M_k)_{k = 1}^\infty$ be a sequence of positive numbers.
We choose a neighborhood $W$ of $(\sigma_0, \psi_0)$ in $\mathbb{R} \times H$ and $L_f > 0$
given in Lemma~\ref{lem:Lip condition for sd delay}.
Since $\tau$ is almost locally Lipschitzian at $(\sigma_0, \psi_0)$,
we may assume the following by choosing $W$ sufficiently small:
there is $L_\tau > 0$ such that
	\begin{equation*}
		|\tau(t, \phi_1) - \tau(t, \phi_2)|
		\le L_\tau \cdot \|\phi_1 - \phi_2\|_\infty
	\end{equation*}
holds for all $(t, \phi_1), (t, \phi_2) \in W \cap D$ satisfying the conditions
\begin{enumerate}
\item[(i)] $\supp(\phi_1 - \phi_2)$ is compact, and
\item[(ii)] $\lip \bigl( \phi_1|_{[-k, 0]} \bigr), \lip \bigl( \phi_2|_{[-k, 0]} \bigr) \le M_k$ for all $k \ge 1$.
\end{enumerate}
By the continuity of $\tau$, we may also assume that $\tau(W \cap D) \subset [0, k_0]$ holds
for some integer $k_0 \ge 1$.
In view of Lemma~\ref{lem:Lip condition for sd delay}, we have
	\begin{align*}
		\bigl\| F_{f, \tau}(t, \phi_1) - F_{f, \tau}(t, \phi_2) \bigr\|_E
		&\le L_f \cdot (\|\phi_1(0) - \phi_2(0)\|_E \\
		&\mspace{80mu} + \|\phi_1(-\tau(t, \phi_1)) - \phi_2(-\tau(t, \phi_2))\|_E) \\
		&\le L_f \cdot (2\|\phi_1 - \phi_2\|_\infty + M_{k_0} \cdot |\tau(t, \phi_1) - \tau(t, \phi_2)|) \\
		&\le L_f (2 + M_{k_0}L_\tau) \cdot \|\phi_1 - \phi_2\|_\infty 
	\end{align*}
for $(t, \phi_1), (t, \phi_2) \in W \cap \dom \bigl( F_{f, \tau} \bigr)$ satisfying the above conditions (i) and (ii).
This shows the conclusion.
\end{proof}

\subsubsection{Constancy of delay functionals about memories}\label{subsec:constancy about memories}

The introduction of the Lipschitz about memories
was motived by the property of delay functionals
introduced by Rezounenko~\cite{Rezounenko 2009, Rezounenko 2012}.
The following property is a generalization of this property.

\begin{definition}[\cite{Nishiguchi 2017}, cf.\ \cite{Rezounenko 2009, Rezounenko 2012}]
\label{dfn:constancy about memories}
We say that $\tau$ is \textit{constant about memories} if there exists $R > 0$ such that
for all $(t, \phi_1), (t, \phi_2) \in D$,
	\begin{equation*}
		\supp(\phi_1 - \phi_2) \subset [-R, 0] \subset I
			\imply
		\tau(t, \phi_1) = \tau(t, \phi_2).
	\end{equation*}
We say that $\tau$ is \textit{locally constant about memories} at $(\sigma_0, \psi_0) \in D$
if there exists a neighborhood $W$ of $(\sigma_0, \psi_0)$ in $\mathbb{R} \times H$ such that
$\tau|_{W \cap D}$ is constant about memories.
When $\tau$ is locally constant about memories at each $(t_0, \phi_0) \in D$,
we simply say that $\tau$ is locally constant about memories.
\end{definition}

We note that variable delays independent from the dependent variable are trivial examples
of delay functionals which are constant about memories.

\begin{example}
For $a, b \in \mathbb{R}$ and $0 < \lambda < 1$, the differential equation
	\begin{equation}\label{eq:pantograph eq}
		\dot{x}(t) = ax(\lambda t) + bx(t) \mspace{20mu} (t \ge 0)
	\end{equation}
is the so-called \textit{pantograph equation}
(see \cite{Ockendon--Tayler 1971} for the origin of the pantograph equation).
\eqref{eq:pantograph eq} is a DDE with an unbounded variable delay in view of
	\begin{equation*}
		\lambda t = t - (1 - \lambda)t, \mspace{15mu} \lim_{t \to +\infty} (1 - \lambda)t = \infty.
	\end{equation*}
The delay $(1 - \lambda)t$ is independent from the dependent variable.
\end{example}

A nontrivial example of a class of delay functionals which are locally constant about memories is obtained
from the following proposition.

\begin{proposition}[cf.\ \cite{Nishiguchi 2017}]
Let $\delta(\cdot) \colon \mathbb{R} \to (0, \infty)$ and $\tau_0 \colon \mathbb{R} \times E \supset \dom(\tau_0) \to \mathbb{R}_+$ be functions.
Suppose that $\delta(\cdot)$ is continuous.
If $\tau \colon \mathbb{R} \times H \supset D \to \mathbb{R}_+$ is given by
$-\delta(t) \in I$ for all $t \in \mathbb{R}$, and
	\begin{align*}
		D
		&= \bigl\{ \mspace{2mu} (t, \phi) \in \mathbb{R} \times H :
				\bigl( t, \phi(-\delta(t)) \bigr) \in \dom(\tau_0) \mspace{2mu} \bigr\}, \\
		\tau(t, \phi)
		&= \tau_0 \bigl( t, \phi(-\delta(t)) \bigr),
	\end{align*}
then $\tau$ is locally constant about memories.
\end{proposition}

\begin{proof}
Fix $(\sigma_0, \psi_0) \in D$ and take $a > 0$.
Since $\delta(\cdot)$ is continuous, $\inf_{|t - \sigma_0| < a} \delta(t) > 0$.
Let $0 < R < \inf_{|t - \sigma_0| < a} \delta(t)$.
Then for all $(t, \phi_1), (t, \phi_2) \in D$
satisfying $|t - \sigma_0| < a$ and $\supp(\phi_1 - \phi_2) \subset [-R, 0]$, we have
	\begin{equation*}
		\tau(t, \phi_1)
		= \tau_0 \bigl( t, \phi_1(-\delta(t)) \bigr)
		= \tau_0 \bigl( t, \phi_2(-\delta(t)) \bigr)
		= \tau(t, \phi_2).
	\end{equation*}
This shows the conclusion.
\end{proof}

The Lipschitz about memories and the constancy about memories are combined in the following lemma.

\begin{proposition}[cf.\ \cite{Nishiguchi 2017}]
\label{prop:constancy about memories}
Suppose that $\tau \colon \mathbb{R} \times H \supset D \to \mathbb{R}_+$ is continuous.
Let $(\sigma_0, \psi_0) \in \dom \bigl( F_{f, \tau} \bigr)$.
If $f$ is locally Lipschitzian, and if $\tau$ is locally constant about memories at $(\sigma_0, \psi_0)$,
then $F_{f, \tau}$ is locally Lipschitzian about memories at $(\sigma_0, \psi_0)$.
\end{proposition}

\begin{proof}
We choose a neighborhood $W$ of $(\sigma_0, \psi_0)$ in $\mathbb{R} \times H$ and $L_f > 0$
given in Lemma~\ref{lem:Lip condition for sd delay}.
By choosing $W$ sufficiently small,
we may assume that there is $R > 0$ such that
	\begin{equation*}
		\tau(t, \phi_1) = \tau(t, \phi_2)
	\end{equation*}
holds for all $(t, \phi_1), (t, \phi_2) \in W \cap D$ satisfying $\supp(\phi_1 - \phi_2) \subset [-R, 0]$.
Therefore, for all $(t, \phi_1), (t, \phi_2) \in W \cap \dom \bigl( F_{f, \tau} \bigr)$,
$\supp(\phi_1 - \phi_2) \subset [-R, 0]$ implies
	\begin{equation*}
		\bigl\| F_{f, \tau}(t, \phi_1) - F_{f, \tau}(t, \phi_2) \bigr\|_E
		\le 2L_f \cdot \|\phi_1 - \phi_2\|_\infty
	\end{equation*}
from Lemma~\ref{lem:Lip condition for sd delay}.
\end{proof}

\subsection{Maximal well-posedness for state-dependent DDEs}

In this subsection, we derive results about the maximal well-posedness of state-dependent DDE~\eqref{eq:sdDDE}
as a consequence of the results obtained in this paper.

\begin{corollary}\label{cor:maximal WP, state-dependent}
Let $E = \mathbb{R}^n$, $I = I^r$ for some $r > 0$ or $I = I^\infty$, and
	\begin{equation*}
		H = C^{0, 1}_\mathrm{loc}(I, E)_\mathrm{co}, \mspace{15mu} \bar{H} = C(I, E)_\mathrm{co}.
	\end{equation*}
Suppose that
(i) $f$ is continuous and $\dom(f)$ is open in $\mathbb{R} \times E \times E$, and
(ii) there exists a continuous map $\bar{\tau} \colon \mathbb{R} \times \bar{H} \supset \bar{D} \to \mathbb{R}_+$
such that $\bar{\tau}$ is an extension of $\tau$ and $\bar{D}$ is open in $\mathbb{R} \times H$.
If $f$ is locally Lipschitzian, and if $\bar{\tau}$ is almost locally Lipschitzian,
then IVP~\eqref{eq:IVP} for $F := F_{f, \tau}$ is maximally well-posed for $C^1$-solutions.
\end{corollary}

\begin{proof}
Let $\bar{F}_{f, \tau} := F_{f, \bar{\tau}}$.
Then $\left( \bar{H}, \bar{F}_{f, \tau} \right)$ is an extension of $(H, F_{f, \tau})$.
By assumptions, $\bar{F}_{f, \tau} = f \circ \rho_{\bar{\tau}}$ is a continuous map
whose domain of definition $\rho_{\bar{\tau}}^{-1}(\dom(f))$ is open in $\mathbb{R} \times H$.
From Proposition~\ref{prop:almost loc Lip for sd delay},
$\bar{F}_{f, \tau}$ is almost locally Lipschitzian.
Applying Theorems~\ref{thm:almost local Lip for I = I^r} and \ref{thm:almost local Lip for I = I^infty},
$F_{f, \tau}$ is locally Lipschitzian about $C^1$-prolongations because
	\begin{equation*}
		\bar{F}_{f, \tau}^{0, 1} = F_{f, \tau}
	\end{equation*}
holds (see the notation in Subsection~\ref{subsec:almost loc Lip}).
Therefore, the conclusion follows as a consequence of Theorem~\ref{thm:maximal WP without uniform Lip, extension}.
\end{proof}

This is a result of maximal well-posedness applicable for a wide class of state-dependent DDEs when $E = \mathbb{R}^n$.
The following was obtained in \cite[Theorem 6.8]{Nishiguchi 2017}.

\begin{corollary}
Let $I = I^r$ for some $r > 0$ or $I = I^\infty$, and
	\begin{equation*}
		H = C(I, E)_\mathrm{co}.
	\end{equation*}
Suppose that
(i) $f$ is continuous and $\dom(f)$ is open in $\mathbb{R} \times E \times E$, and
(ii) $\tau$ is continuous where $D$ is open in $\mathbb{R} \times H$.
If $f$ is locally Lipschitzian, and if $\tau$ is locally constant about memories,
then IVP~\eqref{eq:IVP} for $F := F_{f, \tau}$ is maximally well-posed.
\end{corollary}

\begin{proof}
From Proposition~\ref{prop:constancy about memories},
$F_{f, \tau}$ is locally Lipschitzian about memories.
Therefore, $F_{f, \tau}$ is uniformly locally Lipschitzian about prolongations from Theorem~\ref{thm:Lip about memories}.
The conclusion is a consequence of Theorem~\ref{thm:maximal WP}.
\end{proof}

\appendix

\section{Maximal semiflows and processes}\label{sec:maximal semiflows and processes}

In this appendix, let $X$ be a topological space.

\subsection{Definitions}

\begin{definition}[ref.\ \cite{Hajek 1968}, \cite{Marsden--McCracken 1976}]
\label{dfn:maximal semiflow}
Let $\varPhi \colon \mathbb{R}_+ \times X \supset \dom(\varPhi) \to X$ be a map.
We call $\varPhi$ a \textit{maximal semiflow} in $X$ if the following properties are satisfied:
\begin{enumerate}
\item[(i)] There exists a function $T(\cdot) \colon X \to (0, \infty]$ such that
$\dom(\varPhi) = \bigcup_{x \in X} [0, T(x)) \times \{x\}$.
\item[(ii)] For all $x \in X$, $\varPhi(0, x) = x$.
\item[(iii)] For all $t_1, t_2 \in \mathbb{R}_+$ and all $x \in X$,
if $(t_1, x) \in \dom(\varPhi)$ and $(t_2, \varPhi(t_1, x)) \in \dom(\varPhi)$, then
	\begin{equation*}
		(t_1 + t_2, x) \in \dom(\varPhi), \mspace{15mu} \varPhi(t_1 + t_2, x) = \varPhi(t_2, \varPhi(t_1, x)).
	\end{equation*}
\end{enumerate}
$T(x)$ is called the \textit{escape time} of $x$,
and the function $T(\cdot) \colon X \to (0, \infty]$ is called the \textit{escape time function}.
\end{definition}

\begin{definition}[cf.\ \cite{Dafermos 1971}]
Let $\varOmega \subset \mathbb{R} \times X$ be a subset
and $\mathcal{P} \colon \mathbb{R}_+ \times \varOmega \supset \dom(\mathcal{P}) \to X$ be a map.
We call $\mathcal{P}$ a \textit{maximal process} in $\varOmega$ if the map
	$\overline{\mathcal{P}} \colon
	\mathbb{R}_+ \times \varOmega \supset \dom(\overline{\mathcal{P}}) \to \mathbb{R} \times X$
defined by
	\begin{equation*}
		\dom(\overline{\mathcal{P}}) = \dom(\mathcal{P}), \mspace{15mu}
		\overline{\mathcal{P}}(\tau, (t, x)) = (t + \tau, \mathcal{P}(\tau, t, x))
	\end{equation*}
becomes a maximal semiflow in $\varOmega$.
We call $\overline{\mathcal{P}}$ the \textit{extended semiflow} of $\mathcal{P}$.
\end{definition}

In this paper, we do not assume the continuity of maximal semiflows and maximal processes in advance
(see Definitions~\ref{dfn:continuity of maximal semiflow} and \ref{dfn:continuity of maximal process}).
In \cite{Hajek 1968}, the terminology of local semiflows is used.

By the paraphrase, $\mathcal{P}$ is a maximal process in $\varOmega$
if and only if the following conditions (the \textit{axiom of maximal processes}) are satisfied:
\begin{enumerate}
\item[(i)] $\overline{\mathcal{P}}(\dom(\mathcal{P})) \subset \varOmega$.
\item[(ii)] There exists a function $T(\cdot) \colon \varOmega \to (0, \infty]$ such that
	\begin{equation*}
		\dom(\mathcal{P}) = \bigcup_{(t, x) \in \varOmega} [0, T(t, x)) \times \{(t, x)\}.
	\end{equation*}
\item[(iii)] For all $(t, x) \in \varOmega$, $\mathcal{P}(0, t, x) = x$.
\item[(iv)] For all $\tau_1, \tau_2 \in \mathbb{R}_+$ and all $(t, x) \in \mathbb{R} \times X$,
conditions
	$(\tau_1, t, x) \in \dom(\mathcal{P})$ and $(\tau_2, t + \tau_1, \mathcal{P}(\tau_1, t, x)) \in \dom(\mathcal{P})$
imply
	\begin{equation*}
		(\tau_1 + \tau_2, t, x) \in \dom(\mathcal{P}), \mspace{15mu}
		\mathcal{P}(\tau_1 + \tau_2, t, x) = \mathcal{P}(\tau_2, t + \tau_1, \mathcal{P}(\tau_1, t, x)).
	\end{equation*}
\end{enumerate}
We call the above $T(\cdot) \colon \varOmega \to (0, \infty]$ the \textit{escape time function} of $\mathcal{P}$.

\begin{definition}[ref.\ \cite{Hajek 1968}, \cite{Marsden--McCracken 1976}]
\label{dfn:continuity of maximal semiflow}
Let $\varPhi \colon \mathbb{R}_+ \times X \supset \dom(\varPhi) \to X$ be a map.
We call $\varPhi$ a \textit{continuous maximal semiflow} in $X$ if
(i) $\varPhi$ is a maximal semiflow in $X$,
(ii) $\dom(\varPhi)$ is open in $\mathbb{R}_+ \times X$, and
(iii) $\varPhi$ is continuous.
\end{definition}

\begin{definition}\label{dfn:continuity of maximal process}
Let $\varOmega \subset \mathbb{R} \times X$ be a subset
and $\mathcal{P} \colon \mathbb{R}_+ \times \varOmega \supset \dom(\mathcal{P}) \to X$.
We call $\mathcal{P}$ a \textit{continuous maximal process} in $\varOmega$
if the extended semiflow $\overline{\mathcal{P}}$ is a continuous maximal semiflow in $\varOmega$.
\end{definition}

\subsection{Continuity}

In this subsection, we investigate the continuity property of maximal semiflows in $X$.
The continuity property of maximal processes is reduced to that of maximal semiflows
by taking the corresponding extended semiflows.

Let $\varPhi$ be a maximal semiflow in $X$ with the escape time function $T(\cdot) \colon X \to (0, \infty]$.
By definition, $(t, x) \in \dom(\varPhi)$ is equivalent to $t < T(x)$.
Therefore, this is also equivalent to $[0, t] \times \{x\} \subset \dom(\varPhi)$.

\begin{lemma}\label{lem:escape time function}
The following properties are equivalent:
\begin{enumerate}
\item[\emph{(a)}] $\dom(\varPhi)$ is open in $\mathbb{R}_+ \times X$.
\item[\emph{(b)}] $T(\cdot)$ is lower semi-continuous.
\end{enumerate}
\end{lemma}

\begin{proof}
(a) $\Rightarrow$ (b):
Fix $x_0 \in X$ and let $T < T(x_0)$.
Since $(T, x_0)$ belongs to the open set $\dom(\varPhi)$ of $\mathbb{R}_+ \times X$,
there are $\ep > 0$ and a neighborhood $N$ of $x_0$ in $X$ such that
	\begin{equation*}
		[T, T + \ep] \times N \subset \dom(\varPhi).
	\end{equation*}
Therefore, $T < T(x)$ holds for all $x \in N$.
This means that $T(\cdot)$ is lower semi-continuous at $x_0$.

(b) $\Rightarrow$ (a):
Let $(t_0, x_0) \in \dom(\varPhi)$.
We choose $\ep > 0$ so that $t_0 + \ep < T(x_0)$.
By the lower semi-continuity,
there is a neighborhood $N$ of $x_0$ in $X$ such that
	\begin{equation*}
		T(x) > t_0 + \ep \mspace{20mu} (\forall x \in N).
	\end{equation*}
Therefore, we have $[0, t_0 + \ep] \times N \subset \dom(\varPhi)$,
which means that $\dom(\varPhi)$ is a neighborhood of $(t_0, x_0)$ in $\mathbb{R}_+ \times X$.
\end{proof}

The following theorem is a corollary of \cite[Theorem 15]{Hajek 1968}
which was proved by using the notion of \textit{germs}.
In this paper, we give a direct proof of that theorem.
See also \cite[Theorem A.7]{Nishiguchi 2017}.

\begin{theorem}[\cite{Hajek 1968}]
\label{thm:continuity of maximal semiflows}
$\varPhi$ is a continuous maximal semiflow in $X$ if and only if both of the following conditions are satisfied:
\begin{enumerate}
\item[\emph{(i)}] For every $x \in X$,
the orbit $[0, T(x)) \ni t \mapsto \varPhi(t, x) \in X$ is continuous.
\item[\emph{(ii)}] For every $x \in X$, there exist $T > 0$ and a neighborhood $N$ of $x$ in $X$ such that
	\begin{enumerate}
	\item[\emph{(ii-1)}] $[0, T] \times N \subset \dom(\varPhi)$, and
	\item[\emph{(ii-2)}] the restriction $\varPhi|_{[0, T] \times N} \colon [0, T] \times N \to X$ is continuous.
	\end{enumerate}
\end{enumerate}
\end{theorem}

\begin{proof}
(Only-if-part).
We check that the above (i) and (ii) hold for each fixed $x \in X$.
From Lemma~\ref{lem:escape time function},
there are $T > 0$ and a neighborhood $N$ of $x$ in $X$ such that $[0, T] \times N \subset \dom(\varPhi)$.
That is, (ii-1) holds.
Since $\varPhi$ is a continuous map, the restricted maps
	\begin{equation*}
		\varPhi|_{[0, T_\varPhi(x)) \times \{x\}}, \mspace{15mu}
		\varPhi|_{[0, T] \times N}
	\end{equation*}
are also continuous.
This shows that (i) and (ii-2) hold.

(If-part).
\textbf{Step 1. Setting.}

Fix $x_0 \in X$.
Define a subset
	\begin{equation*}
		S_{x_0} \subset (0, T(x_0))
	\end{equation*}
by the following manner:
$T \in S_{x_0}$ if and only if
there exists a neighborhood $N$ of $x_0$ in $X$ with the following properties:
\begin{enumerate}
\item[(1)] $[0, T] \times N \subset \dom(\varPhi)$, and
\item[(2)] $\varPhi|_{[0, T] \times N}$ is continuous.
\end{enumerate}
By the assumption (ii), $S_{x_0} \ne \emptyset$.
Therefore,
	\begin{equation*}
		\sup(S_{x_0}) \in (0, T(x_0)]
	\end{equation*}
exists.
Here we interpret $\sup(S_{x_0}) = \infty$ when $T(x_0) = \infty$.
If $\sup(S_{x_0}) = T(x_0)$ is established,
for every $T < T(x_0)$, there is $T' \in S_{x_0}$ such that $T' > T$.
Therefore, by the definition of $S_{x_0}$, we have the following properties:
\begin{itemize}
\item $T(\cdot)$ is lower semi-continuous at $x_0$.
\item $\varPhi$ is continuous at each $(t_0, x_0) \in \dom(\varPhi)$.
\end{itemize}

\textbf{Step 2. Proof by a contradiction.}

From Step 1, it is sufficient to prove
	\begin{equation*}
		\sup(S_{x_0}) = T(x_0) \mspace{20mu} (\forall x_0 \in X).
	\end{equation*}
Fix $x_0 \in X$.
Let
	\begin{equation*}
		t_* := \sup(S_{x_0}).
	\end{equation*}
We suppose $t_* < T(x_0)$ and derive a contradiction.
Let
	\begin{equation*}
		x_* := \varPhi(t_*, x_0) \in X.
	\end{equation*}
By the assumption (ii), there are $T_* > 0$ and a neighborhood $N_*$ of $x_*$ in $X$ such that
	\begin{equation*}
		[0, T_*] \times N_* \subset \dom(\varPhi), \mspace{15mu} \text{$\varPhi|_{[0, T_*] \times N_*}$ is continuous}.
	\end{equation*}
We note that one cannot conclude $t_* \in S_{x_0}$ in general.
By the assumption (i), one can choose $t' \in (t_* - (T_*/2), t_*)$ so that $\varPhi(t', x_0) \in N_*$.
Since $t' < t_*$, there is a neighborhood $N'$ of $x_0$ such that
	\begin{equation*}
		[0, t'] \times N' \subset \dom(\varPhi), \mspace{15mu} \text{$\varPhi|_{[0, t'] \times N'}$ is continuous}
	\end{equation*}
by the definition of $S_{x_0}$.
By the continuity of $\varPhi(t', \cdot)|_{N'}$ at $x_0$, we may assume $\varPhi(t', N') \subset N_*$
by choosing $N'$ sufficiently small.
Then for all $t \in [t', t' + T_*]$ and all $x \in N'$,
	\begin{equation*}
		(t', x) \in \dom(\varPhi), \mspace{15mu} (t - t', \varPhi(t', x)) \in [0, T_*] \times N_* \subset \dom(\varPhi).
	\end{equation*}
Therefore, $(t, x) = (t' + (t - t'), x) \in \dom(\varPhi)$ by the maximality, that is
	\begin{equation*}
		[t', t' + T_*] \times N' \subset \dom(\varPhi).
	\end{equation*}
This shows
	\begin{equation*}
		[0, t' + T_*] \times N' = ([0, t'] \times N') \cup ([t', t' + T_*] \times N') \subset \dom(\varPhi).
	\end{equation*}
We check the continuity of $\varPhi|_{[0, t' + T_*] \times N'}$.
By the pasting lemma in General Topology,
it is sufficient to show the continuity of $\varPhi|_{[0, t'] \times N'}$ and $\varPhi|_{[t', t' + T_*] \times N'}$
because $[0, t'] \times N'$ and $[t', t' + T_*] \times N'$ are closed sets of $[0, t' + T_*] \times N'$.
We know the continuity of $\varPhi|_{[0, t'] \times N'}$.
The continuity of $\varPhi|_{[t', t' + T_*] \times N'}$ is obtained in view of
	\begin{equation*}
		\varPhi(t, x) = \varPhi(t - t', \varPhi(t', x))
			\mspace{20mu}
		\left( \forall (t, x) \in [t', t' + T_*] \times N' \right)
	\end{equation*}
and the continuity of $\varPhi(t', \cdot)|_{N'}$ and $\varPhi|_{[0, T_*] \times N_*}$.

Then, we obtain $t' + T_* \in S_{x_0}$ by the definition, which contradicts $t_* = \sup(S_{x_0})$ because
	\begin{equation*}
		t' + T_* > t_* + (T_*/2) > t_*.
	\end{equation*}
Thus, $t_* = T(x_0)$ follows.

Since $x_0 \in X$ is arbitrary, this completes the proof.
\end{proof}

\begin{corollary}\label{cor:continuity of maximal processes}
Let $\varOmega \subset \mathbb{R} \times X$ be a subset and
$\mathcal{P}$ be a maximal process in $\varOmega$
with the escape time function $T(\cdot) \colon \varOmega \to (0, \infty]$.
Then $\mathcal{P}$ is a continuous maximal process in $\varOmega$ if and only if
both of the following conditions are satisfied:
\begin{enumerate}
\item[\emph{(i)}] For every $(t, x) \in \varOmega$,
the orbit $[0, T(t, x)) \ni \tau \mapsto \mathcal{P}(\tau, t, x) \in \varOmega$ is continuous.
\item[\emph{(ii)}] For every $(t, x) \in \varOmega$,
there exist $T > 0$ and a neighborhood $N$ of $(t, x)$ in $\varOmega$ such that
	\begin{enumerate}
	\item[\emph{(ii-1)}] $[0, T] \times N \subset \dom(\mathcal{P})$, and
	\item[\emph{(ii-2)}] the restriction $\mathcal{P}|_{[0, T] \times N} \colon [0, T] \times N \to \varOmega$ is continuous.
	\end{enumerate}
\end{enumerate}
\end{corollary}

\begin{proof}
Let $\overline{\mathcal{P}}$ be the extended semiflow of $\mathcal{P}$.
Then the continuity of $\overline{\mathcal{P}}$ is equivalent to the continuity of $\mathcal{P}$.
Therefore, the conclusion is obtained
by applying Theorem~\ref{thm:continuity of maximal semiflows} as $\varPhi = \overline{\mathcal{P}}$.
\end{proof}

\section{Fundamental properties for maximal existence and uniqueness}\label{ap:maximal sol}

\subsection{Local uniqueness and maximal uniqueness}

\begin{lemma}[\cite{Nishiguchi 2017}]
\label{lem:maximal uniqueness for C^1-solutions}
Suppose that IVP~\eqref{eq:IVP} satisfies the local uniqueness for $C^1$-solutions.
Then the following statement holds for every $(t_0, \phi_0) \in \dom(F)$:
For any $C^1$-solutions $x_i \colon J + I \to E$ $(i = 1, 2)$ of $\eqref{eq:IVP}_{t_0, \phi_0}$, $x_1 = x_2$ holds.
\end{lemma}

\begin{proof}
Fix $(t_0, \phi_0) \in \dom(F)$.
Let $x_i \colon J + I \to E$ $(i = 1, 2)$ be $C^1$-solutions of $\eqref{eq:IVP}_{t_0, \phi_0}$.
Since
	\begin{equation*}
		J = \bigcup_{T \in [0, |J|)} [t_0, t_0 + T],
	\end{equation*}
it is sufficient to consider the case $J = [t_0, t_0 + T]$ for some $0 < T < \infty$.
We consider the set $S$ given by
	\begin{equation*}
		S = \left\{ \mspace{2mu} \tau \in (t_0, t_0 + T] : x_1|_{[t_0, \tau]} = x_2|_{[t_0, \tau]} \mspace{2mu} \right\}.
	\end{equation*}
By the local uniqueness, $S \ne \emptyset$.
Therefore,
	\begin{equation*}
		\alpha := \sup(S) \in (t_0, t_0 + T]
	\end{equation*}
exists.
Then $x_1|_{[t_0, \alpha)} = x_2|_{[t_0, \alpha)}$,
which implies $x_1|_{[t_0, \alpha]} = x_2|_{[t_0, \alpha]}$ by the continuity.
When $\alpha = t_0 + T$, the conclusion is obtained.
We suppose $\alpha < t_0 + T$ and derive a contradiction.

Let $(t_1, \phi_1) := (\alpha, I_{\alpha}x_1) = (\alpha, I_{\alpha}x_2)$.
Then $(t_1, \phi_1) \in \dom(F)$, and
	\begin{equation*}
		x_1|_{[t_1, t_0 + T] + I}, \mspace{15mu} x_2|_{[t_1, t_0 + T] + I}
	\end{equation*}
are $C^1$-solutions of $\eqref{eq:IVP}_{t_1, \phi_1}$.
By the local uniqueness,
there exists $\delta > 0$ such that $x_1|_{[t_1, t_1 + \delta]} = x_2|_{[t_1, t_1 + \delta]}$.
This means
	\begin{equation*}
		x_1|_{[t_0, t_1 + \delta]} = x_2|_{[t_0, t_1 + \delta]},
	\end{equation*}
which contradicts $t_1 = \alpha = \sup(S)$.
Therefore, we obtain $\alpha = t_0 + T$.
\end{proof}

This proof is independent from the class of solutions (see also \cite[Lemma 3.3]{Nishiguchi 2017}).

\subsection{Maximal solutions}

\begin{definition}\label{dfn:binary relation}
Let $(t_0, \phi_0) \in \dom(F)$.
We define a binary relation $\le$ on the set of $C^1$-solutions of $\eqref{eq:IVP}_{t_0, \phi_0}$ as follows:
For any $C^1$-solutions $x_i \colon J_i + I \to E$ $(i = 1, 2)$ of $\eqref{eq:IVP}_{t_0, \phi_0}$,
	\begin{equation*}
		x_1 \le x_2 \iff \text{$J_1 \subset J_2$ and $x_1|_{J_1} = x_2|_{J_1}$}.
	\end{equation*}
In this case, $x_2$ is said to be an \textit{extension} of $x_1$.
\end{definition}

The above binary relation satisfies the following properties:
\begin{itemize}
\item (Reflexivity) $x \le x$.
\item (Antisymmetry) $x_1 \le x_2$ and $x_2 \le x_1$ imply $x_1 = x_2$.
\item (Transitivity) $x_1 \le x_2$, $x_2 \le x_3$ implies $x_1 \le x_3$.
\end{itemize}
That is, $\le$ is a partial order on the set of $C^1$-solutions of $\eqref{eq:IVP}_{t_0, \phi_0}$
for each $(t_0, \phi_0) \in \dom(F)$.

\begin{lemma}\label{lem:local uniqueness for C^1-solutions and total order}
If IVP~\eqref{eq:IVP} satisfies the local uniqueness for $C^1$-solutions,
then for all $(t_0, \phi_0) \in \dom(F)$,
$\le$ is a total order on the set of $C^1$-solutions of $\eqref{eq:IVP}_{t_0, \phi_0}$.
\end{lemma}

\begin{proof}
Fix $(t_0, \phi_0) \in \dom(F)$.
Let $x_i \colon J_i + I \to E$ $(i = 1, 2)$ be $C^1$-solutions of $\eqref{eq:IVP}_{t_0, \phi_0}$.
Then $J_1 \subset J_2$ or $J_1 \supset J_2$.
Without loss of generality, we may assume $J_1 \subset J_2$.
Since $x_2|_{J_1 + I}$ is also a $C^1$-solution of $\eqref{eq:IVP}_{t_0, \phi_0}$,
$x_1 = x_2|_{J_1 + I}$ holds from Lemma~\ref{lem:maximal uniqueness for C^1-solutions}.
This means $x_1 \le x_2$, and therefore, $x_1$ and $x_2$ are comparable.
\end{proof}

For each $(t_0, \phi_0) \in \dom(F)$,
a maximal element of the set of $C^1$-solutions of $\eqref{eq:IVP}_{t_0, \phi_0}$ with respect to the partial order
is called a \textit{maximal $C^1$-solution}.
From Lemma~\ref{lem:local uniqueness for C^1-solutions and total order},
a maximal $C^1$-solution is unique if it exists under the local uniqueness for $C^1$-solutions.

\begin{lemma}\label{lem:join of C^1-solutions}
Fix $(t_0, \phi_0) \in \dom(F)$.
Let $x_0 \colon [t_0, t_1] + I \to E$ be a $C^1$-solution of $\eqref{eq:IVP}_{t_0, \phi_0}$
and $x_1 \colon [t_1, t_2] + I \to E$ be a $C^1$-solution of $\eqref{eq:IVP}_{t_1, \phi_1}$, where $\phi_1 := I_{t_1}x_0$.
Then $x \colon [t_0, t_2] + I \to E$ defined by $I_{t_0}x = \phi_0$, and
	\begin{equation*}
		x(t) =
		\begin{cases}
			x_0(t), & t_0 \le t \le t_1, \\
			x_1(t), & t_1 \le t \le t_2
		\end{cases}
	\end{equation*}
is a $C^1$-solution of $\eqref{eq:IVP}_{t_0, \phi_0}$.
\end{lemma}

\begin{proof}
Let $t \in [t_1, t_2]$.
For all $\theta \in I$, we have
	\begin{align*}
		I_tx_1(\theta)
		&= x_1(t + \theta) \\
		&=
		\begin{cases}
			x_1(t + \theta), & t_1 - t \le \theta \le 0, \\
			\phi_1(t - t_1 + \theta), & \theta \le t_1 - t
		\end{cases} \\
		&=
		\begin{cases}
			x_1(t + \theta), & t_1 - t \le \theta \le 0, \\
			x_0(t + \theta), & \theta \le t_1 - t.
		\end{cases}
	\end{align*}
This shows $I_tx_1 = I_tx$.
Therefore, for all $t \in [t_0, t_2] \setminus \{t_1\}$, we have $\dot{x}(t) = F(t, I_tx)$.
We also have
	\begin{equation*}
		\dot{x}_-(t_1) = \dot{x}_0(t_1) = F(t_1, I_{t_1}x) = \dot{x}_1(t_1) = \dot{x}_+(t_1),
	\end{equation*}
which shows that $x$ is differentiable at $t_1$ and $\dot{x}(t_1) = F(t_1, I_{t_1}x)$.
The continuity of $\dot{x}$ also follows.
Thus, $x$ is a $C^1$-solution of $\eqref{eq:IVP}_{t_0, \phi_0}$.
\end{proof}

Applying Lemma~\ref{lem:join of C^1-solutions},
we have the following property under the assumption that IVP~\eqref{eq:IVP} satisfies the local existence for $C^1$-solutions:
If $x \colon J + I \to E$ is a maximal $C^1$-solution of $\eqref{eq:IVP}_{t_0, \phi_0}$,
$J = [t_0, t_0 + T)$ holds for some $0 < T \le \infty$.

\begin{proposition}\label{prop:maximal C^1-solutions and solution process}
If IVP~\eqref{eq:IVP} satisfies the local existence and local uniqueness for $C^1$-solutions,
then for every $(t_0, \phi_0) \in \dom(F)$, $\eqref{eq:IVP}_{t_0, \phi_0}$ has the unique maximal $C^1$-solution
	\begin{equation*}
		x_F(\cdot; t_0, \phi_0) \colon [t_0, t_0 + T_F(t_0, \phi_0)) + I \to E.
	\end{equation*}
Furthermore, $\mathcal{P}_F$ defined by
	\begin{align*}
		\dom(\mathcal{P}_F)
		&= \bigcup_{(t_0, \phi_0) \in \dom(F)} [0, T_F(t_0, \phi_0)) \times \{(t_0, \phi_0)\}, \\
		\mathcal{P}_F(\tau, t_0, \phi_0)
		&= I_{t_0 + \tau} [x_F(\cdot; t_0, \phi_0)]
	\end{align*}
is a maximal process in $\dom(F)$.
\end{proposition}

\begin{proof}
\begin{flushleft}
\textbf{Step 1. Construction of maximal solution}
\end{flushleft}

Fix $(t_0, \phi_0) \in \dom(F)$.
From Lemma~\ref{lem:local uniqueness for C^1-solutions and total order},
a maximal $C^1$-solution of $\eqref{eq:IVP}_{t_0, \phi_0}$ is unique if it exists.
We now construct the maximal $C^1$-solution.
For each $C^1$-solution $x$ of $\eqref{eq:IVP}_{t_0, \phi_0}$,
let $J_x$ be a left-closed interval with the left end point $t_0$ satisfying
	\begin{equation*}
		\dom(x) = J_x + I.
	\end{equation*}
Let $J := \bigcup_{x} J_x$.
Then $J$ is also a left-closed interval with the left end point $t_0$.
We define the map $X \colon J + I \to E$ by $I_{t_0}X = \phi_0$ and
	\begin{equation*}
		X(t) = x(t) \mspace{20mu} \text{when $t \in J_x$}.
	\end{equation*}
From Lemma~\ref{lem:maximal uniqueness for C^1-solutions}, this is well-defined.
Furthermore, $X$ is a $C^1$-solution of $\eqref{eq:IVP}_{t_0, \phi_0}$ from Lemma~\ref{lem:join of C^1-solutions}.
We finally check the maximality of $X$.
Let $x$ be a $C^1$-solution of $\eqref{eq:IVP}_{t_0, \phi_0}$ satisfying $x \ge X$.
By the construction, we have $J_x \subset J$, which implies $J_x = J$.
Therefore, $x = X$.
Let $J_X = [t_0, t_0 + T_F(t_0, \phi_0))$.

\begin{flushleft}
\textbf{Step 2. Axiom of maximal processes}
\end{flushleft}

Fix $(t_0, \phi_0) \in \dom(F)$.
Let
	\begin{equation*}
		x_F(\cdot; t_0, \phi_0) \colon [t_0, t_0 + T_F(t_0, \phi_0)) + I \to E
	\end{equation*}
be the maximal $C^1$-solution of $\eqref{eq:IVP}_{t_0, \phi_0}$.

We check that $\mathcal{P}_F$ satisfies the axiom of maximal processes.
\begin{enumerate}
\item[(i)] Let $(\tau, t_0, \phi_0) \in \dom(\mathcal{P}_F)$, where $\tau \in [0, T_F(t_0, \phi_0))$.
Then by definition,
	\begin{equation*}
		(t_0 + \tau, I_{t_0 + \tau}[x_F(\cdot; t_0, \phi_0)]) \in \dom(F).
	\end{equation*}
\item[(ii)] $\dom(\mathcal{P}_F) = \bigcup_{(t_0, \phi_0) \in \dom(F)} [0, T_F(t_0, \phi_0)) \times \{(t_0, \phi_0)\}$.
\item[(iii)] For all $(t_0, \phi_0) \in \dom(F)$,
$\mathcal{P}_F(0, t_0, \phi_0) = (t_0, I_{t_0}[x_F(\cdot; t_0, \phi_0)]) = (t_0, \phi_0)$.
\item[(iv)] Let $\tau_1, \tau_2 \in \mathbb{R}_+$ and $(t_0, \phi_0) \in \dom(F)$.
Suppose
	\begin{equation*}
		(\tau_1, t_0, \phi_0) \in \dom(\mathcal{P}_F), \mspace{15mu}
		(\tau_2, t_0 + \tau_1, \mathcal{P}_F(\tau_1, t_0, \phi_0)) \in \dom(\mathcal{P}_F).
	\end{equation*}
These are equivalent to
	\begin{align*}
		[t_0, t_0 + \tau_1] + I
		&\subset \dom \bigl( x_F(\cdot; t_0, \phi_0) \bigr), \\
		[t_0 + \tau_1, (t_0 + \tau_1) + \tau_2] + I
		&\subset \dom \bigl( x_F \bigl( \cdot; t_0 + \tau_1, \mathcal{P}_F(\tau_1, t_0, \phi_0) \bigr) \bigr),
	\end{align*}
respectively.
From Lemma~\ref{lem:join of C^1-solutions} and by the maximality, we have
	\begin{equation*}
		[t_0, t_0 + \tau_1 + \tau_2] + I \subset \dom \bigl( x_F(\cdot; t_0, \phi_0) \bigr)
	\end{equation*}
and
	\begin{align*}
		\mathcal{P}_F(\tau_1 + \tau_2, t_0, \phi_0)
		&= I_{t_0 + (\tau_1 + \tau_2)}[x_F(\cdot; t_0, \phi_0)] \\
		&= I_{(t_0 + \tau_1) + \tau_2} \bigl[ x_F \bigl( \cdot; t_0 + \tau_1, \mathcal{P}_F(\tau_1, t_0, \phi_0) \bigr) \bigr] \\
		&= \mathcal{P}_F \bigl( \tau_2, t_0 + \tau_1, \mathcal{P}_F(\tau_1, t_0, \phi_0) \bigr).
	\end{align*}
\end{enumerate}

This completes the proof.
\end{proof}

\section{Continuity for families of maps}\label{sec:family of maps}

Let $X$ be a topological space, $Y = (Y, \mathcal{V})$ be a uniform space, and $\Lambda$ be a topological space.
In this appendix, we consider a family $(f_\lambda)_{\lambda \in \Lambda}$ of maps from $X$ to $Y$
and the map $f \colon \Lambda \times X \to Y$ defined by
	\begin{equation*}
		f(\lambda, x) = f_\lambda(x).
	\end{equation*}
For the theory of uniform spaces, we refer the reader to Kelley~\cite{Kelley 1955}.
In the following, we briefly recall some notations:
\begin{itemize}
\item For $V \in \mathcal{V}$, $V^{-1}$ is defined by
$V^{-1} = \{\mspace{2mu} (y, y') \in Y \times Y : (y', y) \in V \mspace{2mu}\}$.
\item For $V_1, V_2 \in \V$, the \textit{composition} $V_1 \circ V_2$ is defined by
	\begin{equation*}
		V_1 \circ V_2
		= \{\mspace{2mu} (y, y') \in Y \times Y
			: \text{$\exists z \in Y$ s.t.\ $(y, z) \in V_1$ and $(z, y') \in V_2$}
		\mspace{2mu}\}.
	\end{equation*}
\item $V \in \mathcal{V}$ is said to be \textit{symmetric} if $V = V^{-1}$.
\item For $y \in Y$ and $V \in \V$, $V[y]$ is defined by
	\begin{equation*}
		V[y] = \{\mspace{2mu} y' \in Y : (y, y') \in V \mspace{2mu}\}.
	\end{equation*}
\end{itemize}
We note that for every $V \in \V$, we can choose a symmetric $V' \in \V$ so that $V' \circ V' \subset V$
by the following properties:
\begin{itemize}
\item For all $U$, $U \cap U^{-1}$ is symmetric.
\item For all $U, U'$, $U \subset U'$ implies $U \circ U \subset U' \circ U'$.
\end{itemize}

\subsection{Equi-continuity and joint continuity}

\begin{definition}\label{dfn:equi-continuity}
$(f_\lambda)_{\lambda \in \Lambda}$ is said to be \textit{equi-continuous} at $x_0 \in X$
if for every $V \in \mathcal{V}$, there exists a neighborhood $N$ of $x_0$ in $X$ such that
for all $(\lambda, x) \in \Lambda \times N$,
	\begin{equation*}
		(f_{\lambda}(x), f_{\lambda}(x_0)) \in V
	\end{equation*}
holds.
We say that $(f_\lambda)_{\lambda \in \Lambda}$ is \textit{locally equi-continuous} at $x_0 \in X$
if for each $\lambda_0 \in \Lambda$, there exists a neighborhood $W$ of $\lambda_0$ in $\Lambda$ such that
$(f_\lambda)_{\lambda \in W}$ is equi-continuous at $x_0$.
\end{definition}

\begin{remark}
By definition, $(f_\lambda)_{\lambda \in \Lambda}$ is locally equi-continuous at every $x \in X$
if and only if for every $(\lambda_0, x_0) \in \Lambda \times X$,
there is a neighborhood $W$ of $\lambda_0$ in $\Lambda$ such that
$(f_\lambda)_{\lambda \in W}$ is equi-continuous at $x_0$.
We note that the neighborhood $W$ depends on $x_0$.
\end{remark}

\begin{theorem}\label{thm:equi-continuity from continuity}
If $f$ is continuous and $\Lambda$ is compact,
then $(f_\lambda)_{\lambda \in \Lambda}$ is equi-continuous at every $x_0 \in X$.
\end{theorem}

\begin{proof}
Suppose that $(f_\lambda)_{\lambda \in \Lambda}$ is not equi-continuous at some $x_0 \in X$,
and derive a contradiction.
Then there exists $V \in \mathcal{V}$ with the following property:
For every neighborhood $N$ of $x_0$ in $X$, there exists $(\lambda_N, x_N) \in \Lambda \times N$ such that
	\begin{equation*}
		(f_{\lambda_N}(x_N), f_{\lambda_N}(x_0)) \not \in V.
	\end{equation*}
Let $\mathcal{N}$ denote the directed set of all neighborhoods of $x_0$ in $X$ by the set inclusion.

By the choice of $(\lambda_N, x_N)$, $(x_N)_{N \in \mathcal{N}}$ converges to $x_0$.
Since $\Lambda$ is compact,
there is a subnet $(\lambda_{N_\alpha})_{\alpha \in A}$ of $(\lambda_N)_{N \in \mathcal{N}}$
which converges to some $\lambda_* \in \Lambda$.
Choose a symmetric $V' \in \mathcal{V}$ so that $V' \circ V' \subset V$.
The continuity of $f$ implies
	\begin{equation*}
		f(\lambda_{N_\alpha}, x_{N_\alpha}) \to f(\lambda_*, x_0)
			\mspace{10mu} \text{and} \mspace{10mu}
		f(\lambda_{N_\alpha}, x_0) \to f(\lambda_*, x_0).
	\end{equation*}
Namely, there is $\alpha_0 \in A$ such that for all $\alpha \in A$, $\alpha \ge \alpha_0$ implies
	\begin{equation*}
		f(\lambda_{N_\alpha}, x_{N_\alpha}) \in V'[f(\lambda_*, x_0)]
			\mspace{10mu} \text{and} \mspace{10mu}
		f(\lambda_{N_\alpha}, x_0) \in V'[f(\lambda_*, x_0)].
	\end{equation*}
These relations show that
	\begin{equation*}
		\bigl( f_{\lambda_{N_\alpha}}(x_{N_\alpha}), f_{\lambda_{N_\alpha}}(x_0) \bigr) \in V
	\end{equation*}
holds for all $\alpha \ge \alpha_0$, which is a contradiction.
This completes the proof.
\end{proof}

\begin{theorem}\label{thm:continuity from local equi-continuity}
If $f(\cdot, x) \colon \Lambda \to Y$ is continuous for every $x \in X$,
and if $(f_\lambda)_{\lambda \in \Lambda}$ is locally equi-continuous at every $x \in X$,
then $f$ is continuous.
\end{theorem}

\begin{proof}
Fix $(\lambda_0, x_0) \in \Lambda \times X$.
We show the continuity of $f$ at $(\lambda_0, x_0)$.
By the local equi-continuity,
there is a neighborhood $W$ of $\lambda_0$ in $\Lambda$ such that
$(f_\lambda)_{\lambda \in W}$ is equi-continuous at $x_0$.
Let $V \in \mathcal{V}$ and choose a symmetric $V' \in \mathcal{V}$ so that $V' \circ V' \subset V$.
By the continuity of $\Lambda \ni \lambda \mapsto f_\lambda(x_0) \in Y$,
we may assume $f_\lambda(x_0) \in V'[f_{\lambda_0}(x_0)]$ by choosing $W$ sufficiently small.
By the equi-continuity, there is a neighborhood $N$ of $x_0$ in $X$ such that for all $(\lambda, x) \in W \times N$,
	\begin{equation*}
		(f_\lambda(x), f_\lambda(x_0)) \in V'.
	\end{equation*}
Therefore, we have
	\begin{equation*}
		f(\lambda, x)\in V[f(\lambda_0, x_0)]
	\end{equation*}
for all $(\lambda, x) \in W \times N$.
This shows that $f$ is continuous at $(\lambda_0, x_0)$.
\end{proof}

\subsection{Continuity of global semiflows in uniform spaces}

\begin{corollary}\label{cor:continuity of global semiflows}
Suppose that $X$ is a uniform space and $\varPhi$ is a global semiflow in $X$.
Then $\varPhi$ is continuous if and only if both of the following properties are satisfied:
\begin{enumerate}
\item[\emph{(i)}] $\varPhi(\cdot, x) \colon \mathbb{R}_+ \to X$ is continuous for every $x \in X$.
\item[\emph{(ii)}] For each $T > 0$, $(\varPhi(t, \cdot))_{t \in [0, T]}$ is equi-continuous at every $x \in X$.
\end{enumerate}
\end{corollary}

\begin{proof}
(Only-if-part).
The property (i) follows by the continuity of $\varPhi$.
Fix $T > 0$.
Applying Theorem~\ref{thm:equi-continuity from continuity},
$(\varPhi(t, \cdot))_{t \in [0, T]}$ is equi-continuous at every $x \in X$.

(If-part).
The property (ii) implies that $(\varPhi(t, \cdot))_{t \in \mathbb{R}_+}$ is locally equi-continuous at every $x \in X$.
Therefore, the conclusion holds from Theorem~\ref{thm:continuity from local equi-continuity}.
\end{proof}

\begin{definition}\label{dfn:locally equi-continuous C_0-semigroup}
Let $X$ be a linear topological space and $(T(t))_{t \in \mathbb{R}_+}$ be a family of linear operators on $X$.
We say that $(T(t))_{t \in \mathbb{R}_+}$ is a \textit{locally equi-continuous $C_0$-semigroup}
if the following properties are satisfied:
\begin{enumerate}
\item[(i)] $(T(t))_{t \in \mathbb{R}_+}$ is a linear semigroup.
\item[(ii)] For every $x \in X$, $\mathbb{R}_+ \ni t \mapsto T(t)x \in X$ is continuous.
\item[(iii)] For every $T > 0$, $(T(t))_{t \in [0, T]}$ is equi-continuous at $0 \in X$.
\end{enumerate}
\end{definition}

Compare the notion of locally equi-continuous $C_0$-semigroups
to the notion of \textit{equi-continuous $C_0$-semigroups}
introduced by Yosida~\cite{Yosida 1980}.

\begin{corollary}\label{cor:locally equi-continuous C_0-semigroup}
Let $X$ be a linear topological space and $(T(t))_{t \in \mathbb{R}_+}$ be a semigroup of linear operators on $X$.
Then the following properties are equivalent:
\begin{enumerate}
\item[\emph{(a)}] $(T(t))_{t \in \mathbb{R}_+}$ is a locally equi-continuous $C_0$-semigroup.
\item[\emph{(b)}] $\mathbb{R}_+ \times X \ni (t, x) \mapsto T(t)x \in X$ is continuous.
\end{enumerate}
\end{corollary}

\begin{proof}
For every $x \in X$,
	\begin{equation*}
		\{\mspace{2mu} x + N : \text{$N$ is a neighborhood of $0$ in $X$} \mspace{2mu}\}
	\end{equation*}
is a local base of $x$.
Therefore, the equi-continuity at $0$ implies the equi-continuity at every $x \in X$.
Applying Corollary~\ref{cor:continuity of global semiflows}, the equivalence is obtained.
\end{proof}

\subsection{Uniform contraction}

In this subsection, we suppose that $X = Y = (X, d)$ is a metric space.

\begin{definition}\label{dfn:uniform contraction}
$(f_\lambda)_{\lambda \in \Lambda}$ is said to be a \textit{uniform contraction} on $X$
if there exists $0 < c < 1$ such that for all $\lambda \in \Lambda$ and all $x, x' \in X$,
	\begin{equation*}
		d(f_\lambda(x), f_\lambda(x')) \le c \cdot d(x, x')
	\end{equation*}
holds.
\end{definition}

\begin{lemma}\label{lem:continuity of uniform contraction}
Suppose that $(f_\lambda)_{\lambda \in \Lambda}$ is a uniform contraction on a metric space $X = (X, d)$.
Then the following properties are equivalent:
\begin{enumerate}
\item[\emph{(a)}] $f$ is continuous.
\item[\emph{(b)}] $f(\cdot, x) \colon \Lambda \to X$ is continuous for every $x \in X$.
\end{enumerate}
\end{lemma}

\begin{proof}
(a) $\Rightarrow$ (b): Trivial.

(b) $\Rightarrow$ (a):
We show that $(f_\lambda)_{\lambda \in \Lambda}$ is equi-continuous at every $x \in X$.
Then (a) follows from Theorem~\ref{thm:continuity from local equi-continuity}.
Fix $x_0 \in X$.
Let $\ep > 0$.
For all $(\lambda, x) \in \Lambda \times X$, $d(x, x_0) < \ep/c$ implies
	\begin{equation*}
		d(f_\lambda(x), f_\lambda(x_0)) \le c \cdot d(x, x_0) < \ep.
	\end{equation*}
Since $\ep$ is arbitrary, this shows that $(f_\lambda)_{\lambda \in \Lambda}$ is equi-continuous at $x_0$.
\end{proof}

The following is a uniform contraction theorem.

\begin{theorem}\label{thm:uniform contraction thm}
Suppose that
(i) $(f_\lambda)_{\lambda \in \Lambda}$ is a uniform contraction on $X$, and
(ii) for each $\lambda \in \Lambda$, $x(\lambda) \in X$ is a fixed point of $f_\lambda$.
If $f(\cdot, x) \colon \Lambda \to X$ is continuous for every $x \in X$,
then $x(\cdot) \colon \Lambda \to X$ is continuous.
\end{theorem}

\begin{proof}
Fix $\lambda_0 \in \Lambda$.
We show that $x(\cdot) \colon \Lambda \to X$ is continuous at $\lambda_0 \in \Lambda$.
Let $\ep > 0$ be given.
By the triangle inequality, we have
	\begin{align*}
		d \bigl( x(\lambda), x(\lambda_0) \bigr)
		&= d \bigl( f_\lambda(x(\lambda)), f_{\lambda_0}(x(\lambda_0)) \bigr) \\
		&\le d \bigl( f_\lambda(x(\lambda)), f_\lambda(x(\lambda_0)) \bigr)
			+ d \bigl( f_\lambda(x(\lambda_0)), f_{\lambda_0}(x(\lambda_0)) \bigr) \\
		&\le c \cdot d(x(\lambda), x(\lambda_0)) + d \bigl( f_\lambda(x(\lambda_0)), f_{\lambda_0}(x(\lambda_0)) \bigr)
	\end{align*}
for all $\lambda \in \Lambda$.
Since $f(\cdot, x(\lambda_0)) \colon \Lambda \to X$ is continuous,
there is a neighborhood $W$ of $\lambda_0$ in $\Lambda$ such that for all $\lambda \in W$, we have
	\begin{equation*}
		d \bigl( f_\lambda(x(\lambda_0)), f_{\lambda_0}(x(\lambda_0)) \bigr) < (1 - c) \cdot \ep.
	\end{equation*}
This shows that
	\begin{equation*}
		d \bigl( x(\lambda), x(\lambda_0) \bigr)
		\le \frac{1}{1 - c} \cdot d \bigl( f_\lambda(x(\lambda_0)), f_{\lambda_0}(x(\lambda_0)) \bigr)
		< \ep
	\end{equation*}
holds for all $\lambda \in W$.
Therefore, the conclusion holds.
\end{proof}

\begin{remark}
In Theorem~\ref{thm:uniform contraction thm}, the completeness of $X$ is unnecessary
(see also \cite[Theorem 1.244]{Chicone 2006}).
Since $f_\lambda$ is a contraction, a fixed point of $f_\lambda$ is unique if it exists.
\end{remark}

\section{Lipschitz conditions about prolongations for ODEs}\label{sec:Lip about prolongations for ODEs}

In this appendix, we consider an ODE
	\begin{equation}\label{eq:ODE}
		\dot{x}(t) = f(t, x(t))
	\end{equation}
for a map $f \colon \mathbb{R} \times E \supset \dom(f) \to E$.
ODE~\eqref{eq:ODE} can be considered as an RFDE with history space $C(I^0, E)_\mathrm{u}$ by
	\begin{equation*}
		\dot{x}(t) = F_f \bigl( t, I^0_tx \bigr).
	\end{equation*}
Here $F_f \colon \mathbb{R} \times C(I^0, E)_\mathrm{u} \supset \dom(F_f) \to E$ is defined by
	\begin{equation*}
		\dom(F_f) = \{\mspace{2mu} (t, \phi) : (t, \phi(0)) \in \dom(f) \mspace{2mu}\},
			\mspace{15mu}
		F_f(t, \phi) = f(t, \phi(0)).
	\end{equation*}
Therefore, by using the isometry
	\begin{equation*}
		j \colon \mathbb{R} \times C(I^0, E)_\mathrm{u} \to \mathbb{R} \times E,
			\mspace{15mu}
		j(t, \phi) := (t, \phi(0)),
	\end{equation*}
$F_f$ is represented by
	\begin{equation*}
		\dom(F_f) = j^{-1}(\dom(f)), \mspace{15mu} F_f = f \circ j.
	\end{equation*}

\begin{proposition}\label{prop:r = 0, rectangles}
Let $I = I^0$.
Then for every $0 \le T < \infty$ and $0 \le \delta \le \infty$,
	\begin{equation*}
		\varLambda_{0, \boldsymbol{0}}(T, \delta) = [0, T] \times \bar{B}_E(0; \delta)
	\end{equation*}
holds under the identification $\Map(I^0, E) = E$.
\end{proposition}

\begin{proof}
$(\subset)$
Let $(\tau, \phi) \in \varLambda_{0, \boldsymbol{0}}(T, \delta)$.
Then $\tau \in [0, T]$ and $\phi = I^0_\tau\beta$ for some $\beta \in \varGamma_{0, \boldsymbol{0}}(\tau, \delta)$.
Therefore, we have
	\begin{equation*}
		\|\phi(0)\|_E = \|\beta(\tau)\|_E \le \|\beta\|_\infty \le \delta.
	\end{equation*}
This shows $\varLambda_{0, \boldsymbol{0}}(T, \delta) \subset [0, T] \times \bar{B}_E(0; \delta)$.

$(\supset)$
Let $(\tau, x_0) \in [0, T] \times \bar{B}_E(0; \delta)$.
We consider the map $\beta \colon [0, \tau] + I^0 \to E$ given by $\beta(t) := (t/\tau)x_0$.
Then for all $t \in [0, \tau]$, $\|\beta(t)\|_E \le \|x_0\|_E \le \delta$.
This means $\beta \in \varGamma_{0, \boldsymbol{0}}(\tau, \delta)$.
Therefore,
	\begin{equation*}
		(\tau, x_0)
		= (\tau, \beta(\tau))
		= (\tau, I^0_\tau\beta)
		\in \varLambda_{0, \boldsymbol{0}}(T, \delta).
	\end{equation*}

This completes the proof.
\end{proof}

For $L > 0$, the map $f$ is said to be $L$-\textit{Lipschitzian} if
	\begin{equation*}
		\|f(t, x_1) - f(t, x_2)\|_E \le L \cdot \|x_1 - x_2\|_E
	\end{equation*}
holds for all $(t, x_1), (t, x_2) \in \dom(f)$.
$f$ is said to be \textit{Lipschitzian} if $f$ is $L$-\textit{Lipschitzian} for some $L > 0$.
$f$ is said to be \textit{locally Lipschitzian}
if for every $(t_0, x_0) \in \dom(f)$, there exists a neighborhood $W$ of $(t_0, x_0)$ in $\dom(f)$ such that
$f|_{W}$ is Lipschitzian.

\begin{lemma}\label{lem:Lip about prolongations for ODEs}
Let $(t_0, x_0) \in \dom(f)$ and $(t_0, \phi_0) := j^{-1}(t_0, x_0)$.
Then the following properties are equivalent:
\begin{enumerate}
\item[\emph{(a)}] $f$ is $L$-Lipschitzian on $\bigl( [t_0, t_0 + T] \times \bar{B}_E(x_0; \delta) \bigr) \cap \dom(f)$.
\item[\emph{(b)}] $F_f$ satisfies $L$-\eqref{eq:Lipschitzian} 
for all $(t, \phi_1), (t, \phi_2) \in \varLambda_{t_0, \phi_0}(T, \delta) \cap \dom(F_f)$.
\end{enumerate}
\end{lemma}

\begin{proof}
(a) $\Rightarrow$ (b):
For all $(t, \phi_1), (t, \phi_2) \in \varLambda_{t_0, \phi_0}(T, \delta) \cap \dom(F_f)$,
	\begin{equation*}
		j(t, \phi_1), j(t, \phi_2) \in \bigl( [t_0, t_0 + T] \times \bar{B}_E(x_0; \delta) \bigr) \cap \dom(f)
	\end{equation*}
from Proposition~\ref{prop:r = 0, rectangles}.
Therefore, we have
	\begin{align*}
		\|F_f(t, \phi_1) - F_f(t, \phi_2)\|_E
		&= \|f(t, \phi_1(0)) - f(t, \phi_2(0))\|_E \\
		&\le L \cdot \|\phi_1(0) - \phi_2(0)\|_E \\
		&= L \cdot \|\phi_1 - \phi_2\|_\infty,
	\end{align*}
which shows (b).

(b) $\Rightarrow$ (a):
For all $(t, x_1), (t, x_2) \in \bigl( [t_0, t_0 + T] \times \bar{B}_E(x_0; \delta) \bigr) \cap \dom(f)$,
	\begin{equation*}
		j^{-1}(t, x_1), j^{-1}(t, x_2) \in \varLambda_{t_0, \phi_0}(T, \delta) \cap \dom(F_f)
	\end{equation*}
from Proposition~\ref{prop:r = 0, rectangles}.
Let $(t, \phi_i) := j^{-1}(t, x_i)$ ($i = 1, 2$).
Then,
	\begin{align*}
		\|f(t, x_1) - f(t, x_2)\|_E
		&= \|F_f \circ j^{-1}(t, x_1) - F_f \circ j^{-1}(t, x_2)\|_E \\
		&\le L \cdot \|\phi_1 - \phi_2\|_\infty \\
		&= L \cdot \|x_1 - x_2\|_E
	\end{align*}
holds.

This completes the proof.
\end{proof}

\begin{proposition}\label{prop:Lip about prolongations for ODEs}
If $f$ is locally Lipschitzian, then $F_f$ is uniformly locally Lipschitzian about prolongations.
\end{proposition}

\begin{proof}
Fix $(t_0, x_0) \in \dom(f)$.
We show that $F_f$ is uniformly locally Lipschitzian about prolongations at $j^{-1}(t_0, x_0)$.
By assumption, there are $T, \delta, L > 0$ such that $f$ is $L$-Lipschitzian on
	\begin{equation*}
		\bigl( [t_0 - T, t_0 + T] \times \bar{B}_E(x_0; \delta) \bigr) \cap \dom(f).
	\end{equation*}
By the triangle inequality, for all $(t, x) \in [t_0 - (T/2), t_0 + (T/2)] \times \bar{B}_E(x_0; \delta/2)$, we have
	\begin{equation*}
		[t, t + (T/2)] \times \bar{B}_E(x; \delta/2) \subset [t_0 - T, t_0 + T] \times \bar{B}_E(x_0; \delta).
	\end{equation*}
Let
	\begin{equation*}
		W_0 := [t_0 - (T/2), t_0 + (T/2)] \times \bar{B}_E(x_0; \delta/2),
	\end{equation*}
which is a neighborhood of $(t_0, x_0)$ in $\mathbb{R} \times E$.
From Proposition~\ref{prop:r = 0, rectangles} and Lemma~\ref{lem:Lip about prolongations for ODEs},
$F_f$ satisfies $L$-\eqref{eq:Lipschitzian} for all $(\sigma, \psi) \in j^{-1}(W_0)$
and all $(t, \phi_1), (t, \phi_2) \in \varLambda_{\sigma, \psi}(T/2, \delta/2) \cap \dom(F_f)$.
This shows the conclusion.
\end{proof}

Lemma~\ref{lem:Lip about prolongations for ODEs} and Proposition~\ref{prop:Lip about prolongations for ODEs} show that
the notion of (uniform) local Lipschitz for history functionals is an extension of the notion of Lipschitz condition for ODEs.

\section{Proofs}\label{sec:proofs}

\subsection{Propositions in Section~\ref{sec:existence and uniqueness}}\label{subsec:proofs in Section 5}

\begin{proof}[\textbf{Proof of Proposition~\ref{prop:continuity about prolongations}}]
By the prolongability of $H$,
	\begin{equation*}
		[0, T_0] \ni u \mapsto \tau_{t_0, \phi_0}^0(u, I_u\beta_0) \in \mathbb{R} \times H
	\end{equation*}
is continuous.
Therefore,
	\begin{equation*}
		K
		:= \left\{ \mspace{2mu} \tau_{t_0, \phi_0}^0(u, I_u\beta_0) : u \in [0, T_0] \mspace{2mu} \right\}
	\end{equation*}
is a compact set of $\mathbb{R} \times H$.

Let $\ep > 0$.
By the continuity of $F$, there are $a > 0$ and a neighborhood $N$ of $\boldsymbol{0}$ in $H$ such that
for all $(t_1, \phi_1) \in \dom(F)$ and all $(t_2, \phi_2) \in K$,
	\begin{equation*}
		|t_1 - t_2| < a \mspace{10mu} \text{and} \mspace{10mu} \phi_1 - \phi_2 \in N
			\imply
		\|F(t_1, \phi_1) - F(t_2, \phi_2)\|_E \le \ep.
	\end{equation*}
For all $\beta \in \varGamma_{0, \boldsymbol{0}}(T_0, \delta)$ and all $u \in [0, T_0]$,
	\begin{equation*}
		\tau_{t_0, \phi_0}^0(u, I_u\beta) - \tau_{t_0, \phi_0}^0(u, I_u\beta_0)
		= (0, I_u[\beta - \beta_0]).
	\end{equation*}
We choose $r > 0$ so that
	\begin{equation*}
		\varLambda_{0, \boldsymbol{0}}(T_0, r) \subset N
	\end{equation*}
since $H$ is regulated by prolongations.
Then $\rho^0(\beta, \beta_0) \le r$ implies that
	\begin{equation*}
		\bigl\| F_{t_0, \phi_0}^0(u, I_u\beta) - F_{t_0, \phi_0}^0(u, I_u\beta_0) \bigr\|_E
		\le \ep
	\end{equation*}
holds for all $u \in [0, T_0]$.
This shows the conclusion.
\end{proof}

\begin{proof}[\textbf{Proof of Proposition~\ref{prop:compactness of transformation on prolongation space}}]
By assumptions, we choose $T_0, \delta, M > 0$ so that
\begin{itemize}
\item $\varLambda_{t_0, \phi_0}(T_0, \delta) \subset \dom(F)$,
\item $\sup \{\mspace{2mu} \|F(t, \phi)\|_E : (t, \phi) \in \varLambda_{t_0, \phi_0}(T_0, \delta) \mspace{2mu}\} \le M$, and
\item for each $\beta_0 \in \varGamma_{0, \boldsymbol{0}}(T, \delta)$,
$\eqref{eq:PC}_{t_0, \phi_0}$ holds as $\rho^0(\beta, \beta_0) \to 0$ in $\varGamma_{0, \boldsymbol{0}}(T, \delta)$.
\end{itemize}
Let $0 < T \le \min\{T_0, \delta/M\}$.

\begin{flushleft}
\textbf{Step 1. Well-definedness}
\end{flushleft}

Let $\beta \in \varGamma_{0, \boldsymbol{0}}(T, \delta)$.
Then for all $s \in [0, T]$, we have
	\begin{equation*}
		\bigl\| \mathcal{S}_{t_0, \phi_0}^0\beta(s) \bigr\|_E
		\le \int_0^T \bigl\| F_{t_0, \phi_0}^0(u, I_u\beta) \bigr\|_E \mspace{2mu} \mathrm{d}u
		\le MT
	\end{equation*}
because
	\begin{equation*}
		\tau_{t_0, \phi_0}^0(u, I_u\beta)
		\in \varLambda_{t_0, \phi_0}(T, \delta)
		\subset \varLambda_{t_0, \phi_0}(T_0, \delta)
	\end{equation*}
(see Remark~\ref{rmk:comparison of rectangle by prolongations}).
This shows $\mathcal{S}_{t_0, \phi_0}^0\beta \in \varGamma_{0, \boldsymbol{0}}(T, \delta)$.

\begin{flushleft}
\textbf{Step 2. Compactness}
\end{flushleft}

The continuity of $\mathcal{S}_{t_0, \phi_0}^0$ with respect to $\rho^0$ follows from $\eqref{eq:PC}_{t_0, \phi_0}$
because
	\begin{equation*}
		\rho^0 \bigl( \mathcal{S}_{t_0, \phi_0}^0\beta, \mathcal{S}_{t_0, \phi_0}^0\beta_0 \bigr)
		\le \int_0^T
			\bigl\| F_{t_0, \phi_0}^0(u, I_u\beta) - F_{t_0, \phi_0}^0(u, I_u\beta_0) \bigr\|_E
		\mspace{2mu} \mathrm{d}u
	\end{equation*}
holds for every $\beta, \beta_0 \in \varGamma_{0, \boldsymbol{0}}(T, \delta)$.

We consider a closed linear subspace of the Banach space $C([0, T], E)_\mathrm{u}$ given by
	\begin{equation*}
		X := \{\mspace{2mu} \chi \in C([0, T], E) : \chi(0) = 0, \|\chi\|_\infty \le \delta \mspace{2mu}\}.
	\end{equation*}
Then the map $j \colon \varGamma_{0, \boldsymbol{0}}(T, \delta) \to X$ given by
	\begin{equation*}
		j(\beta) := \beta|_{[0, T]}
	\end{equation*}
is isometrically isomorphic.
Let $K$ be the image of $\mathcal{S}_{t_0, \phi_0}^0$,
i.e., $K = \mathcal{S}_{t_0, \phi_0}^0(\varGamma_{0, \boldsymbol{0}}(T, \delta))$.
For the compactness, it is sufficient to show that $j(K)$ is relatively compact.

Let $\beta \in \varGamma_{0, \boldsymbol{0}}(T, \delta)$.
Then for all $s_1, s_2 \in [0, T]$, we have
	\begin{equation*}
		\bigl\| \mathcal{S}_{t_0, \phi_0}^0\beta(s_1) - \mathcal{S}_{t_0, \phi_0}^0\beta(s_2) \bigr\|_E
		\le \left| \int_{s_1}^{s_2} \bigl\| F_{t_0, \phi_0}^0(u, I_u\beta) \bigr\|_E \mspace{2mu} \mathrm{d}u \right| \\
		\le M|s_1 - s_2|,
	\end{equation*}
which shows that $\mathcal{S}_{t_0, \phi_0}^0\beta$ is $M$-Lipschitz continuous.
Therefore, by the Ascoli--Arzel\`{a} theorem, $j(K)$ is relatively compact.

This completes the proof.
\end{proof}

\begin{proof}[\textbf{Proof of Proposition~\ref{prop:compactness of transformation on prolongation space by continuity}}]
The properties
\begin{itemize}
\item $F$ is continuous along prolongations,
\item $F$ is locally bounded about prolongations
\end{itemize}
follow by the assumptions (i) and (ii) (see also Lemmas~\ref{lem:(E1) implies (E0)} and \ref{lem:(E2) implies (E0)}).
The assumption (iii) implies that
there are $T_0, \delta > 0$ such that $\varLambda_{t_0, \phi_0}(T_0, \delta) \subset \dom(F)$.
Then for each fixed $\beta_0 \in \varGamma_{0, \boldsymbol{0}}(T_0, \delta)$, we have
	\begin{align*}
		&\int_0^{T_0}
			\bigl\| F_{t_0, \phi_0}^0(u, I_u\beta) - F_{t_0, \phi_0}^0(u, I_u\beta_0) \bigr\|_E
		\mspace{2mu} \mathrm{d}u \\
		&\mspace{20mu}
			\le T_0 \cdot \sup_{u \in [0, T_0]} \bigl\| F_{t_0, \phi_0}^0(u, I_u\beta) - F_{t_0, \phi_0}^0(u, I_u\beta_0) \bigr\|_E,
	\end{align*}
which converges to $0$ as $\beta \to \beta_0$ in $\varGamma_{0, \boldsymbol{0}}(T_0, \delta)$
from Proposition~\ref{prop:continuity about prolongations}.
Therefore, the conclusion is an application of Proposition~\ref{prop:compactness of transformation on prolongation space}.
\end{proof}

\begin{proof}[\textbf{Proof of Theorem~\ref{thm:Peano thm}}]
From Proposition~\ref{prop:compactness of transformation on prolongation space by continuity},
there exist $T, \delta > 0$ such that
	\begin{equation*}
		\mathcal{S}_{t_0, \phi_0}^0 \colon
		\varGamma_{0, \boldsymbol{0}}(T, \delta) \to \varGamma_{0, \boldsymbol{0}}(T, \delta)
	\end{equation*}
is a compact map.
Applying the Schauder fixed point theorem,
$\mathcal{S}_{t_0, \phi_0}^0$ has a fixed point in $\varGamma_{0, \boldsymbol{0}}(T, \delta)$
because $\varGamma_{0, \boldsymbol{0}}(T, \delta)$ is a closed convex subset of the Banach space
	\begin{equation*}
		\bigl( \varGamma_{0, \boldsymbol{0}}(T, \infty), \|\cdot\|_\infty \bigr).
	\end{equation*}
Therefore, $\mathcal{T}_{t_0, \phi_0}$ also has a fixed point in $\varGamma_{t_0, \phi_0}(T, \delta)$
by the diagram
	\begin{equation*}
		\xymatrix{
			\varGamma_{t_0, \phi_0}(T, \delta)
				\ar[r]^{\mathcal{T}_{t_0, \phi_0}}
				\ar[d]_{N_{t_0, \phi_0}^0}
				\ar@{}[dr]|\circlearrowleft
			& \varGamma_{t_0, \phi_0}(T, \delta)
				\ar[d]^{N_{t_0, \phi_0}^0} \\
			\varGamma_{0, \boldsymbol{0}}(T, \delta)
				\ar[r]_{\mathcal{S}_{t_0, \phi_0}^0}
			& \varGamma_{0, \boldsymbol{0}}(T, \delta).
		}
	\end{equation*}
That fixed point is a solution of $\eqref{eq:IVP}_{t_0, \phi_0}$.
\end{proof}

\subsection{Propositions and Theorem in Section~\ref{sec:mechanisms}}\label{subsec:proofs in Section 6}

\begin{proof}[\textbf{Proof of Proposition~\ref{prop:continuity about prolongations, uniform}}]
Fix $(\sigma_0, \psi_0) \in W$, $0 < T \le T_0$, and $\beta_0 \in \varGamma_{0, \boldsymbol{0}}(T, \delta)$.
By the prolongability of $H$,
	\begin{equation*}
		[0, T] \ni u \mapsto \tau_{\sigma_0, \psi_0}^0(u, I_u\beta_0) \in \mathbb{R} \times H
	\end{equation*}
is continuous.
Therefore,
	\begin{equation*}
		K
		:= \left\{ \mspace{2mu} \tau_{\sigma_0, \psi_0}^0(u, I_u\beta_0) : u \in [0, T] \mspace{2mu} \right\}
	\end{equation*}
is a compact set of $\mathbb{R} \times H$.

Let $\ep > 0$.
By the continuity of $F$,
there are $a > 0$ and a neighborhood $N$ of $\boldsymbol{0}$ in $H$ such that
for all $(t_1, \phi_1) \in \dom(F)$ and all $(t_2, \phi_2) \in K$,
	\begin{equation*}
		|t_1 - t_2| < a \mspace{10mu} \text{and} \mspace{10mu} \phi_1 - \phi_2 \in N
			\imply
		\|F(t_1, \phi_1) - F(t_2, \phi_2)\|_E \le \ep.
	\end{equation*}
We choose a neighborhood $N'$ of $\boldsymbol{0}$ in $H$ so that $N' + N' \subset N$.
We also choose $r > 0$ so that
	\begin{equation*}
		\varLambda_{0, \boldsymbol{0}}(T, r) \subset N'
	\end{equation*}
since $H$ is regulated by prolongations.
For all $(\sigma, \psi, \beta) \in \mathbb{R} \times H \times \varGamma_{0, \boldsymbol{0}}(T, \delta)$
and all $u \in [0, T]$,
	\begin{align*}
		\tau_{\sigma, \psi}^0(u, I_u\beta) - \tau_{\sigma_0, \psi_0}^0(u, I_u\beta_0)
		&= \Bigl( \sigma - \sigma_0, I_u[\beta - \beta_0] + I_u[\bar{\psi} - \bar{\psi}_0] \Bigr) \\
		&= \bigl( \sigma - \sigma_0, I_u[\beta - \beta_0] + S_0(u)(\psi - \psi_0) \bigr).
	\end{align*}
Therefore, there are a neighborhood $W'$ of $(\sigma_0, \psi_0)$ in $W$ such that
for all $(\sigma, \psi) \in W'$, all $\beta \in \varGamma_{0, \boldsymbol{0}}^1(T, \delta, 0)$
satisfying $\rho^1(\beta, \beta_0) \le r$,
and all $u \in [0, T]$,
	\begin{align*}
		\tau_{\sigma, \psi}^0(u, I_u\beta) - \tau_{\sigma_0, \psi_0}^0(u, I_u\beta_0)
		&\in (-a, a) \times (N' + N') \\
		&\subset (-a, a) \times N.
	\end{align*}
This shows the conclusion.
\end{proof}

\begin{proof}[\textbf{Proof of Proposition~\ref{prop:compactness of transformation on prolongation space, uniform}}]
By the continuity of $F$ at $(t_0, \phi_0)$,
there exists a neighborhood $W_0$ of $(t_0, \phi_0)$ in $\mathbb{R} \times H$ and $M > 0$ such that
	\begin{equation*}
		\sup_{(t, \phi) \in W_0 \cap \dom(F)} \|F(t, \phi)\|_E \le M.
	\end{equation*}
$W_0$ is a uniform neighborhood of $(t_0, \phi_0)$ by prolongations from Theorem~\ref{thm:nbd by prolongations}.
Therefore, $W_0 \cap \dom(F)$ is also a uniform neighborhood of $(t_0, \phi_0)$ by prolongations.
Then there are a neighborhood $W$ of $(t_0, \phi_0)$ in $\dom(F)$ and $T_0, \delta > 0$ such that
	\begin{equation*}
		\bigcup_{(\sigma, \psi) \in W} \varLambda_{\sigma, \psi}(T_0, \delta) \subset W_0 \cap \dom(F).
	\end{equation*}
We choose $0 < T < \min\{T_0, \delta/M\}$.

\begin{flushleft}
\textbf{Step 1. Well-definedness}
\end{flushleft}

Let $(\sigma, \psi, \beta) \in W \times \varGamma_{0, \boldsymbol{0}}(T, \delta)$.
Then for all $s \in [0, T]$, we have
	\begin{equation*}
		\bigl\| \mathcal{S}_{\sigma, \psi}^0\beta(s) \bigr\|_E
		\le \int_0^T \bigl\| F_{\sigma, \psi}^0(u, I_u\beta) \bigr\|_E \mspace{2mu} \mathrm{d}u
		\le MT.
	\end{equation*}
This shows $\mathcal{S}_{\sigma, \psi}^0\beta \in \varGamma_{0, \boldsymbol{0}}(T, \delta)$.

\begin{flushleft}
\textbf{Step 2. Compactness}
\end{flushleft}

The continuity at fixed $(\sigma_0, \psi_0, \beta_0) \in W \times \varGamma_{0, \boldsymbol{0}}(T, \delta)$ follows
because
	\begin{equation*}
		\rho^0 \bigl( \mathcal{S}_{\sigma, \psi}^0\beta, \mathcal{S}_{\sigma_0, \psi_0}^0\beta_0 \bigr)
		\le \int_0^T
			\bigl\| F_{\sigma, \psi}^0(u, I_u\beta) - F_{\sigma_0, \psi_0}^0(u, I_u\beta_0) \bigr\|_E
		\mspace{2mu} \mathrm{d}u,
	\end{equation*}
where the right-hand side converges to $0$ as $(\sigma, \psi, \beta) \to (\sigma_0, \psi_0, \beta_0)$
from Proposition~\ref{prop:continuity about prolongations, uniform}.
In the same way as the proof of Proposition~\ref{prop:compactness of transformation on prolongation space},
$\mathcal{S}_{\sigma, \psi}^0\beta$ is $M$-Lipschitz continuous
for every $(\sigma, \psi, \beta) \in W \times \varGamma_{0, \boldsymbol{0}}(T, \delta)$,
and therefore, the compactness of the image follows.
\end{proof}

For the proof of Theorem~\ref{thm:maximal WP without uniform Lip, extension},
let $\mathcal{\bar{T}}_{\sigma, \psi}$ and $\mathcal{\bar{S}}^0_{\sigma, \psi}$ be the transformations for $\bar{F}$,
namely, the transformations obtained by replacing $F$ as $\bar{F}$ in Notation~\ref{notation:transformations for integral eq}.

\begin{proof}[\textbf{Proof of Theorem~\ref{thm:maximal WP without uniform Lip, extension}}]
IVP~\eqref{eq:IVP} satisfies the local existence and local uniqueness for $C^1$-solutions
from Corollaries~\ref{cor:local existence with Lip, (E2)} and \ref{cor:local uniqueness (E2)}.
Then the following statements hold from Proposition~\ref{prop:maximal C^1-solutions and solution process}:
\begin{itemize}
\item For each $(t_0, \phi_0) \in \dom(F)$, $\eqref{eq:IVP}_{t_0, \phi_0}$ has the unique maximal $C^1$-solution
	\begin{equation*}
		x_F(\cdot; t_0, \phi_0) \colon [t_0, t_0 + T_F(t_0, \phi_0)) + I \to E,
	\end{equation*}
where $0 < T_F(t_0, \phi_0) \le \infty$.
\item The solution process $\mathcal{P}_F$ defined by \eqref{eq:sol process}
given in Subsection~\ref{subsec:maximal well-posedness} is a maximal process in $\dom(F)$.
\end{itemize}

We now show that $\mathcal{P}_F$ is a continuous maximal process in $\dom(F)$.
For this purpose, we use Corollary~\ref{cor:continuity of maximal processes}
and Theorems~\ref{thm:equi-continuity from continuity} and \ref{thm:continuity from local equi-continuity}.

\begin{flushleft}
\textbf{Step 1. Continuity of orbits}
\end{flushleft}

This follows by the $C^1$-prolongability of $H$ (see Remark~\ref{rmk:H and bar{H}}).

\begin{flushleft}
\textbf{Step 2. Lower semi-continuity of escape time function}
\end{flushleft}

Fix $(t_0, \phi_0) \in \dom(F)$.
Applying Proposition~\ref{prop:compactness of transformation on prolongation space, uniform},
we choose a neighborhood $\overline{W}$ of $(t_0, \phi_0)$ in $\dom \bigl( \bar{F} \bigr)$
and $T, \delta > 0$ so that
$\bigl( \mathcal{\bar{S}}^0_{\sigma, \psi} \bigr)_{(\sigma, \psi) \in W}$ is an well-defined compact transformation
on $\varGamma_{0, \boldsymbol{0}}(T, \delta)$.
Therefore, $\mathcal{\bar{S}}^0_{\sigma, \psi}$ has a unique fixed point
	\begin{equation*}
		\eta(\cdot; \sigma, \psi) \in \varGamma_{0, \boldsymbol{0}}(T, \delta)
	\end{equation*}
for each $(\sigma, \psi) \in \overline{W}$.
Let $(\sigma, \psi) \in \overline{W}$.
Then
	\begin{equation*}
		\chi(\cdot; \sigma, \psi)
		:= A_{\sigma, \psi}^0[\eta(\cdot; \sigma, \psi)]
		\in \varGamma_{\sigma, \psi}(T, \delta)
	\end{equation*}
is a fixed point of
	$\mathcal{\bar{T}}_{\sigma, \psi} \colon
	\varGamma_{\sigma, \psi}(T, \delta) \to \varGamma_{\sigma, \psi}(T, \delta)$,
i.e., a solution of $\eqref{eq:extended IVP}_{\sigma, \psi}$.
Let $W := \overline{W} \cap \dom(F)$.
Then $W$ is a neighborhood of $(t_0, \phi_0)$ in $\dom(F)$.
From Lemma~\ref{lem:sol of extended IVP},
$\chi(\cdot; \sigma, \psi)$ is a $C^1$-solution of $\eqref{eq:IVP}_{\sigma, \psi}$ for each $(\sigma, \psi) \in W$.
By the maximality of $x_F(\cdot; \sigma, \psi)$, we have $T < T_F(\sigma, \psi)$.
Since this holds for every $(\sigma, \psi) \in W$,
	\begin{equation*}
		[0, T] \times W \subset \dom(\mathcal{P}_F)
	\end{equation*}
is derived.

\begin{flushleft}
\textbf{Step 3. Equi-continuity}
\end{flushleft}

Fix $(t_0, \phi_0) \in \dom(F)$.
We choose $W, T, \delta$ in Step 2.
For every $(\tau, \sigma, \psi) \in [0, T] \times W$, we have
	\begin{align*}
		\mathcal{P}_F(\tau, \sigma, \psi)
		&= I_{\sigma + \tau}[x_F(\cdot; \sigma, \psi)] \\
		&= I_{\sigma + \tau}[\chi(\cdot; \sigma, \psi)] \\
		&= I_\tau[\eta(\cdot; \sigma, \psi)] + I_\tau\bar{\psi}.
	\end{align*}
This shows that for each fixed $(\sigma_0, \psi_0) \in W$, we have
	\begin{align*}
		&\mathcal{P}_F(\tau, \sigma, \psi) - \mathcal{P}_F(\tau, \sigma_0, \psi_0) \\
		&= I_\tau[\eta(\cdot; \sigma, \psi) - \eta(\cdot; \sigma_0, \psi_0)]
			+ S_0(\tau)(\psi - \psi_0).
	\end{align*}
Then the equi-continuity of $(\mathcal{P}_F(\tau, \cdot)|_W)_{\tau \in [0, T]}$ follows by the following reasons:
\begin{itemize}
\item Applying Theorem~\ref{thm:continuity and uniqueness},
we have the convergence
	\begin{equation*}
		\rho^0 \bigl( \eta(\cdot; \sigma, \psi), \eta(\cdot; \sigma_0, \psi_0) \bigr) \to 0
		\mspace{20mu}
		\text{as $(\sigma, \psi) \to (\sigma_0, \psi_0)$ in $W$}.
	\end{equation*}
This implies
	\begin{equation*}
		I_\tau[\eta(\cdot; \sigma, \psi) - \eta(\cdot; \sigma_0, \psi_0)] \to 0
	\end{equation*}
uniformly in $\tau \in [0, T]$ because $\bar{H}$ is regulated by prolongations.
\item The uniform convergence of the second term follows by the continuity of
$\mathbb{R}_+ \times H \ni (t, \phi) \mapsto S_0(t)\phi \in H$ (see Remark~\ref{rmk:H and bar{H}}).
\end{itemize}
This completes the proof.
\end{proof}

\section*{Acknowledgements}
I would like to thank Professor Hans-Otto Walther (Justus Liebig University Giessen)
and Professor Hiroshi Kokubu (Kyoto University)
for the useful comments and discussions.
This work was partially supported by Research Alliance Center for Mathematical Sciences, Tohoku University
for Start-up expense.


\begin{thebibliography}{99}
\addcontentsline{toc}{section}{References}

\bibitem{Chicone 2006} %(MR2224508)
	\newblock C. Chicone,
	\newblock ``Ordinary Differential Equations with Applications,''
	\newblock Second edition. Springer, New York, 2006.

\bibitem{Dafermos 1971} %(MR0291596)
	\newblock C. M. Dafermos,
	\newblock \emph{An invariance principle for compact processes},
	\newblock J. Differential Equations \textbf{9} (1971), 239--252.

\bibitem{Diekmann--vanGils--Lunel--Walther 1995} %(MR1345150)
	\newblock O. Diekmann, S. A. van Gils, S. M. Verduyn Lunel and H.-O. Walther,
	\newblock ``Delay Equations. Functional, Complex, and Nonlinear Analysis,"
	\newblock Springer-Verlag, New York, 1995.

\bibitem{Erneux 2009} %(MR2498700)
	\newblock T. Erneux,
	\newblock ``Applied Delay Differential Equations,"
	\newblock Springer, New York, 2009.

\bibitem{Gordon 1991} %(MR1138145)
	\newblock R. Gordon,
	\newblock \textit{Riemann integration in Banach spaces}, 
	\newblock Rocky Mountain J. Math. \textbf{21} (1991), 923--949. 

\bibitem{Granas--Dugundji} %(MR1987179)
	\newblock A. Granas and J. Dugundji,
	\newblock ``Fixed Point Theory,''
	\newblock Springer-Verlag, New York, 2003.

\bibitem{Hajek 1968} %(MR0239575)
	\newblock O. H\'{a}jek,
	\newblock \textit{Local characterisation of local semi-dynamical systems},
	\newblock Math. Systems Theory \textbf{2} (1968), 17--25.

\bibitem{Hale 1963b} %(MR0157064)
	\newblock J. K. Hale,
	\newblock \textit{A stability theorem for functional-differential equations},
	\newblock Proc. Nat. Acad. Sci. U.S.A. \textbf{50} (1963), 942--946.

\bibitem{Hale 1965} %(MR0183938)
	\newblock J. K. Hale,
	\newblock \textit{Sufficient conditions for stability and instability of autonomous functional-differential equations},
	\newblock J. Differential Equations \textbf{1} (1965), 452--482. 

\bibitem{Hale 1969} %(MR0244582)
	\newblock J. K. Hale,
	\newblock \textit{Dynamical systems and stability},
	\newblock J. Math. Anal. Appl. \textbf{26} (1969), 39--59.

\bibitem{Hale 1977} %(MR0508721)
	\newblock J. K. Hale,
	\newblock ``Theory of Functional Differential Equations,"
	\newblock Springer-Verlag, New York, 1977.

\bibitem{Hale 1988} %(MR0941371)
	\newblock J. K. Hale,
	\newblock ``Asymptotic Behavior of Dissipative Systems,''
	\newblock American Mathematical Society, Providence, RI, 1988.

\bibitem{Hale 2006b} %(MR2337813)
	\newblock J. K. Hale,
	\newblock \textit{History of delay equations},
	\newblock Delay differential equations and applications, 1--28, 
	NATO Sci. Ser. II Math. Phys. Chem., 205, Springer, Dordrecht, 2006. 

\bibitem{Hale--J.Kato 1978} %(MR0492721)
	\newblock J. K. Hale and J. Kato,
	\newblock \textit{Phase space for retarded equations with infinite delay},
	\newblock Funkcial. Ekvac. \textbf{21} (1978), 11--41.

\bibitem{Hale--Lunel 1993} %(MR1243878)
	\newblock J. K. Hale and S. M. Verduyn Lunel,
	\newblock ``Introduction to Functional Differential Equations,''
	\newblock Springer-Verlag, New York, 1993.

\bibitem{Hartung--Krisztin--Walther--Wu 2006} %(MR2457636)
	\newblock F. Hartung, T. Krisztin, H.-O. Walther and J. Wu,
	\newblock \textit{Functional differential equations with state-dependent delays: theory and applications},
	\newblock Handbook of differential equations: ordinary differential equations. Vol. III, 435--545, Elsevier/North-Holland, Amsterdam, 2006.

\bibitem{Hino--Murakami--Naito 1991} %(MR1122588)
	\newblock Y. Hino, S. Murakami, and T. Naito,
	\newblock ``Functional-differential Equations with Infinite Delay," 
	\newblock Springer-Verlag, Berlin, 1991.

\bibitem{Kappel--Schappacher 1978} %(MR0512478)
	\newblock F. Kappel and W. Schappacher,
	\newblock \textit{Autonomous nonlinear functional differential equations and averaging approximations},
	\newblock Nonlinear Anal. \textbf{2} (1978), 391--422.

\bibitem{Kappel--Schappacher 1980} %(MR0587220)
	\newblock F. Kappel and W. Schappacher,
	\newblock \textit{Some considerations to the fundamental theory of infinite delay equations},
	\newblock J. Differential Equations. \textbf{37} (1980), 141--183.

\bibitem{Kato 1978} %(MR0492740)
	\newblock J. Kato,
	\newblock \textit{Stability problem in functional differential equations with infinite delay},
	\newblock Funkcial. Ekvac. \textbf{21} (1978), 63--80.

\bibitem{Kelley 1955} %(MR0070144)
	\newblock J. L. Kelley,
	\newblock ``General Topology,"
	\newblock D. Van Nostrand Company, Inc., Toronto-New York-London, 1955.

\bibitem{Kolmanovskii--Nosov 1986} %(MR0860947)
	\newblock V. B. Kolmanovski\u{i} and V. R. Nosov,
	\newblock ``Stability of Functional Differential Equations,''
	\newblock Academic Press, London, 1986.

\bibitem{Lakshmanan--Senthilkumar 2010} %(MR2757455)
	\newblock M. Lakshmanan and D. V. Senthilkumar,
	\newblock ``Dynamics of Nonlinear Time-delay Systems,'' 
	\newblock Springer, Heidelberg, 2010.

\bibitem{Louihi--Hbid--Arino 2002} %(MR1900458)
	\newblock M. Louihi, M. L. Hbid and O. Arino,
	\newblock \textit{Semigroup properties and the Crandall Liggett approximation for a class of differential equations with state-dependent delays},
	\newblock J. Differential Equations. \textbf{181} (2002), 1--30.

\bibitem{Marsden--McCracken 1976} %(MR0494309)
	\newblock J. E. Marsden and M. McCracken,
	\newblock ``The Hopf Bifurcation and Its Applications,''
	With contributions by P. Chernoff, G. Childs, S. Chow, J. R. Dorroh, J. Guckenheimer, L. Howard, N. Kopell, O. Lanford, J. Mallet-Paret, G. Oster, O. Ruiz, S. Schecter, D. Schmidt and S. Smale.
	\newblock Springer-Verlag, New York, 1976. 

\bibitem{Mallet-Paret--Nussbaum--Paraskevopoulos 1994} %(MR1272890)
	\newblock J. Mallet-Paret, R. D. Nussbaum, and P. Paraskevopoulos,
	\newblock \textit{Periodic solutions for functional-differential equations with multiple state-dependent time lags},
	\newblock Topol. Methods Nonlinear Anal. \textbf{3} (1994), 101--162.

\bibitem{Nishiguchi 2017}
	\newblock J. Nishiguchi,
	\newblock \textit{A necessary and sufficient condition for well-posedness of initial value problems of retarded functional differential equations},
	\newblock J. Differential Equations \textbf{263} (2017), 3491--3532.

\bibitem{Ockendon--Tayler 1971}
	\newblock J. R. Ockendon and A. B. Tayler,
	\newblock \textit{The dynamics of a current collection system for an electric locomotive},
	\newblock Proc. R. Soc. Lond. A \textbf{322}, (1971), 447--468.

\bibitem{Rezounenko 2009} %(MR2515314)
	\newblock A. V. Rezounenko,
	\newblock \textit{Differential equations with discrete state-dependent delay: uniqueness and well-posedness in the space of continuous functions},
	\newblock Nonlinear Anal. \textbf{70} (2009), 3978--3986.

\bibitem{Rezounenko 2012} %(MR2834276)
	\newblock A. V. Rezounenko,
	\newblock \textit{A condition on delay for differential equations with discrete state-dependent delay},
	\newblock J. Math. Anal. Appl. \textbf{385} (2012), 506--516.

 \bibitem{Schumacher 1978} %(MR0477379)
	\newblock K. Schumacher,
	\newblock \textit{Existence and continuous dependence for functional-differential equations with unbounded delay},
	\newblock Arch.\ Rational Mech.\ Anal.\ \textbf{67} (1978), 315--335.

\bibitem{Smith 2011} %(MR2724792)
	\newblock H. Smith,
	\newblock ``An introduction to Delay Differential Equations with Applications to the Life Sciences,"
	\newblock Springer, New York, 2011.

\bibitem{Walther 2003c} %(MR2019242)
	\newblock H.-O. Walther,
	\newblock \textit{The solution manifold and $C^1$-smoothness for differential equations with state-dependent delay},
	\newblock J. Differential Equations \textbf{195} (2003), 46--65.

\bibitem{Walther 2014a} %(MR3210290)
	\newblock H.-O. Walther,
	\newblock \textit{Topics in delay differential equations},
	\newblock Jahresber. Dtsch. Math.-Ver. \textbf{116} (2014), 87--114.

\bibitem{Walther 2016}
	\newblock H.-O. Walther,
	\newblock \textit{Semiflows for differential equations with locally bounded delay on solution manifolds
	in the space $C^1((-\infty, 0], \mathbb{R}^n)$},
	\newblock Topol. Methods Nonlinear Anal. \textbf{48} (2016), 507--537.

\bibitem{Yosida 1980} %(MR1336382)
	\newblock K. Yosida,
	\newblock ``Functional Analysis,"
	\newblock Springer-Verlag, Berlin, 1980.
\end{thebibliography}
\end{document}